\documentclass{amsart}
\usepackage{enumerate}
\usepackage{amssymb}
\usepackage{marvosym}
\usepackage{tabularx}
\usepackage{caption}
\usepackage{subcaption}
\usepackage{mathrsfs}
\usepackage{mathtools}
\mathtoolsset{showonlyrefs}
\usepackage[normalem]{ulem}

\newtheorem{thm}{Theorem}[section]
\newtheorem{cor}[thm]{Corollary}
\newtheorem{lem}[thm]{Lemma}
\newtheorem{prop}[thm]{Proposition}

\theoremstyle{definition}
\newtheorem{defn}[thm]{Definition}

\newtheorem{problem}[thm]{Problem}

\newtheorem{example}[thm]{Example}
\theoremstyle{remark}
\newtheorem{rem}[thm]{Remark}

\numberwithin{equation}{subsection}
\numberwithin{figure}{section}

\newcommand{\C}{{\mathbb C}}
\newcommand{\D}{{\mathbb D}}
\newcommand{\T}{{\mathbb T}}
\newcommand{\R}{{\mathbb R}}
\newcommand{\s}{{\mathbb S}}

\newcommand{\N}{{\mathbb N}}

\newcommand{\G}{\mathcal{G}}
\newcommand{\A}{\mathbb{A}}
\newcommand{\HS}{\mathbb{H}}

\newcommand{\calS}{\mathcal{S}}

\newcommand{\calH}{{\mathcal H}}
\newcommand{\calD}{{\mathcal D}}
\newcommand{\calP}{{\mathcal P}}
\newcommand{\calT}{{\mathcal T}}
\newcommand{\calI}{\mathcal{I}}

\newcommand{\calA}{{\mathcal A}}

\newcommand{\calN}{{\mathcal N}}
\newcommand{\calF}{{\mathcal F}}

\newcommand{\calK}{{\mathcal K}}

\newcommand{\calG}{{\mathcal G}}

\newcommand{\calM}{{\mathcal M}}
\newcommand{\calE}{{\mathcal E}}
\newcommand{\calZ}{{\mathcal Z}}

\newcommand{\dir}{\mathscr{D}}
\newcommand{\diffH}{\mathrm{d}_{\mathrm{H}}}
\newcommand{\diff}{{\mathrm d}}
\newcommand{\diffs}{\mathrm{ds}}
\newcommand{\diffA}{\mathrm{dA}}

\newcommand{\imag}{{\mathrm i}}
\newcommand{\Ordo}{\mathrm{O}}
\newcommand{\ordo}{\mathrm{o}}
\newcommand{\e}{\mathrm e}
\renewcommand{\Re}{\operatorname{Re}}
\renewcommand{\Im}{\operatorname{Im}}

\renewcommand{\s}{\mathrm{s}}
\renewcommand{\c}{\mathrm{c}}

\makeatletter
\let\@wraptoccontribs\wraptoccontribs
\makeatother

\begin{document}
\title[The forbidden region for random zeros]
{The forbidden region for random zeros: appearance of quadrature domains}

\author{Alon Nishry}
\author{Aron Wennman}


\begin{abstract}
Our main discovery is a surprising interplay between 
quadrature domains on the one hand, and the
zero process of the Gaussian Entire Function (GEF) on the other. 
Specifically, consider the GEF conditioned on the rare \emph{hole event} that 
there are no zeros in a given large Jordan domain.  
We show that in the natural scaling limit, a quadrature domain 
enclosing the hole emerges as a 
\emph{forbidden region}, where the zero density vanishes.
Moreover, we give a description of those holes for which 
the forbidden region is a disk.

The connecting link between random zeros and potential theory 
is supplied by a constrained extremal problem for 
the Zeitouni-Zelditch functional. To solve this problem, we recast it
in terms of a seemingly novel obstacle problem,
where the solution is forced to be harmonic inside the hole.
\end{abstract}
\maketitle
\tableofcontents

\section{Introduction and main results}
\label{s:intro}

\subsection{Random zeros and forbidden regions}
\label{ss:problem-formul}
We are interested in the asymptotic conditional intensity of a stationary point process in the
plane, conditional on the rare \emph{hole event} that there are no points in a large region $\G$.
To give some context, we mention briefly
two examples where the behavior on the hole 
event is well-understood.
For the homogeneous Poisson process, the spatial independence property
shows that the effect of the hole is not felt outside $\G$.
For the Ginibre ensemble, a planar Coulomb gas at critical temperature, 
the situation is more interesting: 
the particles accumulate 
near the boundary of $\G$ in such a way that
the electrostatic potential generated by the points is asymptotically unchanged
outside the hole (a procedure known as \emph{balayage}).  
Hence, there are no macroscopic effects outside $\overline{\G}$, 
see Figure \ref{fig:square-hole}(A). For details we refer to
the works by Adhikari and Reddy \cite{Adhikari}, Armstrong, Serfaty and Zeitouni \cite{ASZ}, 
and Jancovici, Lebowitz and Manificat \cite{JLM}.

Our focus is on the zero process of the
Gaussian Entire Function (GEF), introduced 
by Bogomolny, Bohigas and Labeuf \cite{BBL},
Kostlan \cite{Kostlan}, and Shub and Smale \cite{ShubSmale} 
in the 1990s. For convenience, we consider a fixed domain $\G$
and a zero process with increasing intensity.
The GEF is given
by the random Taylor series
\begin{equation}\label{eq:def-GEF}
F_L(z)=\sum_{n=0}^\infty \xi_n\frac{(Lz)^n}{\sqrt{n!}},\qquad z\in\C,
\end{equation}
where $\xi_n\sim\calN_\C(0,1)$ are independent standard 
complex Gaussian random variables, and
its zero set forms an invariant point process with intensity $\pi^{-1}L^2$
with respect to area measure.
In fact, the GEF is the only Gaussian analytic
function with this property, see \cite[\S~2.5]{HKPV}.
Since its introduction, the GEF has been widely studied, with
contributions including \cite{ForresterHonner, GP, Hannay, NazarovSodinNormality, 
NSV0, NSV, SodinTsirelson1} to mention a few. 
For background on the GEF zeros and related models
we refer to the monograph \cite{HKPV} by Hough, Krishnapur, Virag and Peres, 
and the ICM notes \cite{NSICM} by Sodin and Nazarov.

For a bounded plane region $\G$, we denote the 
{\em hole event for random zeros} by
\begin{equation}
\calH_L(\G)=\big\{F_L\text{ is zero-free in the region }\G\big\}.
\end{equation}
We let $\mu_{L}^\C$ be the random \emph{empirical measure} obtained by 
placing a unit point charge at each zero of $F_L$, and denote by
$\mu_{L,\G}^\C$ the measure $\mu_L^\C$ conditioned on $\calH_L(\G)$.

We will show in the course of the proof of 
Theorem~\ref{thm:main-GEF} (c.f. Proposition~\ref{prop:struct}) 
that for a rather general hole $\G$,
the rescaled measure $L^{-2}\mu_{L,\G}^\C$ converges 
to a limiting measure $\mu_\G^\C$, 
which splits into
a singular part supported on $\partial\G$, and a continuous
part supported on the complement of a larger region $\Omega=\Omega(\G)$ 
containing $\G$.
There is always a macroscopic gap $\Omega\setminus\overline{\G}$ 
between the two components
of the limiting measure, where the limiting zero density vanishes.  
We refer to $\Omega$ as the 
\emph{forbidden region}\footnote{This terminology appears in other contexts,
e.g.\ in quantum mechanics and semiclassical analysis \cite{Forbidden} 
and for random polynomials and partial Bergman kernels \cite{ShiffmanZelditch, Partial}. 
These notions are related with ours in that a limiting 
density of states vanishes, 
but for entirely different reasons.}. 
This phenomenon was suspected to occur for circular holes
by Nazarov and Sodin, and it was recently proved by Ghosh and the first-named
author in \cite{GhoshNishry1}.
We find the following problems natural.

\begin{problem}\label{prob:shape}
\emph{Determine the possible shapes of forbidden regions.}
\end{problem}

\begin{problem}\label{prob:inverse}
\emph{Given a forbidden region, determine which holes give rise to it.}
\end{problem}

Regarding the first problem, we will show that if 
the hole is a smooth Jordan domain, the 
forbidden region takes the shape of a quadrature domain. Under mild conditions on $\G$, 
we solve the inverse problem
(Problem~\ref{prob:inverse}) when the forbidden region is a disk.
We start with the inverse problem.

\subsection{The inverse problem and disk-like domains}\label{ss:disk-like}
The connection between the hole $\G$ and the associated
forbidden region $\Omega(\G)$ appears through a delicate variational 
problem (see \S~\ref{ss:extremal-intro} below), 
and one might initially suspect that little
could be established, besides regularity properties
of the measure $\mu_\G^\C$ and (free) boundary $\partial\Omega$. 
However, in one important special case we have a 
complete solution to the inverse problem.

\begin{figure}[t!]
	\vspace{12pt}
	\begin{subfigure}[t]{.42\textwidth}
		\centering
		\includegraphics[width=1\linewidth]{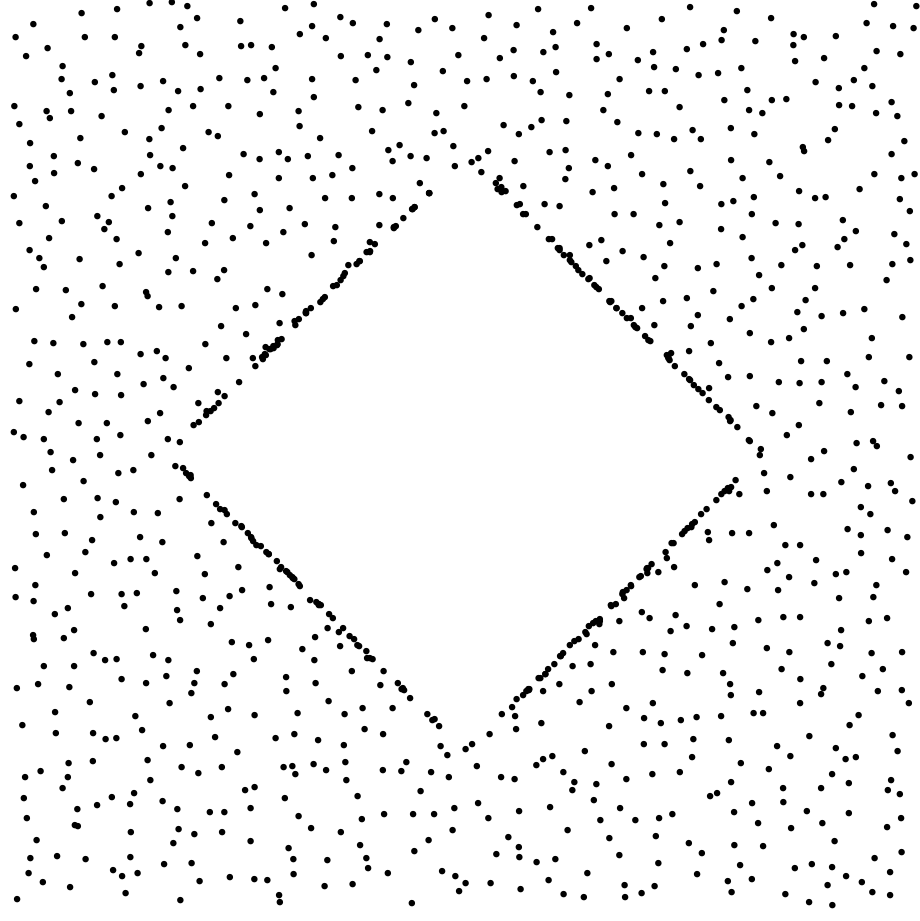}
		\subcaption{Ginibre ensemble}
		\label{fig:Ginibre_hole}
	\end{subfigure}
	\hspace{25pt}
	\begin{subfigure}[t]{.42\textwidth}
		\includegraphics[width=\linewidth]{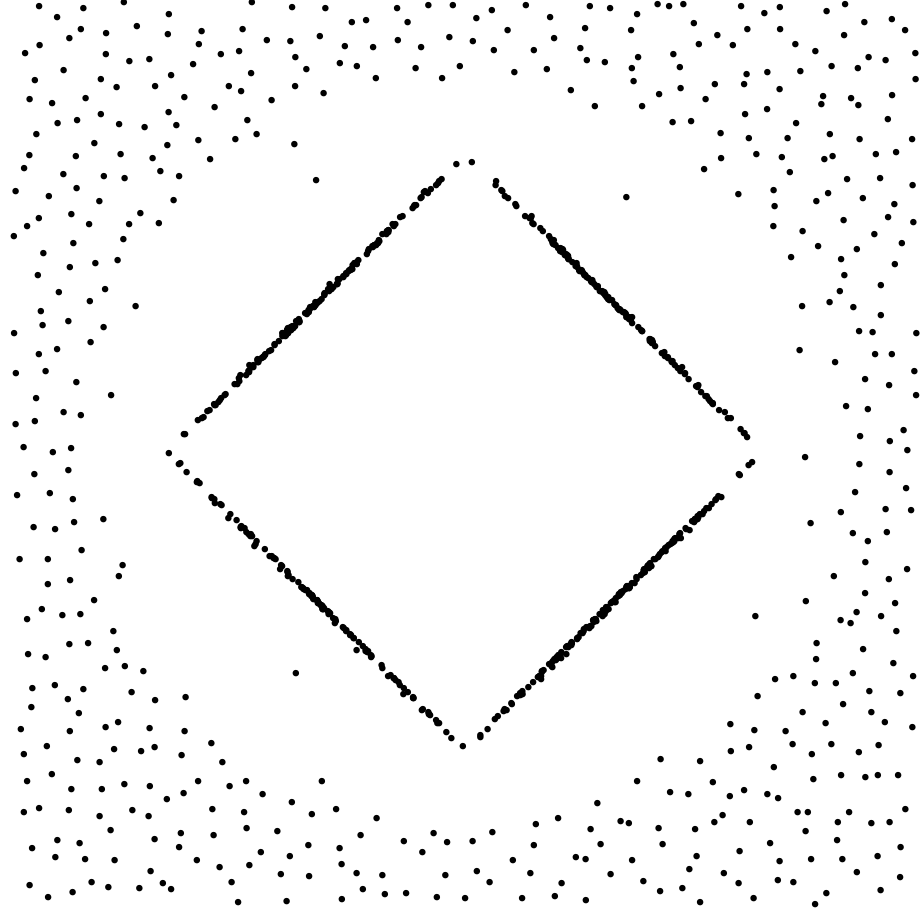}
		\subcaption{GEF zeros}
	\end{subfigure}
	\caption{Illustration of the square hole event.} 
	\label{fig:square-hole}
\end{figure}

\begin{defn}\label{def:almost-circ}
A Jordan domain $\G$ is said to be \emph{disk-like}
with center $z_0$ and radius $r$, if the Riemann mapping
$\varphi\colon \G\to\D$, which maps $z_0$ to the origin with 
$\varphi'(z_0)=:r^{-1}$, satisfies the bound
\begin{equation}\label{eq:almost-circ}
|\varphi(z)|\ge \tfrac{|z-z_0|}{r}
\exp\big(-\tfrac{|z-z_0|^2}{2\e r^2}\big),\qquad z\in\G.
\end{equation}
\end{defn}
One can verify\footnote{For a graphical illustration for regular $n$-gons with 
$n\ge 4$, see \cite{mathematica-Neumann}. The equilateral triangle is not disk-like.} 
that the square, the regular pentagon, 
ellipses up to a critical eccentricity as well as 
a wide class of more general perturbations of the disk are all disk-like.
That $\G$ is disk-like implies that $z_0$ is a local conformal center with (inner) conformal 
radius $r$ (see \cite[\S~6.3]{PolyaSzego}
and \S~\ref{s:a-circ} below). 
In addition, disk-likeness implies that
$\calG\subset\D(z_0,\sqrt{\e}\,r)$.

\begin{thm}\label{thm:a-circ}
Let $\G$ be a Jordan domain with piecewise smooth boundary 
without cusps. 
Then the forbidden region $\Omega$ is
the disk $\D(z_0,\sqrt{\e}r)$ if and only if $\G$ is disk-like with center $z_0$
and radius $r$.
\end{thm}
It follows from the proof (see \S~\ref{s:a-circ}) that the 
leading order asymptotics for the probability of the hole event $\calH_L(\G)$
for disk-like $\G$ depends only on the radius $r$.
It is plausible that this result should hold even under less restrictive regularity 
conditions on the boundary $\partial\G$. 
For an illustration see Figure \ref{fig:square-hole}(B).

\subsection{Appearance of quadrature domains}
A priori, it is not clear if there are any restrictions
on the shape of the forbidden region $\Omega$. 
However, the stability of the circular case
suggests a strong rigidity, and this led us to consider the following notion,
which is classical in potential theory.
We let $\diffA$ denote area measure.

\begin{defn}\label{defn:q-dom}
A domain $\Omega=\Omega_\nu$ is said to be a (subharmonic) 
quadrature domain with respect to
a finite measure $\nu$,
if $\Omega$ contains $\mathrm{supp}(\nu)$ and if for all integrable
subharmonic functions $u$ on $\Omega$ it holds that
\begin{equation}\label{eq:q-dom-ineq}
\int u(z)\diff \nu(z)\le \frac{1}{\pi}\int_{\Omega}u(z)\diffA(z).
\end{equation} 
\end{defn}
For background on quadrature domains, we refer to the surveys 
\cite{GustafssonPutinar, GustafssonShapiro} by Gustafsson and Putinar, and Gustafsson and 
Shapiro, respectively.
The most notable example of a quadrature domain is a disk. 
More generally, \emph{classical} quadrature domains correspond to finitely
supported positive measures $\nu$, and these may be thought of as 
potential theoretic \emph{sums} of disks, see \cite{GustafssonPutinar}.
Classical quadrature domains
are rather outstanding domains with algebraic boundaries 
\cite[\S~5.1]{Shapiro},
which appear in several areas (see \S~\ref{ss:related} below). 
Going back to Problem~\ref{prob:shape}, 
they appear as forbidden regions, seemingly out of nowhere.

\begin{thm}\label{thm:main-GEF}
Let $\G$ be a Jordan domain with $C^2$-smooth boundary.
The rescaled conditional empirical measures
$L^{-2}\mu_{L,\G}^\C$ converge vaguely in distribution as $L\to\infty$ to the measure
\[
\diff\mu_\G^\C=\sum_{\lambda\in\Lambda}
\rho_\lambda\diff\omega_{\G,\lambda} 
+ \tfrac{1}{\pi}\chi_{\C\setminus\Omega_\nu}\diffA,
\]
where $\Lambda\subset\G$ is a finite set, $\rho_\lambda$ are strictly positive weights,
$\omega_{\G,\lambda}$ are the harmonic 
measures supported on $\partial\G$ relative to the
points $\lambda$, and $\Omega_\nu$
is the (unique) quadrature domain with respect to
$\nu=\sum_{\lambda}\rho_\lambda\delta_\lambda$.
\end{thm}

The assumption that $\partial\G$ is $C^2$-smooth is made to fix ideas. 
A stronger version (Theorem~\ref{thm:main-GEF2}) is proven in \S~\ref{s:Weyl} below.

The notion of vague convergence in distribution is described in \cite[Ch.\ 4]{Kallenberg}.
In our case, it amounts to proving that for any continuous 
compactly supported test-function $f$, the random variables 
$L^{-2}\int f\,\diff\mu_{L,\G}^\C$ converges in probability to
$\int f\,\diff\mu_\G^\C$.

\begin{rem}
Let $\nu=\sum_{\lambda}\rho_\lambda\delta_\lambda$ be a finitely supported 
positive measure, and denote by $\Omega$ the associated
subharmonic quadrature domain. Then, for any bounded harmonic function $h$ on $\Omega$,
we have the quadrature rule
\[
\sum_{\lambda}\rho_\lambda h(\lambda)=\frac{1}{\pi}\int_{\Omega}h(z)\diffA(z).
\] 
Quadrature domains for \emph{harmonic functions} are those domains $\Omega$
who satisfy this equality for all bounded harmonic $h$. For many 
measures $\nu$ the two notions are equivalent, but in general a harmonic 
quadrature domain is \emph{not} uniquely determined by the measure $\nu$ 
(see \cite[Theorem 4.1 and \S~10.4]{GustafssonPutinar}). 
\end{rem}

\subsection{Characterization of extremal measures}
\label{ss:extremal-intro}
The work \cite{ZZ} of Zeitouni and Zelditch and \cite{GhoshNishry1} suggest that
the solution to the hole problem can be understood in terms of constrained minimizers 
of the convex energy functional
\begin{equation}\label{eq:def-funct}
I_{\alpha}(\mu)=-\Sigma(\mu)+2\sup_{z\in\C}\big(U^\mu(z)
-\tfrac{|z|^2}{2\alpha}\big),
\end{equation}
over the class of probability measures on $\C$ which give no mass to the hole,
where $\alpha>0$ is a large parameter.
Here, $U^\mu$ and $\Sigma(\mu)$ denote, respectively, the logarithmic potential 
and the (negative) logarithmic energy
\begin{equation}\label{eq:pot-energy}
U^\mu(z)=\int\log|z-w|\,\diff\mu(w),\quad -\Sigma(\mu)=-\int \log|z-w|\,\diff\mu(z)\diff\mu(w).
\end{equation}
The functional \eqref{eq:def-funct} was originally introduced in 
\cite{ZZ} and studied in \cite{GhoshNishry1}
in connection with the circular hole problem $\G=\D$.
It may be mentioned that $I_\alpha$ is the Large Deviations Principle (LDP)
rate function for the zero process associated to the (Weyl) polynomials
obtained by truncating the GEF.

Here, we study the constrained extremal problem
for a general bounded Jordan domain $\G$. We denote by 
$\mu_{\alpha,\G}$ the extremal measure obtained as the unique 
solution to the problem
\begin{equation}\label{eq:basic-extremal}
\text{minimize}\quad I_\alpha(\mu)\quad \text{subject to}\quad \mu\in\calM_\G,
\end{equation} 
where $\calM_\G$ is the collection of
all probability measures on $\C$ with $\mu(\G)=0$.

If $\partial\G$ is a piecewise smooth curve without cusps, 
we denote its set of corner points by $\calE=\calE(\partial\G)$.
We say that $\mu$ has \emph{regular support} in $\overline{\G}$, 
if $\mathrm{supp}(\mu)$ is contained in a set $\Lambda$ 
consisting of a finite number
of analytic cross-cuts\footnote{A cross-cut is an arc $\gamma\subset\G$ without loops,
which connects two boundary points $z,w\in \partial\G$, thereby splitting $\G$ in two connected 
components.} 
and countably many points, such that in addition 
$\Lambda\cap\partial\G\subset\calE$. 
We denote by $\mathrm{Bal}(\nu,\G^c)$ the \emph{balayage} of a measure $\nu$
to the boundary $\partial\G$. We recall the definition of 
this potential theoretic notion in 
\S~\ref{ss:bal}.

The following result is used to establish our main results for the GEF.

\begin{thm}\label{thm:s-qdom}
There exists an absolute constant $\alpha_0$, 
such that if $\G\subset\D$ is a Jordan domain with 
piecewise $C^2$-smooth boundary without cusps, then there exists a
positive measure $\nu$ with regular support in $\overline\G$,
such that for all $\alpha\ge\alpha_0$ the extremal measure is given by
\[
\diff\mu_{\alpha,\G}=\tfrac{1}{\alpha}\diff\mathrm{Bal}(\nu,\G^c)
+\tfrac{1}{\pi\alpha}\chi_{\D(0,\sqrt{\alpha})\setminus \Omega_\nu}\diffA,
\]
where $\Omega_\nu$ denotes the subharmonic quadrature
domain with respect to $\nu$.
\end{thm}

The proof is contained in \S~\ref{s:variational}.
For comparison with Theorem~\ref{thm:main-GEF}, 
we mention that the balayage of a unit point 
charge at $\lambda\in \G$ to $\G^c$ 
is the harmonic measure $\omega_{\G,\lambda}$.

\subsection{Connection to an obstacle problem}
\label{ss:intro-obst}
Theorem~\ref{thm:s-qdom} is obtained by
a variational argument in two steps, which we briefly describe.

\medskip

\noindent {\sc Step 1:}\:{\em An implicit obstacle problem}.\:
For a \emph{thin obstacle} $g$ in the Sobolev space $H^{\frac12}(\partial\G)$ 
(see \S~\ref{ss:sobolev} below for details), the full obstacle function 
$\phi_\alpha(z)=\tfrac{1}{2\alpha}|z|^2$ and $D=\D(0,\sqrt{\alpha})$,
we consider the class of functions
\begin{equation}\label{eq:class-g-intro}
\calK_g=\Big\{u\in H^1(D):\, u\le g \text{ on }\partial\G,\,\,u\le\phi_\alpha \text{ on } D,
\;\,u=\phi_\alpha \text{ on } \partial D \Big\}.
\end{equation}
Whenever $\calK_g$ is non-empty, there is a unique solution $u_{g}$ to
the obstacle problem
\begin{equation}
\label{eq:obst-g}
\inf_{u\in\calK_{g}}\dir(u)\coloneqq \inf_{u\in\calK_g}\int_{D}|\nabla u|^2\diffA.
\end{equation}
We will show that for $\alpha$ large
enough, there is \emph{some} thin obstacle function $g$ such that 
the extremal measure $\mu_{\alpha,\G}$ equals the Riesz measure 
$\mu_g=\tfrac{1}{2\pi}\Delta u_g$. 
In fact, it holds that $I_\alpha(\mu_g)=\dir(u_g)+C(\alpha)$ 
for an explicit constant $C(\alpha)$, and
if we denote by $\widetilde{g}$ the harmonic extension of $g$ to $\G$, the extremal 
problem \eqref{eq:basic-extremal} reduces to
\begin{equation}\label{eq:reform-extremal-intro}
\text{minimize }\quad \dir(u_g) \quad\text{subject to}\quad \widetilde{g}\le 
\phi_\alpha\text{ on }\G\,\text{ and }\, g\in H^{\frac12}(\partial\G).
\end{equation}
Here, the condition that the harmonic extension $\widetilde{g}$ lies below the
full obstacle $\phi_\alpha$ ensures that $\mu_g(\G)=0$.

\medskip

\noindent {\sc Step 2:}\:{\em Finding the optimal thin obstacle}.\:
We wish to find the thin obstacle $g$ which solves \eqref{eq:reform-extremal-intro}, and to this end 
we devise a perturbation argument. We put 
\[
B_\alpha(\mu)= \sup_{z\in\C}
\big(U^{\mu}(z)-\tfrac{1}{2\alpha}|z|^2\big)
\]
and define the {\em interior coincidence set} 
\[
\calI=\big\{z\in\overline{\G}: U^{\mu_{\alpha,\G}}(z)-\tfrac{1}{2\alpha}|z|^2
=B_\alpha(\mu_{\alpha,\G})\big\}.
\]
If $g$ is the solution to \eqref{eq:reform-extremal-intro}, then $\calI$
is also the coincidence set on $\overline{\G}$ for $u_g$ with 
the full obstacle $\phi_\alpha$.
For a harmonic function $h$ on $\G$ 
which is negative on $\calI$, 
we denote by $u_\epsilon=u_{g_\epsilon}$ the solutions to the 
obstacle problem \eqref{eq:obst-g} 
where $g$ is replaced by
$g_\epsilon=g+\epsilon h\vert_{\partial\G}$ with $\epsilon>0$.
It turns out that for such $h$, the Riesz measures 
$\mu_\epsilon\coloneqq \mu_{g_\epsilon}$ belong to $\calM_\G$ for small $\epsilon$.
One of the keys to our approach is the variational formula of 
Corollary~\ref{cor:variational} below, which reads
\begin{equation}\label{eq:intro-variational}
\dir(u_{\epsilon})= \dir(u_{0})
-4\pi \epsilon\int_{\partial\G} h\,\diff\mu_{0}+\ordo(\epsilon),
\qquad \epsilon\to 0^+.
\end{equation}
In view of \eqref{eq:reform-extremal-intro}, the integral appearing on the right-hand side must be 
\emph{negative} for all admissible perturbations $h$, or else $g$
would not be extremal.
Using harmonic interpolation, this inequality yields Theorem~\ref{thm:s-qdom} provided 
that the set $\calI$ is finite.
When $\partial\G$ is $C^2$-smooth, we show that 
the potential $U^{\mu_{\alpha,\G}}$ is {\em non-degenerate}, meaning that 
\begin{equation}\label{eq:non-deg}
\sup_{z\in\partial \G}\big(U^{\mu_{\alpha,\G}}(z)
-\tfrac{|z|^2}{2\alpha}\big)<B_\alpha(\mu).
\end{equation}
In particular, the coincidence set $\calI$ is separated 
from $\partial\G$, see Theorem~\ref{thm:s-ndeg} below. 
A classical regularity result of Caffarelli and Rivi{\`e}re 
\cite{CaffarelliRiviere} implies that under such a separation condition, 
either $\calI$ has positive area (impossible),
or it is a finite set.

For a more general Jordan domain $\G$, we argue 
by approximation with smooth domains from inside.

\begin{rem}\label{rem:topology}
We stress that the assumption that $\G$ is simply connected
is crucial for the conclusion of
Theorem~\ref{thm:main-GEF}. Indeed, in the theorem of Caffarelli-Rivi{\`e}re,
closed curves are also possible components of the coincidence set. 
Since we work with simply connected
$\G$ and a subharmonic full obstacle, the region enclosed by the
curve would then be part of the coincidence set as well. 
Since $\calI$ has vanishing area, this cannot happen.
If $\G$ is not simply connected, 
then we cannot draw this conclusion, and the interior
coincidence is not finite in general.
\end{rem}

\subsection{Related work}
\label{ss:related}
In one dimension, hole probabilities 
are also known as \emph{gap probabilities}, and
have been studied extensively.
We mention in particular work by Majumdar, Nadal, 
Scardicchio and Vivo \cite{Nadal} on a log-gas with quadratic potential, 
where the probability of large gaps is computed, and the conditional 
distribution is described in detail. A closely related problem
is the computation of \emph{persistence probabilities} for real stochastic processes, 
see \cite{Persistence}.

In the two-dimensional setting, the hole event 
has been studied for Coulomb gases, including the Ginibre ensemble. In this case,
the LDP functional is the weighted logarithmic energy (see 
\cite{SaffTotik}).
In the work by Armstrong, Serfaty and Zeitouni \cite{ASZ} the 
constrained minimization problem for the limiting
deficiency and overcrowding events\footnote{In the 
$p$-deficiency (resp.\ $q$-overcrowding) problem,
the hole event is replaced by the event that $\G$ contains at 
most a fraction $p$ with $0\le p<1$ of the expected number of 
points (resp. at least a fraction $q> 1$ of the expected number).}
is studied with variational techniques;
in the work \cite{Adhikari} by Adhikari 
and Reddy, a connection 
to the Ginibre ensemble is drawn and 
the hole probabilities are computed for a
large class of holes. 
Shirai \cite{Shirai} studies annular holes for 
Ginibre-type point processes,
and describes the 
distribution of points on the hole event, in particular near the 
singular part of the limiting measure.

For the random zeros considered in this work, the asymptotics of 
hole probabilities was first studied by Sodin and Tsirelson
\cite{SodinTsirelson} who estimated the decay rate of 
hole probabilities for circular holes. 
In \cite{Nishry},
the exact leading order asymptotics was found, 
and as mentioned previously the existence of forbidden regions was observed
for the first time in \cite{GhoshNishry1} in the context of a 
circular hole. 
For a more detailed account of related work, see \cite{GhoshNishry2}
and the references therein.

The emergence of a 
forbidden region seems unique to random zeros,
as it comes about due to the non-local nature of the 
term $B_\alpha$ of the functional $I_\alpha$.
To our knowledge, similar phenomena have not been discovered elsewhere.
To the best of our knowledge, the 
functional $I_\alpha$ appeared first in work by 
Zeitouni and Zelditch \cite{ZZ}, 
who established a LDP 
for zero sets of random polynomials with Gaussian coefficients.
This was originally carried out 
on Riemann surfaces (see also the works by 
Butez \cite{Butez} and Zelditch \cite{Zelditch}).

The hole event is also interesting for dynamic point processes, 
for instance the dynamic GEF obtained
by letting the random variables $\xi_n=\xi_{n,t}$ 
undergo an independent Ornstein-Uhlenbeck evolution. 
It was shown by Hough \cite{Hough} that the
probability that a hole to persist for times 
$0\le t<T$ is exceedingly small (of order
$\exp(-T\e^{cL^2})$). For other 
dynamic point processes such as lattice gases,
one can answer more delicate questions, 
such as how the hole appears and 
how it subsequently disappears, see the article 
\cite{Meerson} by Krapivsky, Meerson and Sasorov.

For further results pertaining to probabilities of rare events for
stationary random processes, 
we refer to \cite{akemann, CardinalSine, Basu, Chafai, FeldheimFeldheim, Krishnapur, PeresVirag}.

Recently, questions pertaining to quantitative stability
of obstacle problems have been studied. Among recent results we mention
work by Serfaty and Serra \cite{SerfatySerra}, 
and Blank and LeCrone \cite{Blank}.
This connects to our developments in \S~\ref{s:perturb-stab}.

Quadrature domains appear in several different areas of mathematics.
Examples include the study of aggregation models in 
probability through the Diaconis-Fulton smash sum \cite{Diaconis} studied by Levine and Peres
\cite{LevinePeres}, fluid mechanics and random matrix theory 
as explained in \cite[\S 6.2]{GTV} (see also the work by Mineev-Weinstein, 
Wiegmann and Zabrodin~\cite{MWZ}), and complex dynamics
as illustrated by the work of Lee, Lyubich and Makarov \cite{Makarov1, Makarov2}.

\subsection{Structure of the paper}
Section~\ref{s:background} contains preliminary material 
from geometric function theory, potential theory
and the theory of free boundary problems.

The main ideas underlying the proof of Theorem~\ref{thm:s-qdom} were sketched above, 
and this program is carried out in
\S~\ref{s:perturb-stab} to \S~\ref{s:variational}. 

In \S~\ref{s:reg} and \S~\ref{s:perturb-holes-quant} we study regularity for the potential of
the extremal measure, and quantitative
stability under domain perturbations, respectively. 
These results are used in \S~\ref{s:Weyl} to prove Theorem~\ref{thm:main-GEF}
(see \ref{ss:guide:probab} for an overview of the proof).

In \S~\ref{s:examples} we obtain Theorem~\ref{thm:a-circ} as a consequence
of a general form of Theorem~\ref{thm:main-GEF} and a \emph{sufficient condition}
for a measure $\mu_0$ to be extremal for $I_\alpha$ over $\calM_\G$. We also
discuss a family of examples of holes for which the forbidden regions are two-point
quadrature domains.

Finally, the appendix (a collaboration with S. Ghosh) contains a result on 
approximation of measures using weighted Fekete
points, which is used to obtain lower bounds for the hole probability.

\section{Preliminaries and notation}\label{s:background}
\subsection{Notational conventions}
We use {\em positive} and {\em negative} in the weak sense,
allowing e.g. the value zero for a positive parameter. Similarly, {\em greater
than} and {\em smaller than} allow for equality, unless otherwise stated.

Unless we indicate otherwise, inequalities between functions in $L^p$ and Sobolev
spaces are interpreted in the almost everywhere sense.

For a set $E\subset\C$, the standard symbols $E^c$, $E^\circ$ and $\overline{E}$
are used to denote complement, interior and closure (as a subset of
Euclidean space $\C=\R^2$). Moreover, $|E|$ will be used to denote the Lebesgue
measure of $E$. If it is understood that $E$ is confined to a one-dimensional
rectifiable subset $\Gamma$, then $|E|$ denotes the arc-length measure of 
$E$. When more precision is needed, we express these measures
in terms of the length and area elements, which are denoted by
$\diffs$ and $\diffA$, respectively.

We will oftentimes deal with asymptotics of quantities depending on
a parameter, and in particular deal with inequalities between quantities
depending on that parameter.
If $(a(t))_{t\in \calT}$ and 
$(b(t))_{t\in \calT}$ are quantities depending on
the parameter $t$ in an index set $\calT$, then we write
\[
a(t)\lesssim b(t)
\]
if there exists a constant $C\in\R^+$ such that for all $t\in\calT$
we have
\[
a(t)\le C b(t).
\]
The constant $C$ may depend on other parameters.
If $a(t)\lesssim b(t)$ and $a(t)\gtrsim b(t)$
hold simultaneously, then we write $a(t)\cong b(t)$.
In addition we will use the standard notation 
$f=\Ordo(g)$ and $f=\ordo(g)$.

For a measure $\mu$ on $\C$, we use the notation $\mu^\s$ and $\mu^\c$ for the
singular and continuous parts in the Lebesgue decomposition 
with respect to area measure.

\subsection{Logarithmic potentials}
For a signed measure $\mu$, we recall that the logarithmic potential $U^\mu$ of
$\mu$ is the convolution of $\mu$ with the logarithmic kernel
\[
U^\mu(z)=\int_\C\log|z-w|\,\diff\mu(w),
\]
where we have chosen the normalization such that
\[
\tfrac{1}{2\pi}\Delta U^\mu=\mu,
\]
in the sense of distributions.
The logarithmic energy of $\mu$ is the quantity
\[
-\Sigma(\mu)=-\int\log|z-w|\,\diff\mu(z)\diff\mu(w)=-\int U^\mu(z)\diff\mu(z)
\]
which we consider for compactly supported finite measures $\mu$ 
for which  the function $|\log |z-w||$ is integrable with 
respect to the product measure $\mu\times\mu$
(so that the above computation is justified by Fubini's theorem).
These measures are known  as measures
of finite logarithmic energy. A standard reference for potential
theory in the plane is the monograph \cite{SaffTotik} by Saff and Totik. 
As a word of caution to the reader, we mention that in \cite{SaffTotik} the 
logarithmic potential is defined to have the \emph{opposite} 
sign compared to what we use here.

\subsection{Sobolev spaces}
\label{ss:sobolev}
Unless otherwise stated, every domain $D$ considered in this article will
be a Lipschitz domain, meaning that at each boundary point $z_0$ 
there exists a number $\epsilon>0$ such that $\partial D\cap \D(z_0,\epsilon)$ 
is the graph of a Lipschitz function $f:[0,1]\to\C$, after applying an appropriate
rotation.

For the basic properties and definitions of Sobolev spaces, we refer the
reader e.g.\ to the monograph \cite{Adams}. Below, we will merely fix notation and
recall some properties that are especially important in what follows.

For a (Lipschitz) domain $D\subset\C$ 
we consider the space $H^1(D)$
of functions $u\in L^2(D)$ whose first order (distributional) 
partial derivatives lie in $L^2(D)$ as well.
If $\Gamma$ is a closed Lipschitz curve enclosing some domain $\Omega$, 
we denote the (Dirichlet) trace space
of $H^1(\Omega)$ by $H^{\frac12}(\Gamma)$. The space 
$H^{\frac12}(\Gamma)$ is itself
a Sobolev space with an intrinsic characterization \cite{Ding}. 
The dual spaces are defined as usual, 
and they are denoted by $H^{-1}(D)=(H^1(D))^*$ and 
$H^{-\frac12}(\Gamma)=(H^{\frac12}(\Gamma))^*$, respectively.
The spaces $H^1(D)$, $H^{\frac12}(\Gamma)$ as well as their duals 
are Hilbert spaces, and as such are in particular weakly
sequentially compact.
For the precise definitions and further basic properties of the 
spaces mentioned above, the reader 
may consult e.g.\ the monograph \cite{Adams}.

We mention that whenever $\Gamma$ is a simple 
closed Lipschitz curve lying inside $D$ and $u$ is a function in
$H^1(D)$ with $\Delta u\in L^2(D)$ (in particular, if $u$ is harmonic), 
then we can also define a Neumann trace 
of $u$ on $\Gamma$.
We denote by $\calN_\Gamma(u)$
the unique trace function which for smooth functions $u$ satisfies
\[
\calN_{\Gamma}(u)=
-\Big(\partial_{\mathrm{n}_{D^+}}+\partial_{\mathrm{n}_{D^-}}\Big)u,
\]
where $D^\pm$ denote the two components of $D\setminus \Gamma$, 
and $\partial_{\mathrm{n}_{D^\pm}}$
denote the normal derivatives in the outward normal direction 
from inside $D^+$ and $D^-$, respectively.
By a standard mollification argument, it is readily verified that 
$\calN_{\Gamma}(u)\in H^{-\frac12}(\Gamma)$ whenever
$u\in H^{1}(D)$ with $\Delta u\in L^2(D)$ (cf.\ Proposition~\ref{prop:reg} below).

\subsection{Harmonic measures and the Poisson kernel}
For a Lipschitz domain $\Omega$ 
and a point $z_0\in\Omega$, the harmonic measure
$\omega_{\Omega,z_0}(\cdot)=\omega(z_0,\cdot,\Omega)$ 
is a Borel measure on $\partial\Omega$, 
with the property that for any $E\subset \partial\Omega$, the measure
$\omega_{\Omega,z_0}(E)$ is the probability 
that a Brownian motion started at $z_0$
exits $\Omega$ through the set $E$.
We will only need some very basic properties of harmonic measure.
The first property is that when $\Omega$ is simply connected, the density
of harmonic measure with respect to arc length measure $\diffs$
on $\partial\Omega$ is given by the Poisson kernel:
\begin{equation}\label{eq:HM-density}
\frac{\diff\omega_{\Omega,z_0}}{\diffs}=P_{\Omega}(z_0,\cdot)
\quad \text{on }\; z\in\partial\Omega.
\end{equation}
This allows us to estimate harmonic measures using conformal 
invariance and Kellog's
theorem, which we proceed to state (see p. 426 in \cite{Goluzin}).
For a domain $\Omega$ we denote by $C^{k,\beta}(\Omega)$ the space of 
$k$ times differentiable functions in $\Omega$, 
whose partial derivatives of order $k$ satisfy a H{\"o}lder 
condition with exponent $\beta$ on $\overline{\Omega}$.

\begin{thm}\label{thm:Kellog}
Denote by $\Omega$ a simply connected domain with $C^{k,\beta}$-smooth
Jordan curve boundary, for some integer $k \ge 1$ and H{\"o}lder exponent 
$0<\beta<1$.
Denote by $\varphi$ a conformal mapping of $\D$
onto $\Omega$ and by $\phi$ its inverse. 
Then $\varphi$ and $\phi$ are $C^{k,\beta}$-smooth on the closed unit disk and the
closure of $\Omega$, respectively.
\end{thm}

The statement for the inverse mapping is, strictly speaking, 
not included in the presentation in \cite{Goluzin}, but
follows from standard methods in view of the non-vanishing of the derivative
on the closure $\overline{\D}$.

Another important feature of harmonic measures is the following simple monotonicity 
property. If $\Omega$ is a subset of $\Omega'$ and $E\subset\partial\Omega$ 
is a subset of $E'\subset\partial\Omega'$, then for any $z_0\in\Omega$,
it trivially holds that
\begin{equation}\label{eq:minotonicity-hm}
\omega_{\Omega,z_0}(E)\le \omega_{\Omega',z_0}(E').
\end{equation}
This is most easily seen from the equivalent definition of 
$z\mapsto \omega(z,E,\Omega)$
as the solution to the Dirichlet problem with boundary data $\chi_{E}$.
We can also argue probabilistically as follows. Any Brownian path which starts at $z_0$
and exits $\Omega$ through $E$ is also a path in $\Omega'$ which exists $\Omega'$
through $E$, so in particular it exits through $E'$. 

\subsection{Piecewise smooth domains and conformal mappings near corners}
A simple closed curve $\Gamma$ is 
said to be $C^2$-smooth if it admits a $C^2$-smooth
parameterization $f:[0,1]\to \Gamma$ with non-vanishing derivative. 
If $\Gamma$ is a simple arc, we say that it is $C^2$-smooth if
there exists a larger arc $\tilde{\Gamma}$ with a $C^2$-smooth (surjective)
parameterization $f:[0,1]\to\tilde{\Gamma}$ with non-vanishing derivative,
such that $\Gamma$ is contained in the image 
$f([\epsilon,1-\epsilon])$ for some $\epsilon>0$.

\begin{defn}
We say that a Jordan curve $\Gamma$ is piecewise smooth
if it is made up of finitely many $C^2$-smooth arcs $(\Gamma_j)_{j=1}^M$.
When two arcs $\Gamma_j$ and $\Gamma_{j+1}$ meet 
(we identify $\Gamma_{M+1}$ with $\Gamma_1$), they form an interior angle 
$\pi\sigma$, where $\sigma=\sigma(\Gamma_j,\Gamma_{j+1})\in[0,2]$. 
If each angle $\sigma\pi$ lies in the open interval 
$(0,2\pi)$ we say that $\G$ is piecewise smooth
without cusps.

A domain is said to be piecewise smooth if its boundary curve is.
\end{defn}

The following is a small modification of Theorem 3.9 of \cite{Pommerenke}
(cf.\ Exercise 3.4.1 \emph{loc.\ cit.}). 

\begin{prop}\label{prop:conf-corner}
Denote by $\G$ a piecewise smooth simply connected domain without cusps, 
and let $w_0\in\partial\G$ a corner point with interior angle $\pi\sigma$.
If $f$ denotes a conformal mapping of $\D$ onto $\G$ with $f(\zeta)=w_0$,
there exists a non-zero complex constant $a$ such that for any $\epsilon>0$,
we have with $\sigma'=\min\{\sigma,1\}$ that
\[
f(z)=f(\zeta)+a(z-\zeta)^\sigma + \Ordo(|z-\zeta|^{\sigma+\sigma'-\epsilon}),
\qquad z\in\overline{\D}.
\]
and
\[
\frac{f'(z)}{(z-\zeta)^{\sigma-1}}=a\sigma+\Ordo(|z-\zeta|^{\sigma'-\epsilon}),
\qquad z\in\overline{\D}.
\]
\end{prop}

\begin{proof}[Proof sketch]
We explain how the proof of Theorem 3.9 of \cite{Pommerenke} should be modified
to yield this claim (in this proof we use the terminology of that reference).
Without loss of generality we may assume that $w_0=0$ and $\zeta=1$.

As explained in the proof of \cite[Theorem~3.9]{Pommerenke}, we may localize 
to a piecewise smooth domain $\G_0$ containing $\G\cap\D(0,\delta)$, for some
small $\delta>0$. This localization is done so that $0$
is the only singular boundary point at $\partial\G_0$, and amounts to constructing a
conformal mapping $\varphi$ of $\D$ onto $f^{-1}(\G_0)\subset\D$ (cf. fig 3.3 in \emph{loc.\ cit.})
which is $C^{1,1-\epsilon'}$-smooth for any given $\epsilon'>0$.

Denote by $H$ the domain resulting from applying the straightening mapping 
$w^{\frac{1}{\sigma}}$ to $\G_0$ near $0$,
with $0$ being mapped to the origin.
We claim the boundary of $H$ consists of two arcs
(corresponding to the two arcs 
$\Gamma^\pm$ meeting at $w_0$)
which are both $C^{1,\sigma(1-\epsilon')}$-smooth for any 
$\epsilon'>0$, and that are tangent at $0$. 
Indeed, the mapping straightens the angle at $w_0$ by construction,
and if $w$ denotes the given parameterization 
of one of the arcs $\Gamma^\pm$, 
then a one-sided parameterization $v$ of $\partial H$ at $0$
is obtained by $v(t)=(w(ct^\sigma))^{1/\sigma}$ where $c=1/|w'(0)|$. 
Following along the lines of the original proof, we have
\begin{equation}\label{eq:v'-def}
v'(t)=c^{\frac1\sigma}w'(ct^\sigma)u(ct^\sigma),\qquad 
u(s)=\big(\tfrac{1}{s}w(s)\big)^{\frac{1-\sigma}{\sigma}}.
\end{equation}
The non-vanishing of $w'(0)$ and hence of $w(s)/s$ for $s$ 
small shows that the modulus of continuity of the function
$u(s)$ is controlled by that of $s^{-1}w(s)$. Here, 
we recall that the modulus
of continuity $\omega(\delta,F)$ of a function $F$ is defined by
\[
\omega(\delta,F)=\sup_{|z-w|\le \delta} |F(z)-F(w)|.
\]
In view of the computation
\[
\frac{w(s_1)}{s_1}-\frac{w(s_1)}{s_2}=\int_0^1(w'(s_1x)-w'(s_2x))\diff x
\]
it follows that $\omega(\delta,u)\lesssim \omega(\delta,w')$. 
It follows from \eqref{eq:v'-def} that $\omega(\delta,v')$ 
is controlled by $\omega(c\delta^{\sigma'},w)$ where $\sigma'=\min\{\sigma,1\}$
(cf.\ (8) and the equation immediately following it in \textit{loc.\ cit.}). 
So, after straightening the corner, the domain is $C^{1,\sigma'(1-\epsilon')}$-smooth 
for any $\epsilon'>0$, where $\sigma'=\min\{\sigma,1\}$. The claim then
follows by invoking Kellogg's theorem 
(Theorem~\ref{thm:Kellog}, cf.\ \cite[Theorem 3.6]{Pommerenke})
for the mapping $h:\D\to H$. Indeed, the mapping $f$ is given in terms of $h$ and the
localization mapping $\varphi$ by
$f(z)=(h(\varphi^{-1}(z))^{\sigma}$, where $h\circ\varphi^{-1}$ is a 
$C^{1,\sigma'(1-\epsilon')}$-smooth for any $\epsilon'$.
\end{proof}

\begin{rem}\label{rem:prop-conf-corner}
For future reference we record that in the above proof, we found that
$f(z)=w_0+(g(z)-g(1))^{\sigma}$, where for any $\epsilon>0$ the mapping 
$g$ belongs to the space $C^{1,\sigma'(1-\epsilon)}$ with $\sigma'=\min\{\sigma,1\}$,
where $w_0=f(1)$. 
Hence both $f$ and $f^{-1}$ are H\"older continuous.
\end{rem}

\begin{rem}\label{rem:prop-conf-corner-stab}
Let $H_t$  and $H$ be a family of Jordan domains with $C^2$-smooth boundaries,
such that $H_t$ approaches $H$ as $t\to 0$. 
Suppose $\partial H_t$ has a parameterization $w_t$, and
denote by $w$ the corresponding parameterization of $\partial H$. 
Then, by a stability theorem of Warschawski's \cite{Warschawski}, 
if $\lVert w_t''-w''\rVert_{L^2[0,1]}\to 0$ and $w_t\to w$ uniformly, 
then we have
\[
\sup_{z\in\overline{\D}}|h_t'(z)-h'(z)|\to 0.
\]
Assume now that $\G_t$ and $\G$ are piecewise 
smooth, such that the above condition
holds for all the $C^2$-smooth boundary subarcs $\Gamma_{t,j}$ and $\Gamma_j$. 
Then apply Warschawski's theorem
to the mapping functions $h_t$ and $h_0$, and 
$\varphi_t$ and $\varphi_0$ appearing in the proof of 
Proposition~\ref{prop:conf-corner} onto the domains $H_t$ and $H$,
obtained by straightening
a given angle at $\G_t$ and $\G$. 
This shows in particular that the implicit constants in 
Remark~\ref{rem:prop-conf-corner} may be taken uniform over appropriately
convergent sequences of piecewise smooth domains $\G_t$. Indeed, 
the functions $g_t$ with $f_t(z)=f_t(1)+(g_t(z)-g_t(1))^{\sigma_t}$ appearing
in Remark~\ref{rem:prop-conf-corner} ($\sigma_t$ being the $t$-dependent angle)
are $C^1$ with uniform control on their $C^1$-norm.
\end{rem}

\begin{defn}\label{defn:cone}
A domain $\G$ is said to meet a uniform cone condition,
if there exist numbers 
$\epsilon$ and $\theta>0$, such that for any boundary point $z_0\in\partial\G$
there exists a $c\in[0,2\pi)$,
such that the planar cone 
\[
W=\e^{\imag c}\big\{z\in\C:|\arg(z-z_0)|\le \theta\big\},
\]
with apex at $z_0$ locally contains $\partial\G$, in the sense that  
$\partial\G\cap \D(z_0,\epsilon)\subset W$.

The parameter $\theta$ is called the {\em aperture} of the cone.
\end{defn}

\subsection{Balayage measures}\label{ss:bal}
For a measure $\nu\in H^{-1}(D)$ where $D$ is a bounded Lipschitz domain, 
we define its balayage $\mathrm{Bal}(\nu,D^c)$ (sweeping)
to the complement $D^c$ as the unique measure $\mu$ supported on 
$\partial D$ which satisfies
\[
U^\mu(z)=U^{\nu}(z),\qquad z\in D^c.
\] 
By the maximum principle, it follows that we have the global bound
\[
U^\mu(z)\ge U^\nu(z),\qquad \text{q.e. }z\in\C,
\]
and from the equality of the potentials near infinity
it follows that $\mu$ and $\nu$ have the same mass.
If $\mu$ has finite logarithmic energy and $D$ is a Lipschitz domain, 
then for $\nu=\mathrm{Bal}(\mu,D^c)$ we have $-\Sigma(\nu)\le -\Sigma(\mu)$.

\begin{rem}
More generally, if $\nu$ is supported on $\C$, we define 
the balayage measure as 
\[
\mathrm{Bal}(\nu,D^c)=\mathrm{Bal}(\chi_D\nu,D^c) 
+ \chi_{D^c}\,\nu,
\]
provided that $\chi_D\,\nu$ is regular enough.
\end{rem}

Let us only say a few words about the existence of Balayage measures.
In view of the assumption $\nu\in H^{-1}(\Omega)$, 
the Dirichlet problem on $\Omega$ with boundary datum 
$f=U^\nu\big\vert_\Omega\in H^{\frac12}(\partial \Omega)$ admits a solution 
$V\in H^1(\Omega)$, and the function $V_0$ which equals $V$ on $\Omega$ and 
$U^\mu$ elsewhere has (distributional) Laplacian in the Sobolev
space $H^{-1}(\Omega)$, in fact in $H^{-\frac12}(\Gamma)$.
The balayage measure of $\nu$ to $\Omega^c$ then equals
$\mu=\tfrac{1}{2\pi}\Delta V_0$.
In particular, under these 
mild conditions, the balayage measure exists. For an introductory 
account of these matters, we refer to \cite[Ch. II]{SaffTotik}.

Let $\nu$ be a measure supported on a bounded domain $\Omega$, which we assume to
be simply connected. Let 
$\mu=\mathrm{Bal}(\nu,\Omega^c)$. For $w\in\Omega$, 
the potential of the Balayage measure
$\mathrm{Bal}(\delta_w,\Omega^c)$ is given by 
\[
U^{\mathrm{Bal}(\delta_w,\Omega^c)}(z)=
\begin{cases}
\log\frac{|z-w|}{|\phi_w(z)|},& z\in\Omega\\
\log|z-w|,& z\in\Omega^c,
\end{cases}
\]
where $\phi_w$ is the Riemann mapping of $\Omega$ onto $\D$ with $\phi_w(w)=0$.
In view of this formula, it is clear that we have the representation
\begin{equation}\label{eq:bal-pot-nu}
U^\mu(z)=\int \log\frac{|z-w|}{|\phi_w(z)|}\diff\nu(w),\qquad z\in\Omega,
\end{equation}
see eq. (5.3) on p. 124 in  \cite{SaffTotik}.
Similarly, one can deduce the representation formula
\begin{equation}\label{eq:bal-dens-nu}
\frac{\diff\mathrm{Bal}(\nu,\Omega^c)}{\diffs}
=\int\frac{\diff\omega_{\Omega,z}(\cdot)}{\diffs}
\,\diff\nu(z)
\end{equation}
for the density with respect to arc length on $\partial\Omega$, whenever $\nu$
is compactly
supported on the domain $\Omega$ (c.f.\ Eq.\ (4.44) on p. 122 in \cite{SaffTotik}).

\subsection{Quadrature domains}
We already encountered the notion of 
(subharmonic\footnote{Quadrature domains can also be defined for 
harmonic or analytic functions instead of subharmonic ones.
Here, the notion of quadrature domains 
refers only to the subharmonic version, which
is the most restrictive of the three.}) quadrature 
domains in Definition~\ref{defn:q-dom}. 
Let $\nu$ be a positive atomic measure 
\[
\nu=\sum_{\lambda\in\calI}\omega_\lambda\delta_\lambda
\]
with finite support $|\calI|<\infty$.
Recall that a planar open set $\Omega$ (not necessarily connected) 
is said to be a quadrature domain for subharmonic 
functions (or simply a `subharmonic quadrature domain') with respect to the measure 
$\nu$ if, for any $L^1$-integrable subharmonic function $u$ on 
$\Omega$ we have
\[
\sum_{\lambda\in\calI}
\omega_\lambda u(\lambda) \le \frac{1}{\pi}\int_{\Omega}u(z)\diffA(z).
\]
The simplest quadrature domain is a disk, as 
follows from the sub-mean value property for 
subharmonic functions.
The simplest non-trivial example is given by 
the \emph{Neumann oval}, which we will return to later in 
\S~\ref{s:examples}.

Given a finitely supported atomic 
measure $\nu$, there exists 
a unique subharmonic quadrature domain $\Omega_\nu$ 
relative to $\nu$\footnote{Actually, this holds for 
much more general measures: $\nu$ has to be 
sufficiently concentrated, which holds, for instance
for measures $\nu$ of \emph{regular support} 
as defined before Theorem~\ref{thm:s-qdom}.} 
(see \cite{Sakai2}, Theorems 3.4 and 3.5), 
and it may be constructed by a sweeping 
out process known as partial balayage 
(see \cite[Ch. 2]{Sakai2}). We will need the following
characterization of $\Omega_\nu$, see 
Theorem~4.8 in \cite{GustafssonShahgholian}. 
We denote by $\mathrm{SH}(\C)$ the class of subharmonic functions
on $\C$.

\begin{thm}\label{thm:q-dom}
Given a finitely supported positive atomic measure $\nu$, 
the quadrature domain
$\Omega_\nu$ is the non-coincidence set 
  $\{z\in\C: U(z)<\tfrac12|z|^2\}$ for the subharmonic envelope
function
\[
U(z)=\sup\big\{u(z):\,
u\in\mathrm{SH}(\C),\,\,u(z)\le \tfrac12|z|^2-U^\nu(z)\big\}.
\]
\end{thm}
In view of Sakai's regularity theorem 
(see Theorem~\ref{thm:sakai} below), the 
boundary $\partial\Omega$ is piecewise real-analytic. As mentioned above,
Aharonov and Shapiro \cite{AharonovShapiro} established the stronger result
that the boundary of a subharmonic quadrature domain is algebraic.

\begin{rem}
We use the notion of quadrature domain in the extended sense when the measure 
$\nu$ is supported on union of a finitely many analytic curves 
and countably many points. The same definition applies, 
and Theorem~\ref{thm:q-dom} holds also in this case.
\end{rem}

\subsection{Several variants of the obstacle problem}
Obstacle problems
may be posed in many different ways, which 
is one of the reasons why they are so useful. 
The {\em Dirichlet energy}
$\dir(u)$ of $u\in H^1(D)$ is defined as
\begin{equation}\label{eq:dir}
\dir(u)=\int_D|\nabla u|^2\,\diffA.
\end{equation}
We remind the reader that for $f$ and $g$ in the Sobolev space $H^1(D)$
(or $H^{\frac12}(\Gamma)$, 
the inequality $f\le g$ is understood as
$f(z)\le g(z)$ a.e.\ $z\in D$ (resp. $z\in\Gamma$).
\begin{defn}\label{defn:obst-bdd}
\noindent {\rm (a)}\, For a Lipschitz 
domain $D$ and functions $\psi\in H^1(D)$, 
and $f\in H^{\frac12}(\partial D)$, we 
denote by $\calK_{\psi}^{f}=\calK_{\psi}^{f}(D)$ 
the class
\begin{equation}\label{eq:obst-class}
 \calK_{\psi}^{f}=\big\{u\in H^1(D): 
 u\le \psi,\,u\vert_{\partial D}=f\big\}.
\end{equation}
The minimizer $v_0$ of $\dir(v)$ over 
$v\in \calK_{\psi}^{f}$ is said
to solve the obstacle problem with full obstacle $\psi$  
and boundary datum $f$.

\smallskip

\noindent {\rm (b)}\, If moreover $\Gamma$ denotes a Lipschitz curve whose 
closure is contained in $D$,
and $g\in H^{\frac12}(\Gamma)$, 
we denote by $\calK_{\psi,g}^{f}=\calK_{\psi,g}^{f}(D)$ the class
\begin{equation}\label{eq:obst-class-mixed}
\calK_{\psi,g}^{f}=\big\{u\in H^1(D): 
u\le \psi,\,u\vert_{\Gamma}\le g,\,u\vert_{\partial D}=f\big\}.
\end{equation}
The minimizer $v_0$ of $\dir(v)$ 
for $v\in \calK_{\psi,g}^{f}$ is said
to solve the mixed obstacle problem with thin obstacle $g$ on $\Gamma$, 
full obstacle $\psi$ on $D$, and boundary datum $f$.
\end{defn}

When the boundary datum agrees with the full obstacle, 
we denote the classes simply by $\calK_\psi$ and $\calK_{\psi,g}$, respectively.
To simplify terminology, if $\calK$ is any convex and nonempty subset of $H^1(D)$, 
we speak of the solution to the obstacle
problem for $\calK$.

\begin{rem}[\S 3.19, 3.21 and in \cite{HKM}]\label{rem:variational-basic}
Whenever the class $\calK$ is non-empty, 
there exists a unique solution to any of the
obstacle problems of Definitions~\ref{defn:obst-bdd}. 
Moreover, the solution is always subharmonic.
In addition, if $v_0$ denotes the
minimizer the Dirichlet energy over the convex 
class $\calK$, it holds that
\[
\int \nabla v_0\cdot\nabla(v-v_0)\diffA\ge 0,\qquad v\in\calK.
\]
\end{rem}

\begin{rem}\label{rem:one-obst}
If $G$ denotes the harmonic function 
on $D\setminus \Gamma$ which equals $g$
on $\Gamma$ and $f$ on $\partial D$, then the solution of the mixed obstacle
problem for $\calK_{\psi,g}^{f}$ agrees with the classical solution for 
$\calK_\eta$, where
\[
  \eta=\min\{G,\psi\}.
\]
To see why the two solutions agree, one simply 
notices that the two classes agree: 
indeed, for any $u\in \calK_{\psi,g}^{f}$ we
have $u\le G$ by the maximum principle. Moreover, 
$u$ has the right boundary data, and $u\le \psi$. 
Hence, $u\in \calK_{\eta}$. The reverse 
direction follows since traces respect
inequalities among functions in $H^1(D)$.
\end{rem}

We will moreover need to consider obstacle problems 
for \emph{unbounded} domains. In this situation, 
minimization of Dirichlet energy 
does not make sense, but we instead rely on an envelope 
formulation involving subharmonic functions 
(see e.g.\ \cite{Convexity} for background on subharmonic functions).
For the definition, we need the following notions.
We denote by $\mathrm{SH}_1(\C)$ the class of 
subharmonic functions of growth
\[
v(z)=\log|z|+\Ordo(1),\qquad z\to\infty.
\]
Equivalently, $\mathrm{SH}_1(\C)$ is the class of 
subharmonic functions whose Riesz masses are probability
measures with finite logarithmic energy. When working 
with global obstacle problems on the entire complex plane,
we will consider obstacles $\psi$ subject to the growth condition
\begin{equation}\label{eq:growth-phi}
\liminf_{|z|\to\infty}\frac{\psi(z)}{\log|z|}>1.
\end{equation}

The following definition is central to our work.

\begin{defn}\label{defn:obst-unb}
{\rm (a)}\, 
Let $\psi\in H^1_{\mathrm{loc}}(\C)$ be a 
real-valued function subject to the growth bound 
\eqref{eq:growth-phi}.
We say that $v_0$ solves the 
global obstacle problem with full obstacle $\psi$ if
$v_0\in\mathrm{SH}_1(\C)$ with $v_0\le\psi$ and
\[
v_0(z)=\mathrm{sup}\big\{v(z):v\in\mathrm{SH}_1(\C),\, v\le \psi\big\}.
\]

\smallskip

\noindent {\rm (b)}\, A function $v_0$ 
is said to solve the global obstacle problem
with thin obstacle $g$ in $H^{\frac12}(\Gamma)$
and full obstacle $\psi\in H^1_{\mathrm{loc}}(\C)$
if $v_0\in\mathrm{SH}_1(\C)$, $v_0\le \psi$, $v_0\vert_{\Gamma}\le g$ and
\[
v_0(z)=
\mathrm{sup}\big\{v(z):v\in\mathrm{SH}_1(\C),\;\, 
v\le \psi,\;\text{and }\,v\big\vert_{\Gamma}\le g\big\}.
\]
\end{defn}

\begin{rem}
In part {\rm (a)} of Definitions~\ref{defn:obst-bdd} and \ref{defn:obst-unb}, 
the solution $v_0$ is known to be 
as regular as the obstacle $\psi$, up to order $C^{1,1}$. 
Provided that $\psi$ is smooth enough, the function $v_0$ solves the PDE
\[
\Delta v_0=\chi_{\calS}\Delta \psi,
\]
in the distributional sense
for some compact set $\calS$. Since this set is a priori unknown, 
its boundary is known as a {\em free boundary}. 
The set $\calS$ is precisely the {\em coincidence
set}
\begin{equation}\label{eq:coincidence}
\calS=\big\{z:v_0(z)=\psi(z)\big\}.
\end{equation}
In part {\rm (b)} of the same definitions, 
the same holds away from $\Gamma$.
In addition to being supported on $\calS$, 
the Laplacian $\Delta v_0$ will contain
a {\em singular part}, supported on $\Gamma$ 
(see Proposition~\ref{prop:reg} below).
\end{rem}

\begin{rem}\label{rem:monotone-coincidence}
In the context of any of the above obstacle problems 
(Definitions~\ref{defn:obst-bdd} and \ref{defn:obst-unb}),
if the thin obstacle $g$ is replaced by another obstacle $\tilde{g}$
with $\tilde{g}<g$, then the corresponding solution $\tilde{v}_0$ satisfies 
$\tilde{v}_0\le v_0$. As a consequence, the coincidence set with the fixed full
obstacle decreases. Similarly, if the full obstacle $\psi$ is replaced by 
$\tilde{\psi}<\psi$, then the coincidence set with the fixed thin obstacle 
$g$ on $\Gamma$ shrinks.
\end{rem}

\begin{rem}\label{rem:obst-defn}
In case of obstacle problems on a bounded domain 
(Definition~\ref{defn:obst-bdd}), 
the solution is in fact known to be given as the upper
envelope 
\[
v_0(z)=\sup\big\{v(z):v \in\calK\cap\mathrm{SH}(D)\big\}, 
\]
where $\mathrm{SH}(D)$ denotes
the class of subharmonic functions on $D$ 
(see \cite[Theorem~6.2, Ch. II]{Kinderlehrer-Stampaccia}).
This observation puts this problem on the same 
ground as those in Definition~\ref{defn:obst-unb}.
When available, the energy minimization point of view is 
often advantageous, since tools from functional analysis
are more readily available.
\end{rem}

The upper envelope $u_0$ of a family of subharmonic 
functions may fail to be upper semi-continuous.
It is then convenient to define its upper 
semi-continuous regularization $u_0^*(z)$. 
It is known that $u_0=u_0^*$ outside a polar
set, i.e. a set of vanishing logarithmic capacity (see
\cite[pp. 24-25]{SaffTotik}).
In view of the Brelot--Cartan 
theorem (\cite[Ch. II.2]{SaffTotik}), 
the function $u_0^*$ is a subharmonic function on the domain $D$
of $u_0$.  As such, the F. Riesz Theorem 
(\cite[Ch. II.3]{SaffTotik}) yields the existence 
of a measure $\mu_0$, called the Riesz mass of $u_0^*$ 
on $D$ and a harmonic function $h$ such that
\[
u_0^*(z)=U^{\mu_0}(z)+h(z),\qquad z\in D.
\]
Abusing the notation slightly, will refer to $\mu_0$ as 
the Riesz mass of $u_0$ as well.

\subsection{Classical regularity and stability for the obstacle problem}
Next, we will need some results pertaining to the regularity 
of the coincidence set $\calS$ for the obstacle problem. 
The following result was obtained by 
Sakai, see \cite{Sakai}, Theorem~1.1.

\begin{thm}\label{thm:sakai}
Denote by $\psi$ a strictly subharmonic 
real-analytic function in $\D$, and assume that
$u_0\in C^1(\D)$ solves
\[
\Delta u_0=0\qquad \text{in }\: \D\setminus\calS,
\]
where $\calS$ is the coincidence set $\calS=\{z:u_0(z)=\psi(z)\}$,
and that $0\in\partial\calS$.
Then there exists a positive number $\epsilon>0$
such that either
\begin{enumerate}[(i)]
\item $0$ is a regular boundary point of $\calS$ 
in the sense that $\calS\cap \D(0,\epsilon)$
is a simply connected domain with non-empty interior, 
and $\partial\calS\cap \D(0,\epsilon)$ is an analytic arc.

\item the set\, $\calS\cap \D(0,\epsilon)$ is open, 
consisting of either one (I) or 
two (II) simply connected components, 
and $\partial\calS\cap \D(0,\epsilon)$
consists of two analytic arcs which terminate 
at the origin in a cusp (case I), 
or pass through and are tangent at $0$
forming a double point (case II).

\item the origin is a degenerate point, in 
the sense that $\calS\cap\D(0,\epsilon)$ is
either an isolated point or an analytic arc.
\end{enumerate}
\end{thm}

\begin{figure}
\includegraphics[width=.4\textwidth]{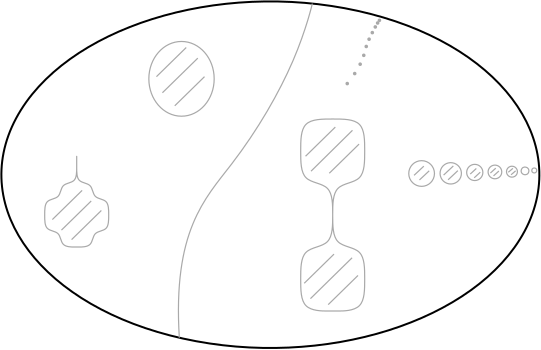}
\caption{The free boundary in the obstacle problem with analytic obstacle, 
showing regular and singular boundary points, 
and possible degeneracies near the boundary.}
\label{fig:Sakai}
\end{figure}

For a schematic illustration of the different free 
boundary types allowed by Sakai's theorem, see Figure~\ref{fig:Sakai}.
In particular, there are only three types of singular boundary points: cusps, double points
and sets with non-empty interior. The latter consists of 
isolated points or a analytic curves, which has to terminate at the 
boundary and be loop-free (cross-cuts). In addition, 
accumulation of isolated points or of components
of $\calI$ with positive area can only occur at the boundary.

In view of Proposition~\ref{prop:reg} below, Sakai's theorem applies
for the solution $v_0$ to the mixed obstacle problem for 
$\calK_{\psi,g}$ for free boundary points away from $\Gamma$, 
wherever $\psi$ is real-analytic with strictly positive Laplacian.
In particular, if $\G$ is simply connected and the interior coincidence set
$\calI=\calS\cap\overline{\G}$ has zero area, then $\partial\calI=\calI$ 
is a union of isolated points which accumulate 
only on $\partial\G$, as well as
real-analytic curves (cross-cuts) in $\G$. 
If we moreover know that $\partial\calI$ is separated from 
$\partial\G$, then $\calI$ can only be a finite 
collection of points. We remark that the finiteness of
$\calI$ may be deduced from the earlier regularity theorem of 
Caffarelli and Rivi{\`e}re \cite{CaffarelliRiviere},
but we will find the other conclusions 
of Sakai's theorem useful later on.

Next, we need an elementary approximation 
property for obstacle problems
on bounded domains. The following is 
Theorem 3.78 of \cite{HKM}.
When formulating results for a general obstacle problem,
we agree to tacitly assume that the sets $\calK$ involved in
their definitions are non-empty.

\begin{prop}
  \label{prop:approx} Assume that $\psi_t\in H^1(D)$ 
  is a convergent sequence, which approximates the limit 
  $\psi\in H^1(D)$ from below, and denote by $v_t$ the 
  solutions to the obstacle problem for $\calK_{\psi_t}$.
  Then $v_t$ converges in $H^1(D)$ to the solution $v$ to the
  obstacle problem for $\calK_\psi$.
\end{prop}

The following stability result will be useful.

\begin{prop}
\label{prop:L-infty}
Let $g\in H^{\frac12}(\Gamma)$ and assume that $h\in C(\Gamma)$.
Let $g_\epsilon=g+\epsilon h$ and let $v_\epsilon$ denote the solution
to the obstacle problem for 
$\calK_{\psi,g_\epsilon}$, where $\psi\in C^{1,1}(D)$.
Then we have the linear $L^\infty$-stability bound
\begin{equation}
\label{eq:L-infty-stability}
\lVert v_\epsilon-v_0\rVert_{L^\infty(D)}
\lesssim \epsilon \lVert h\rVert_{\infty},\qquad \text{as }\;\epsilon\to 0.
\end{equation}
\end{prop}

\begin{proof}
We define an auxiliary function $H$ as the solution to 
the thin obstacle problem on $D$ with thin obstacle $-\lVert h\rVert_\infty$
on $\Gamma$,
vanishing Dirichlet boundary condition on $\partial D$ and
trivial full obstacle which equals $+\infty$ throughout $D$. Notice that $H$
is homogeneous with respect to rescaling of $h$.

We claim that 
\[
v_0+\epsilon H\le v_\epsilon\le v_0-\epsilon H\qquad \text{on }D. 
\]
The lower bound is evident in view of the fact that
\[
v_0+\epsilon H\le g -\epsilon \lVert h\rVert_\infty\le g+\epsilon h\qquad\text{on }\,\Gamma 
\]
so that $v_0+\epsilon H\in \calK_{\psi,g_\epsilon}\cap\mathrm{SH}(D)$
in view of the bound $H\le 0$ on $D$.
Similarly, we have
\[
v_\epsilon +\epsilon H\le g+\epsilon h-\epsilon\lVert h\rVert_\infty\le g\qquad \text{on }\,\Gamma,
\]
so $v_\epsilon+\epsilon H\in \calK_{\psi,g}\cap\mathrm{SH}(D)$.
But then it follows that
$v_\epsilon\le v_0-\epsilon H$ on $D$,
so we get the two-sided bound
\[
-\epsilon \lVert H\rVert_\infty\le v_\epsilon-v_0
\le \epsilon \lVert H\rVert_\infty
\]
on $D$. As the supremum norm of $H$ only depends on
the curve $\Gamma$ and the domain $D$, and linearly on 
the number $\lVert h\rVert_{L^\infty(\Gamma)}$,
this completes the proof.
\end{proof}

Lastly, we need a restriction property (see p. 61 in \cite{HKM}).

\begin{prop}
\label{prop:localize}
Denote by $\Omega$ a subdomain of $D$ with Lipschitz boundary, 
let $v_0$ denote the
solution to the obstacle problem for 
$\calK_{\psi}^{f}(D)$ and set $g=v_0\vert_{\partial \Omega}$.
Then $v_0\vert_{\Omega}$ solves the obstacle 
problem for $\calK_{\psi}^{g}(\Omega)$.
\end{prop}

\section{Perturbation theory
for the mixed obstacle problem}
\label{s:perturb-stab}
\subsection{Perturbations of the thin obstacle: a preliminary view}
The topic of the following 
section could be of independent interest,
and therefore we will work in slightly greater generality than
needed for the applications we have in mind.
Denote by $\Gamma$ a simple Lipschitz curve in $\C$,
which lies in a bounded domain $D$ with Lipschitz boundary. 
Denote by $\psi\in H^1(D)$ a full 
obstacle function (below we will impose additional regularity 
on $\psi$), and let $f\in H^{\frac12}(\partial D)$ 
be a boundary datum with $f\le \psi\vert_{\partial D}$. 
Let moreover $g\in H^{\frac12}(\Gamma)$ be 
a thin obstacle with $g\le\psi\vert_{\Gamma}$, 
where we remind the reader once again that inequalities in Sobolev
spaces are understood in the almost everywhere sense, unless specifically
stated otherwise.

We recall the notation $\dir(u)$ for the Dirichlet energy
\[
\dir(u)=\int_D|\nabla u|^2\,\diffA,\qquad u\in H^1(D),
\]
where $\diffA$ denotes planar area measure, and
consider the solution $v_0$ to the  obstacle problem
for $\calK_{\psi,g}^f$, that is
\begin{multline}\label{eq:obst-thin}
\inf_{u\in \calK_{\psi,g}^{f}}\dir(u),
\\
\text{where }\;\;\calK_{\psi,g}^{f}=
\big\{u\in H^1(D)\,:\,u\vert_{\partial D}=
f,\,u\le \psi\text{ on } D, \text{ and } 
u\vert_\Gamma\le g\big\}.
\end{multline}
Recall that, in view of \cite[Theorem 6.4, Ch II]{Kinderlehrer-Stampaccia}, 
the solution $v_0$ is just the upper envelope of the class 
$\calK_{\psi,g}^{f}\cap\mathrm{SH}(D)$.

Let now $h\in H^{\frac12}(\Gamma)$, define for $\epsilon>0$ 
a family of thin obstacles $g_\epsilon$ on 
$\Gamma$ by $g_\epsilon=g+\epsilon h$ and let $v_\epsilon$ be
the solution to the obstacle problem for $\calK_{\psi,g_\epsilon}^f$.
We are interested in the stability of this 
obstacle problem as $\epsilon$ varies. 
In particular,
we seek a \emph{variational formula} for the Dirichlet energy of the solution. 

Assuming that $h$ and $g$ are sufficiently smooth, 
and that $g-\psi\vert_{\Gamma}$ is negative and bounded away from zero,
one can prove $C^{1,\beta}$-stability for $v_\epsilon$ 
valid up to the curve $\Gamma$. Moreover, we recall the
linear $L^\infty$-stability with respect 
to the data on the curve $\Gamma$ of Proposition~\ref{prop:L-infty}.
Using these facts, it is rather straightforward to derive a variational 
formula for the Dirichlet energy
\begin{equation}\label{eq:variational}
\dir(v_\epsilon)=\dir(v_0)
-4\pi\epsilon\int_\Gamma h\,\diff\calN_\Gamma(v_0)+\ordo(\epsilon)
\end{equation}
as $\epsilon\to0$.
We omit the details since we will
show something more general below.
We return to the variational formula after discussing a preliminary 
result on the regularity of the solution of
the extremal problem in \eqref{eq:obst-thin}.

\subsection{Structure of the solution to the mixed obstacle problem}
If the obstacle $\psi$ in Definition~\ref{defn:obst-bdd} is regular enough, then
the Riesz mass $\mu_0$ of the solution splits into two components. 
One is singular with
respect to area measure and lives on the curve $\Gamma$, while the other 
is continuous with respect to area measure and is supported on the set 
$\calS=\{z\in D:v_0(z)=\psi(z)\}$. This is the content of the next result.
It is likely well-known to experts, 
but we include a proof of it for completeness.

\newpage

\begin{prop}\label{prop:reg}
The Riesz mass $\mu_0$ of the solution $v_0$ to the obstacle problem
\[
\inf_{v\in\calK}\dir(v),\quad \calK=\big\{v\in H^1(D)
:\,v\vert_\Gamma\le g,\;v\le \psi,\,v\vert_{\partial D}=f\big\}
\]
with thin obstacle $g\in H^{\frac12}(\Gamma)$, 
boundary datum $f\in H^{\frac12}(\partial D)$ 
and full obstacle $\psi\in C^{1,1}(\overline{D})$
on a Lipschitz domain $D$ containing the Lipschitz curve $\Gamma$ enjoys a decomposition
\[
\diff\mu_0=\diff\mu_0^\s+\frac{1}{2\pi}\Delta \psi\chi_\calS\diffA,
\]
where $\mu_0^\s\in H^{-\frac12}(\Gamma)$, 
and the solution
$v_0$ belongs to $C^{1,\beta}(D\setminus\Gamma)$ for any $\beta<1$. In addition,
Green's formula
\begin{equation}\label{eq:Green}
\int_D\nabla u\cdot\nabla v_0\,\diffA=
-2\pi \int_\Gamma u\,\diff\mu_0^\s - \int_{\calS}u\,\Delta \psi\,\diffA
\end{equation}
holds for any $u\in H^{1}_0(D)$. 
\end{prop}

\begin{proof}
The proof splits into two steps: first we obtain the result under the condition that
$g\le \psi\vert_\Gamma-\delta$ for some $\delta>0$, 
then this condition is relaxed by an approximation argument.
\\

\noindent {\sc Step 1.}\,
Assume first that $g\le \psi\vert_{\Gamma}-\delta$ for some $\delta>0$. 
It is clear that $\calS$ is separated from $\Gamma$, 
by upper continuity of $v_0-\psi$, since otherwise 
there would exist a sequence of points $z_j$ converging to 
$\Gamma$, along which $u-\psi$ vanishes, which 
would force $v_0-\psi$ to vanish on $\Gamma$.
As a consequence, $v_0$ is harmonic in a region $V\setminus \Gamma$,
where $V$ is a neighbourhood of $\Gamma$.

In view of Proposition~\ref{prop:localize} we may think 
of $v_0$ as the solution
to the classical obstacle problem in each component of 
$D\setminus \Gamma$ separately, 
with boundary data $g_0=v_0\vert_{\Gamma}$. 
Since the coincidence set $\calS$ remains bounded away from
the fixed boundary $\Gamma$, it follows from 
classical theory, see e.g.\  \cite[Theorem~2.3]{CK},
that the solution is $C^{1,1}$ smooth away from $\Gamma$
and that $\Delta v_0=0$ outside the coincidence set 
$\calS_0$.
In fact, the distributional 
Laplacian away from $\Gamma$ takes the form
$\chi_{D\setminus \Gamma}\mu_0=\tfrac{1}{2\pi}\chi_{\calS}\Delta \psi$,
see \cite[Theorem~3.10]{HM}.

We may then define a two-sided Neumann trace
on $\Gamma$. By a standard convolution argument 
there exists a sequence $(v_j)_j$
of smooth functions whose gradients are continuous 
up to the boundary of $D\setminus\Gamma$,
such that $v_j\to v_0$ in $H^1(D)$. 
In addition we have $\Delta v_j\to \Delta v_0$ 
pointwise a.e. and in $L^2$.
But then by the standard Green's formula we have
\begin{equation}
\int_{D}\nabla u\cdot \nabla v_j\,\diffA
=-2\pi\int_{\Gamma} u\,\diff\mu_j^{\s}-
\int_{D}u\Delta v_j\,\diffA,
\end{equation} 
where $\mu_j^\s$ is the Riesz measure of $v_j$ on $\Gamma$. From the convergences 
$v_j\to v_0$ in $H^1(D)$ and $\Delta v_j\to \Delta v_0$ in $L^2(D)$ it follows that 
$\mu_j^\s$ converges to a bounded functional on $H^{\frac12}(\Gamma)$,
and we find that the limiting Green's formula for $v_0$ holds.

In summary, we have found that $\mu_0$ is supported on $\Gamma\cup\calS$ and since 
$\calS$ is relatively compact in $D\setminus\Gamma$, 
the Riesz mass decomposes as
\[
\diff\mu_0=\diff\mu_0^\s+\frac{1}{2\pi}\Delta \psi\chi_{\calS}\,\diffA
\]
and that Green's formula holds.
The singular measure (with respect to $\diffA$)
$\mu_0^\s$ is the Neumann trace $\calN_\Gamma(v_0)$ and
belongs to $H^{-\frac12}(\Gamma)$.

\noindent {\sc Step 2.}\:
We now consider a general thin obstacle $g$, which does
not necessarily lie strictly below the full obstacle. 
We will approximate $g$ and $f$ from below, and show that 
we have a strong enough convergence
of the corresponding solutions in order to 
reach the desired conclusions for the limiting object.
We let $g_\epsilon=g-\epsilon$ and $f_\epsilon=f-\epsilon$, 
and denote by $v_\epsilon$ the corresponding solutions. 
For each $v_\epsilon$ we may apply the above argument.
By the approximation property of Proposition~\ref{prop:approx}, 
we know that $v_\epsilon\to v_0$ in $H^1(D)$. 
But then it follows that
$\mu_\epsilon\to\mu_0$. 
Let $\calS_\epsilon$ denote the coincidence set for $v_\epsilon$
with the full obstacle $\psi$. The set $\calS_\epsilon$ increases
monotonically as $\epsilon$ tends to zero (see Remark~\ref{rem:monotone-coincidence}), 
and in view of the regularity assumption we may conclude that 
$\mu_\epsilon^\c=\frac{1}{2\pi}\chi_{\calS_\epsilon}\Delta \psi$
is convergent in $L^2$ (e.g. by monotone convergence) towards the
limit function $\frac{1}{2\pi}\chi_{\calS^\star}\Delta \psi$,
where $\calS^\star=\cup_\epsilon\calS_\epsilon$. But then
it also follows from Green's formula for $v_\epsilon$
that $\mu_\epsilon^\s$ is convergent as a 
functional on $H^{\frac12}(\Gamma)$ towards $\mu_0^\s\in H^{-\frac12}(\Gamma)$.
In summary, we have the convergence
\[
\int_D\nabla u\cdot\nabla v_0\,\diffA=
-\lim_{\epsilon\to 0}\Big(2\pi \int_\Gamma u\,\diff\mu_\epsilon^s
+\int_{\calS_\epsilon}u\,\Delta \psi\,\diffA\Big)=
-2\pi \int_\Gamma u\,\diff\mu_0^\s - \int_{\calS^\star} u\,\Delta \psi\,\diffA
\]
It only remains to conclude that $\calS^\star=\calS$ up to the removal
of a null set.
But this follows from the fact that 
$v_0\in\mathrm{SH}(D)$, that $v_0\vert_{D\setminus\Gamma}$ is continuous,
the linear stability result of Proposition~\ref{prop:L-infty} together with 
the stability for the coincidence set under perturbations, see e.g. 
\cite{Blank}. Indeed, outside any fixed neighbourhood $\Gamma_\delta$ of $\Gamma$, 
the above results combine to say that 
$|(\calS\setminus\calS_\epsilon)\cap\Gamma_\delta^c|=\Ordo(\epsilon)$, 
so that $|(\calS\setminus\calS^\star)\cap\Gamma_\delta^c|=0$.
This completes the proof.

As a consequence of the smoothness of $\psi$, 
we saw that the solution $v_0$ solves the PDE 
$\Delta v_0=\Delta \psi\chi_{\calS}
\in L^\infty$ in $D\setminus \Gamma$. 
In view of \cite[Proposition~1.1]{Zhang} this implies that 
$v_0\in W^{2,p}(D\setminus\Gamma)$ for any $p<\infty$, so by the classical
Sobolev embedding theorem it follows that $v_0\in C^{1,\beta}$ for any $\beta<1$.
\end{proof}

\subsection{A few subtleties}
\label{ss:subtleties}
In this section, we verify two seemingly obvious facts, 
which are slightly subtle due to the
low regularity of the function 
$g\in H^{\frac12}(\Gamma)$.

\begin{prop}\label{prop:subtleties1}
  Let $f\in H^{\frac12}(\Gamma)$ and assume that $f\ge 0$ 
  in the sense of traces, i.e. there exists $u\in H^1(D)$ with $u\ge 0$ 
  Lebesgue a.e. in $D$, with $u\vert_\Gamma=f$. Then, for any positive measure
  $\mu\in H^{-\frac12}(\Gamma)$ we have $\int f\diff\mu\ge 0$.
\end{prop}

\begin{proof}
A standard dilation and mollification trick 
shows that any non-negative $u\in H^1(D)$ can be approximated
by non-negative smooth functions in $H^1(D)$. 

We denote by $u$ the non-negative $H^1_0(D)$-function 
which is harmonic in $D\setminus \Gamma$ while $u\vert_\Gamma=f$, 
and apply this approximation trick to get
$u_j\in C^\infty(D)$ with $u_j\to u$ in $H^1(D)$ and $u_j\ge 0$. 
Denote by $v$ 
an element of $H^1_0(D)$ with Riesz mass $\mu$. Then
\[
  \int_\Gamma f\,\diff\mu=\frac{1}{2\pi} \int_D\nabla u\cdot\nabla v\,\diffA=
  \lim_{j\to\infty} \frac{1}{2\pi}\int_D\nabla u_j\cdot\nabla v\,\diffA
  =\lim_{j\to\infty} \int_\Gamma u_j\,\diff\mu\ge 0.
\]
This completes the proof.
\end{proof}

If the obstacles $\psi$ and $g$ are smooth, and $u$ denotes the solution
to the obstacle problem for $\calK_{\psi,g}$, 
then $(u-g)\diff\calN_\Gamma=0$ pointwise on $\Gamma$, and $(u-\psi)\Delta u=0$
on $D\setminus \Gamma$. The latter survives in our setting, but we prefer 
not to attempt to interpret the former in a pointwise sense. However, we have
the following averaged version.

\begin{prop}\label{prop:thin-coincidence-spt}
  Assume that $v_0$ solves the mixed obstacle problem with thin 
  obstacle $g$ on $\Gamma$
  and full obstacle $\psi\in C^{1,1}(\overline{D})$. 
  Then, if $\mu_0^\s$ denotes the 
  singular part of the Riesz mass of $v_0$, 
  we have 
  $\int_\Gamma (v_0-g)\diff\mu_0^\s=0$.
\end{prop}

\begin{proof}
  Denote by $\eta$ the $H^1(D)$-obstacle discussed in Remark~\ref{rem:one-obst} 
  \[
   \eta= \min\{G,\psi\},
 \]
 where $G$ is the harmonic function in $D\setminus \Gamma$ which 
 equals $f$ on $\partial D$ and $g$ on $\Gamma$.
On the one hand, $\eta\in \calK_\eta=\calK_{\psi,g,f}$ 
so in view of Remark~\ref{rem:variational-basic} we have
\[
  \int_D\nabla (v_0-\eta)\cdot\nabla v_0\,\diffA\le 0,
\]
by the variational inequalities for minimization of the Dirichlet energy 
(obtained as usual by noticing that $u_t=(1-t)v_0+t\eta$ 
is a competitor for the minimization for any $t\ge 0$,
and then expanding the energy and letting $t\to 0^+$). 
On the other hand we have $v_0\le G$ by the maximum
principle, and $v_0\le \psi$ by definition, so it holds that $v_0\le \eta$,
from which we conclude that
\[
  \int \nabla (v_0-\eta)\cdot \nabla u\,\diffA
  =-2\pi \int (v_0-\eta)\diff\mu_0\ge 0.
\]
Hence, $\langle v_0-g,v_0\rangle =2\pi \int_{D}(v_0-\eta)\diff\mu_0=0$. 
By Proposition~\ref{prop:reg}, Green's formula applies, so by \eqref{eq:Green} we have
\begin{multline*}
  \int_\Gamma (v_0-g)\diff\mu_0^\s \\ = 
  \frac{1}{2\pi}\int \nabla (v_0-\eta)\cdot \nabla v_0\,\diffA
  - \frac{1}{2\pi}\int_{\calS_0}(v_0-\eta)\Delta \psi\diffA
  =- \frac{1}{2\pi}\int_{\calS_0}(v_0-\eta)\Delta \psi\diffA
\end{multline*}
where $\calS_0$ denotes the coincidence set $\{v_0=\psi\}$.
But since $v_0\le \eta\le \psi$ we find that
$v_0=\psi=\eta$ on the support of $\Delta v_0$,
so the last integral vanishes and the claim follows.
\end{proof}

\subsection{The variational formula}
We begin with a preliminary stability estimate for the
singular part of the Riesz mass of $v_\epsilon$.

\begin{lem}
\label{lem:prel-stab}
In the setting of Proposition~\ref{prop:reg}, it holds that
\[
\int_\Gamma h\,\diff\calN_\Gamma(v_0-v_\epsilon)
=\ordo(1)+\Ordo\big(\dir(v_\epsilon-v_0)^{\frac12}\big)\qquad \text{as }\,\epsilon\to0.
\]
\end{lem}

\begin{proof}
Let us look at the integral
\[
\int_{D}\nabla (v_0-\psi)\cdot\nabla H\,\diffA,
\]
where $H$ is the harmonic extension to $D\setminus \Gamma$ 
of the boundary values given by $h$ on $\Gamma$
and by $0$ on $\partial D$. 
In terms of the Neumann jump $\calN_\Gamma(v_0)$, 
Green's formula \eqref{eq:Green} reads
\[
\int_{D}\nabla (v_0-\psi)\cdot\nabla H\,\diffA=
-\int_{D\setminus \calS}H\Delta \psi
-2\pi\int_{\Gamma}h\,\diff\calN_\Gamma(v_0).
\]
Subtracting the same calculation for $v_\epsilon$, we see that
\begin{equation}\label{eq:h-v-eps-v}
2\pi\int_{\Gamma}h\,\diff\calN_\Gamma(v_0-v_\epsilon)=
\int_D\nabla (v_\epsilon-v_0)\nabla H\,\diffA
-\int_{\calS_\epsilon\Delta\calS_0}
\mathrm{sgn}(v_\epsilon-v_0)H\Delta\psi\,\diffA,
\end{equation}
where $A\Delta B$ denotes the symmetric 
difference $(A\setminus B)\cup (B\setminus A)$.
Invoking the Cauchy-Schwarz inequality, we find
\[
2\pi \int_{\Gamma}h\,\diff\calN_\Gamma(v_0-v_\epsilon)=
-\int_{\calS_\epsilon\Delta\calS_0}
\mathrm{sgn}(v_\epsilon-v_0)H\Delta\psi\,\diffA+
\Ordo(\dir(v_\epsilon-v_0)^{\frac12}).
\]
It remains to show that the symmetric difference $\calS_\epsilon\Delta\calS_0$
has $|\calS_\epsilon\Delta\calS_0|=\ordo(1)$. 
Indeed, since $H\in L^2$ and $\Delta\psi\in L^\infty$, 
the claim would then follow by an 
application of the Cauchy-Schwarz inequality.
To study the symmetric difference, we note that $\calS_\epsilon$ and $\calS_0$
are the coincidence sets for a classical obstacle 
problem on the domain $D\setminus \Gamma$
where the full obstacle is given by $\psi$, and the 
boundary data on $\Gamma$ given by the functions $v_\epsilon\vert_\Gamma$ and
$v_0\vert_\Gamma$, respectively, which by Proposition~\ref{prop:L-infty} 
differ by $\Ordo(\epsilon)$.
We wish to deduce the result by invoking a 
result of Blank and LeCrone (see below),
but their theorem requires that the boundary data is continuous.
This issue may be circumvented, as follows.
First, notice that it is enough to prove that
for any given $\eta>0$, we have
\begin{equation}\label{eq:eta-stab-measure}
|(\calS_\epsilon\Delta\calS_0)
\cap\{z\in D:\mathrm{d}(z,\Gamma)\ge\eta\}|=\Ordo(\epsilon)
\end{equation}
as $\epsilon\to 0$ with $\eta$ held fixed.
Indeed, if we assume that \eqref{eq:eta-stab-measure} holds, it follows that
\[
\limsup_{\epsilon\to 0}|\calS_\epsilon\Delta\calS_0|\le 
|\{z\in D: \diff(z,\Gamma)\le \eta\}| \le C_0\eta,
\]
where $C_0$ is a uniform constant that depends on $\Gamma$. 
But $\eta$ was arbitrary,
so it follows that $|\calS_\epsilon\Delta\calS_0|=\ordo(1)$.

To see why the former bound \eqref{eq:eta-stab-measure} 
holds, let $\Gamma_\eta^\pm,$ denote two smooth Jordan curves
between $\Gamma$ and the sets 
\[
D^\pm_\eta\coloneqq \{z\in D^\pm: \mathrm{d}(z,\Gamma)\ge\eta\},
\]
respectively,
where we recall that $D^\pm$ denote the two 
components of $D\setminus \Gamma$.
We require that the two curves $\Gamma_\eta^\pm$ lie at a positive distance
from $\Gamma$ as well as from the corresponding set $D^\pm_\eta$.
The restrictions of $v_\epsilon$ and $v_0$ to the curves $\Gamma_\eta^\pm$ 
are $C^1$-smooth
by Proposition~\ref{prop:reg}, and by the linear 
stability bound \eqref{eq:L-infty-stability}
the difference $v_0-v_\epsilon$ is of order $\Ordo(\epsilon)$ on $D$,
so in particular on each of the two curves.
Hence, the assumptions of Blank and LeCrone's results in 
\cite{Blank} are satisfied,
so applying their result twice, once for each domain $D_\eta^\pm$, 
we get that \eqref{eq:eta-stab-measure} holds.
\end{proof}

The following stability estimate for 
the energy $\dir(v_\epsilon)$ as $\epsilon\to0$
is a key ingredient in what follows.

\begin{lem}\label{lem:stab}
In the setting of Proposition~\ref{prop:reg}, we have the stability 
\[
\dir(v_\epsilon-v_0)=\ordo(\epsilon)
\]
as $\epsilon\to 0$. 
\end{lem}

\begin{proof}
Recall that for $\epsilon\ge 0$, the function $v_\epsilon$ solves the 
obstacle problem for the class $\calK_{\psi,g_\epsilon}^f$, 
where $g_\epsilon=g+\epsilon h$
and $g,h \in H^{\frac12}(\Gamma)$.
By Green's formula \eqref{eq:Green} applied to the difference $v_\epsilon-v_0$
we have
\[
\dir(v_\epsilon-v_0)=-2\pi\int_\Gamma (v_\epsilon-v_0)\diff\calN_\Gamma(v_\epsilon-v_0)
-\int_{D\setminus\Gamma} (v_\epsilon-v_0)\Delta(v_\epsilon-v_0)\diffA,
\]
where we write $D\setminus \Gamma$ for the domain of integration in the area
integral simply to indicate that Green's formula was 
applied in this domain.
We add and subtract $\epsilon h$ inside the boundary integral, and find
\begin{multline*}
\dir(v_\epsilon-v_0)
=-2\pi\int_\Gamma(v_\epsilon-v_0-\epsilon h)\diff\calN_\Gamma(v_\epsilon-v_0)
-\int_{D\setminus \Gamma} (v_\epsilon-v_0)\Delta(v_\epsilon-v_0)\diffA
\\
+2\pi \epsilon \int_\Gamma h\,\diff\calN_\Gamma(v_0-v_\epsilon).
\end{multline*}
We next claim that
\begin{equation}
\label{eq:term1-neg}
-\int_\Gamma (v_\epsilon-v_0-\epsilon h)\diff\calN_\Gamma(v_\epsilon-v_0)\le 0
\end{equation}
and that
\begin{equation}
\label{eq:term2-neg}
-\int_{D\setminus \Gamma} (v_\epsilon-v_0)\Delta(v_\epsilon-v_0)\diffA\le 0.
\end{equation}
We first argue heuristically.
On the support of the (positive) measure $\calN_\Gamma(v_\epsilon)$, we have
$v_\epsilon=g+\epsilon h$, while $-v_0\ge -g$ holds on the entire curve $\Gamma$, 
so that 
$v_\epsilon-v_0-\epsilon h\ge 0$ holds on 
$\mathrm{supp}(\calN_\Gamma(v_\epsilon))$. It would follow that
\begin{equation}\label{eq:term-1-1-neg}
-\int_\Gamma (v_\epsilon-v_0-\epsilon h)\diff\calN_\Gamma(v_\epsilon)\le 0.
\end{equation}
Moreover, on the support of $\diff\calN_\Gamma(v_0)$,
we have $v_\epsilon-v_0-\epsilon h\le 0$, so that
\begin{equation}\label{eq:term-1-2-neg}
\int_\Gamma(v_\epsilon-v_0-\epsilon h)\diff\calN_\Gamma(v_0)\le 0.
\end{equation}
Adding these up would give \eqref{eq:term1-neg}. 

Since the function $g$ is possibly very irregular, 
we prefer not to rely on pointwise equalities 
on the supports $\mathrm{supp}(\calN_\Gamma(v_\epsilon))\subset\Gamma$. 
We may justify \eqref{eq:term-1-1-neg} and $\eqref{eq:term-1-2-neg}$ 
by reinterpreting the above heuristics in a weak sense, with the aid of
\S~\ref{ss:subtleties}.
Indeed, we illustrate this for the term \eqref{eq:term-1-1-neg}.
We rewrite
\[
v_\epsilon-v_0-\epsilon h=(v_\epsilon-g-\epsilon h)-(v_0-g),
\]
and notice that in view of
Proposition~\ref{prop:thin-coincidence-spt}, we have
\[
\int (v_\epsilon-g-\epsilon h)\diff\calN_\Gamma(v_\epsilon)=0,
\]
from which we conclude that
\[
-\int (v_\epsilon-v_0-\epsilon h)\diff\calN_\Gamma(v_\epsilon)=
-\int (g-v_0)\diff\calN_\Gamma(v_\epsilon)\le 0,
\]
where the last inequality uses Proposition~\ref{prop:subtleties1}
and the bound $v_0\le g$ on $\Gamma$ (a.e.).
As a consequence, the weak interpretation of the initial 
inequality yields the desired claim.

The claim \eqref{eq:term2-neg} is proven similarly, 
but the argument can be understood in the pointwise sense,
in view of the regularity imposed on the full obstacle $\psi$:
outside the union $\calS_\epsilon\cup\calS_0$, both functions are harmonic and 
so do not contribute to the integral.
On the coincidence set $\calS_\epsilon$ (which contains 
$\mathrm{supp}(\Delta v_\epsilon)$), we have $v_\epsilon-v_0\ge 0$, so
\[
-\int_{D\setminus \Gamma} (v_\epsilon-v_0)\Delta v_\epsilon\,\diffA\le 0.
\]
Similarly, on $\calS_0$ we have $v_\epsilon-v_0\le 0$, so
\[
\int_{D\setminus \Gamma} (v_\epsilon-v_0)\Delta v_0\,\diffA\le 0.
\]
Adding up these terms gives \eqref{eq:term2-neg}.
In conclusion, we find that
\begin{equation}\label{eq:energy-upper bound}
\dir(v_\epsilon-v_0)\le 2\pi \epsilon\int_\Gamma h\,\diff\calN_\Gamma(v_0-v_\epsilon).
\end{equation}
In view of Lemma~\ref{lem:prel-stab}, we find that
\[
\dir(v_\epsilon-v_0)\le C\epsilon\dir^{\frac12}(v_\epsilon-v_0)+\ordo(\epsilon),
\]
which implies that $\dir(v_\epsilon-v_0)=\ordo(\epsilon)$, which proved the claim.
\end{proof}

\begin{rem}
In case it holds that $g<\psi\vert_\Gamma$, 
so that $\calS_\epsilon$ stays bounded away from
$\Gamma$, we can improve this to an 
optimal stability bound of the form
\[
\dir(v_\epsilon-v_0)=\Ordo(\epsilon^2).
\]
Indeed, the number $\eta$ appearing in the 
proof of Lemma~\ref{lem:prel-stab} can be chosen 
sufficiently small but fixed, that
the whole of the symmetric difference $\calS_\epsilon\Delta\calS_0$ is 
contained in $D_\eta^+\cup D_\eta^-$,
so we find $|\calS_\epsilon\Delta\calS_0|=
\Ordo(\epsilon)$. Inserting
this into \eqref{eq:h-v-eps-v} and applying 
the Cauchy-Schwarz inequality, 
we find that
\[
\dir(v_\epsilon-v_0)\le 
C\epsilon\dir^{\frac12}(v_\epsilon-v_0)
+C'\epsilon^2,
\]
which can be solved to yield the claimed conclusion.
\end{rem}

We can now deduce our key stability lemma.
\begin{cor}\label{cor:variational}
In the setting of Lemma~\ref{lem:stab}, we have the 
variational formula
\[
\dir(v_\epsilon)= \dir(v_0) - 4\pi\epsilon\, \int_\Gamma h\,\diff\calN_\Gamma(v_0)+
\ordo(\epsilon).
\]
\end{cor}

\begin{rem}
If in addition we assume that $g_\epsilon-\psi\vert_{\Gamma}$ 
is negative and bounded away 
from zero for small enough $\epsilon$, the error term
may be improved to $\Ordo(\epsilon^2)$.
\end{rem}

\begin{proof}
As observed above, we by Green's formula \eqref{eq:Green}
\[
\int_{D}\nabla u\cdot\nabla v_\epsilon\,\diffA=
-2\pi \int_\Gamma u\,\diff\calN_\Gamma(v_\epsilon)
-\int_{D\setminus \Gamma} u\Delta v_\epsilon\,\diffA,\qquad
u\in H_0^1(D).
\]
A computation using this yields that
\begin{align}\label{eq:pre-variational0}
\dir(v_\epsilon)=&\dir(v_0)
+2\int_D \nabla (v_\epsilon-v_0)\,\cdot \nabla v_0\,\diffA+\dir(v_\epsilon-v_0)
\\
=&\dir(v_0)-4\pi \int_\Gamma (v_\epsilon-v_0)\diff\calN_\Gamma(v_0)
-2\int_{D\setminus \Gamma}(v_\epsilon-v_0)\Delta v_0\,\diffA+\dir(v_\epsilon-v_0).
\label{eq:pre-variational}
\end{align}
The last term is $\ordo(\epsilon)$ by 
the improved stability bound of Lemma~\ref{lem:stab}. 
The main non-trivial 
contribution to the first variation in 
\eqref{eq:pre-variational} should come from the second term, 
while the third should be small. We notice 
that the second term of \eqref{eq:pre-variational} 
can be renormalized with the function $h$, 
\[
-2\int_\Gamma (v_\epsilon-v_0)\diff\calN_\Gamma(v_0)
=-4\pi\epsilon \int_\Gamma h\,\diff\calN_\Gamma(v_0)
-4\pi\int_\Gamma \big((v_\epsilon-v_0)-\epsilon h)\diff\calN_\Gamma(v_0).
\]
The key point here is that the last integral on the 
right-hand side is negative (so that the last term is positive).
Indeed, repeating the argument in the proof of Lemma~\ref{lem:stab}
we find that where the (positive) 
measure $\calN_\Gamma(v_0)$ is supported, 
$v_0$ reaches up to its obstacle $g$, 
while $v_\epsilon\le g+\epsilon h$, and
hence $v_\epsilon-v_0-\epsilon h\le 0$. Similarly, we have that
\[
-4\pi\int_\Gamma \big((v_\epsilon-v_0)- 
\epsilon h\big)\diff\calN_\Gamma(-v_\epsilon)\ge 0,
\]
so we obtain a 
two-sided 
bound for the integral over $\Gamma$
\begin{multline*}
-4\pi\epsilon\int_\Gamma h\,\diff\calN_\Gamma(v_0)\le 
-4\pi\int_\Gamma (v_\epsilon-v_0)\diff\calN_\Gamma(v_0)
\\
\le -4\pi\epsilon\int_\Gamma h\,\diff\calN_\Gamma(v_0)
-4\pi\int_\Gamma \big((v_\epsilon-v_0)-\epsilon h)
\diff\calN_\Gamma(v_0-v_\epsilon)
\end{multline*}
An analogous comparison shows that where the 
positive measure $\Delta v_0$ is supported, we have
$v_\epsilon-v_0\le 0$. Similarly, on the support 
of the measure $\Delta v_\epsilon$, we have
$v_\epsilon-v_0\ge0$. Hence, we obtain
\[
0\le -2\int_{D\setminus\Gamma}(v_\epsilon-v_0)\Delta v_0\,\diffA\le 
2\int_{D\setminus\Gamma}(v_\epsilon-v_0)\Delta (v_\epsilon-v_0)\diffA.
\]
Adding up everything, we see that
\begin{multline*}
-4\pi\epsilon \int_\Gamma h\,\diff\calN_\Gamma(v_0)\le 
2\int_D \nabla (v_\epsilon-v_0)\cdot\nabla v_0\,\diffA
\le -4\pi \epsilon \int_\Gamma h\,\diff\calN_\Gamma(v_0) 
\\
-4\pi\int_\Gamma \big((v_\epsilon-v_0)-\epsilon h)
\diff\calN_\Gamma(v_0-v_\epsilon)-2\int_{D\setminus\Gamma}
(v_0-v_\epsilon)\Delta (v_\epsilon-v_0)\diffA.
\end{multline*}
By applying Green's formula in reverse we see that 
we may rewrite the right-hand side as
\begin{multline*}
-4\pi \int_\Gamma \big((v_\epsilon-v_0)-\epsilon h)\diff
\calN_\Gamma(v_0-v_\epsilon)-2\int_{D\setminus\Gamma}(v_\epsilon-v_0)
\Delta (v_0-v_\epsilon)\diffA\\
=4\pi \int_\Gamma (v_\epsilon-v_0)\diff\calN_\Gamma(v_\epsilon-v_0)
+2\int_{D\setminus \Gamma}(v_\epsilon-v_0)
\Delta(v_\epsilon-v_0)\diffA - 4\pi \epsilon 
\int_\Gamma h\,\diff\calN_\Gamma(v_\epsilon-v_0)\\
=-2\dir(v_\epsilon-v_0)-4\pi \epsilon 
\int_\Gamma h\,\diff\calN_\Gamma(v_\epsilon-v_0)=\ordo(\epsilon),
\end{multline*}
where the last step uses Lemma~\ref{lem:prel-stab} and Lemma~\ref{lem:stab}.
So, we obtain the inequality
\[
-4\pi \epsilon\int_\Gamma h\,\diff\calN_\Gamma(v_0)\le 
2\int_D\nabla (v_\epsilon-v_0)\cdot \nabla v_0\,\diffA
\le -4\pi \epsilon\int_\Gamma h\,\diff\calN_\Gamma(v_0)
+\ordo(\epsilon).
\]
Using this and Lemma~\ref{lem:stab} in 
\eqref{eq:pre-variational}, it we find that
\[
\dir(v_\epsilon)= 
\dir(v_0)-4\pi \epsilon \int_{\Gamma}h\,\diff\calN_\Gamma(u_0)+\ordo(\epsilon).
\]
This completes the proof of the claim.
\end{proof}

\section{Reformulation of the constrained extremal problem}
\label{s:fund-obst}
\subsection{Constrained extremal measures solve an obstacle problem}
The following result is the starting point for 
our approach to obtaining Theorem~\ref{thm:s-qdom}. 
It serves as a first step in a two-step 
variational characterization of the equilibrium measure.
We recall briefly the extremal problem under consideration.
The functional is given, for $\alpha>0$, by
\begin{equation}\label{def:I-alpha}
I_\alpha(\mu)=-\Sigma(\mu)+2B_\alpha(\mu),
\end{equation}
where
\begin{equation}\label{eq:def-B-alpha}
B_\alpha(\mu)=\sup_{z\in\C}\big(U^\mu(z)-\tfrac{1}{2\alpha}|z|^2\big).
\end{equation}
The minimizer of $I_\alpha$ among all $\mu\in\calM_\G$ - the class of
compactly supported probability measures on $\C$ with
finite logarithmic energy, which give no mass to $\G$ - is denoted by 
$\mu_{\alpha,\G}$. 
In order to keep the notation manageable, 
we allow ourself to denote the measure $\mu_{\alpha,\G}$ 
by $\mu_0$ when there can occur no confusion.

\begin{thm}\label{thm:basic-obst}
There exists a function $g\in H^{\frac12}(\partial\G)$ 
such that $\mu_{\alpha,\G}$ is the Riesz mass of the 
solution $v_0$ to the mixed 
obstacle problem with thin obstacle 
$g$ on $\partial \G$ and global obstacle
$\tfrac{1}{2\alpha}|z|^2$, in the sense of 
Definition~\ref{defn:obst-unb}(b).
\end{thm}
\begin{proof}
To ease notation, set $\mu_0=\mu_{\alpha,\G}$
and $v_0=U^{\mu_{\alpha,\G}}$.
Define the function $w$ by $w=v_0-B_\alpha(\mu_{0})$,
and denote by $g$ its restriction to $\partial \G$.
We moreover consider the solution $u$ to the obstacle problem with full obstacle 
$\tfrac1{2\alpha}|z|^2$ and thin obstacle $g$,
see Definition~\ref{defn:obst-unb}, Part (b).
It is immediate that the upper semicontinuous regularization $u^*$ of $u$ is a subharmonic 
function and that its Riesz mass is a probability measure. Moreover, it may
happen that $u^*\ne u$ only on a polar subset of $\Gamma$. In addition,
since $u^*$ satisfies the sub-mean value inequality and $u^*(z)\le \frac1{2\alpha}|z|^2$ a.e.,
it is easy to see that
$\sup_{z\in\C} \big(u^*(z)-\tfrac{1}{2\alpha}|z|^2\big)=0$.
Abusing notation slightly, we write $u$
for the regularized function.

Observe that
$w$ is harmonic 
in the hole $\G$, so it follows that the set
\[
\calI'=\big\{z\in\overline{\G}:w(z)=\tfrac{|z|^2}{2\alpha}\big\}
\] 
has zero (planar) Lebesgue measure. 
As a consequence of the maximum principle we have $u\le w$ on $\G$, and we 
conclude that also the interior coincidence set 
\[
\calI=\big\{z\in\overline{\G}:u(z)=\tfrac{|z|^2}{2\alpha}\big\}
\]
has zero Lebesgue measure as well. Since $u$ solves the obstacle problem
in $\G$ with obstacle $\tfrac1{2\alpha}|z|^2$, it follows 
(e.g. by Proposition~\ref{prop:reg}) that the Riesz measure
of $u$ is contained in $\calI$. By Grishin's lemma \cite{Grishin} 
(see also \cite{Sodin} and \cite{BrezisPonce}), 
we have $\Delta u\le \Delta \psi_\alpha$ on $\calI$, 
so it follows that the Riesz measure $\eta$ of $u$ gives total mass 0 to $\G$.
Let $C_0$ be the constant
such that $U^\eta=u+C_0$. Such a constant exists, since by Riesz' theorem
it holds that $U^\eta=u+h$ for some harmonic function, which has to be constant
since the functions $U^\eta(z)$ and $u(z)$ have the same growth at infinity, up to
order $\Ordo(1)$.

The function $w$ 
is an admissible subharmonic function for the obstacle problem on $\C$
with thin obstacle $g$ on $\partial\G$ and full obstacle $\tfrac{1}{2\alpha}|z|^2$,
so we have
$u\ge w=v_0-B_{\alpha}(\mu_{0})$, and as a consequence also
\[
U^\eta(z)\ge v_0-B_{\alpha}(\mu_{0})+C_0=U^{\mu_0}-B_{\alpha}(\mu_{0})+C_0.
\] 
From this it follows that, using the notation 
$\Sigma(\mu,\mu')=-\int U^\mu\,\diff\mu'$,
we have
\[
-\Sigma(\eta)=-\Sigma(\eta,\eta)\le -\Sigma(\mu_{0},\eta) - C_0+B_\alpha(\mu_0)
\le -\Sigma(\mu_0)-2C_0+2B_\alpha(\mu_0).
\]
Since $u(z)\le \frac{1}{2\alpha}|z|^2$ a.e.\ and since $u$ 
satisfies the sub-mean value property, 
it follows that $B_{\alpha,\G}(\eta)\le C_0$. 
Adding up the two terms that make up $I_\alpha(\eta)$
we find that the constants $2C_0$ cancel,
and that $I_\alpha(\eta)\le I_\alpha(\mu_{0})$. 
Since the extremal measure $\mu_{0}$ is the 
unique minimizer of $I_\alpha$, the result follows.
\end{proof}

\subsection{Two elementary observations}
We begin with an observation which tells us how the 
potential of the extremal measure 
$\mu_{\alpha,\G}$ behaves on its support away from the hole.

\begin{prop}\label{prop:contact-at-mass}
Let
$z_0\in\mathrm{supp}(\mu_{\alpha,\G})\setminus \partial \G$. Then
\[
U^{\mu_{\alpha,\G}}(z_0)-\frac{|z_0|^2}{2\alpha}
=B_{\alpha}(\mu_{\alpha,\G}).
\]
\end{prop}

\begin{proof}
Suppose, on the contrary, that we have
\[
U^{\mu_{\alpha,\G}}(z_0)-\frac{|z_0|^2}{2\alpha}
  <B_{\alpha}(\mu_{\alpha,\G}).
\]
We may choose a small neighbourhood $V$ of $z_0$ 
where the equilibrium measure has positive mass.
If $V$ is made small enough, we may then replace $\mu_{\alpha,\G}$ by 
\[
\eta=\chi_{V^c}\mu_{\alpha,\G}+\mathrm{Bal}(\mu_{\alpha,\G},V^c)
\]
without increasing the value of $B_{\alpha}(\mu_{\alpha,\G})$. 
Since the energy $-\Sigma$ only decreases under the 
balayage operation it follows that 
$I_{\alpha}(\eta)<I_{\alpha}(\mu_{\alpha,\G})$, which is a contradiction.
\end{proof}

We also have the following simple lemma.
\begin{lem}\label{lem:balance}
Assume that $\G$ is contained in the unit disk.
For $\alpha> \e$, the value 
$B_\alpha(\mu_{\alpha,\G})$ is attained by
$U^{\mu_{\alpha,\G}}(z)-\frac{|z|^2}{2\alpha}$ on $\G$ 
as well as on $\overline{\G}^c$. Moreover, the singular mass
is bounded by
\[
\mu_{\alpha,\G}^{\s}(\C)\le \e\alpha^{-1}.
\]
\end{lem}
\begin{proof}
We first show that the measure $\mu_{\alpha,\G}$ 
is not concentrated on $\partial \G$, meaning that 
\[
\mu_{\alpha,\G}^\s(\C)=\mu_{\alpha,\G}(\partial\G)<1.
\]
In fact, we will find a stronger bound on the total mass
of $\mu_{\alpha,\G}^{\s}$.
Suppose that $\mu_{\alpha,\G}(\partial\G)=\frac{p}{\alpha}$ for some $p>1$,
so that $\mu_{\alpha,\G}$ is a candidate for the overcrowding problem 
on $\D$ with parameter $p$ (see \cite{GhoshNishry1}). 
But if $p>\e$, we have from the assumption of mass concentration on $\partial\G$,
the result of \cite{GhoshNishry1} on the optimal values for the overcrowding problem,
and lastly from the inclusion $\G\subset\D$ that
\[
I_{\alpha}(\mu_{\alpha,\G})\ge 
I_{\alpha}(\mu^p_{\alpha,\D})>I_{\alpha}(\mu_{\alpha,\D})\ge 
I_{\alpha}(\mu_{\alpha,\G}),
\]
which yields a contradiction.

Since there exist points $z\in \partial \G^c$ 
which belong to the support of $\mu_{\alpha,\G}$,
we find that $B_\alpha(\mu_{\alpha,\G})$ 
is attained outside the closure of $\G$ by 
Proposition~\ref{prop:contact-at-mass}.

For the other direction, assume that the supremum is not 
attained on $\overline{\G}$. But then it is not attained 
on $\overline{\G_1}$, where $\G_1$ contains a neighbourhood
of $\overline{\G}$. We split $\mu_{\alpha,\G}$
into a sum $\mu_1+\mu_2$, where $\mu_1$ is supported 
on $\partial \G$ and $\mu_2$ is supported on 
$\overline{\G}^c$. If 
$\mu_1$ is non-zero, then we may form the balayage measure 
$\eta=\mathrm{Bal}(\mu_1,\G_1^c)$ and put 
\[
\mu_t=(1-t)\mu_1+t \eta + \mu_2
\]
to obtain an admissible measure $\mu_t$ for which 
$-\Sigma(\mu_t)<-\Sigma(\mu_{\alpha,\G})$ while 
$B_\alpha(\mu_t)=B_\alpha(\mu_{\alpha,\G})$. 
This contradicts minimality of $I_\alpha(\mu_{\alpha,\G})$.
It remains to be prove that $\mu_1$ is non-zero. 
This is obvious, however, from the obstacle problem.
Indeed, if $\mu_1=0$ this means that the obstacle on 
$\Gamma$ is inactive, which would say that
\[
U^{\mu_{\alpha,\G}}(z)=\sup\big\{u(z):
  u\in\mathrm{SH}_1,\,u(z)\le \tfrac{|z|^2}{2\alpha}+C\big\},
\]
for some constant $C$, the solution of which is the potential of the 
unconstrained minimization of $I_\alpha$, that is 
equal to $U^{\alpha^{-1}\chi_{\D(0,\sqrt{\alpha})}}$,
which is inadmissible. The conclusion of the lemma follows.
\end{proof}

\subsection{The structure of extremal measures}
We next supply a crude description of the structure 
of equilibrium measures on the hole event.
The information gained on the coincidence set is of 
particular importance for later applications,
as it allows to localize the problem of minimizing $I_\alpha(\mu)$
over $\calM_\G$ to the disk 
$\D(0,\sqrt{\alpha})$, fix the value of $B_\alpha(\mu_{\alpha,\G})$,
and replace the logarithmic energy
by a Dirichlet integral.

Before we proceed, we need a specific function to compare
the solution of an obstacle problem with.
Denote by $\nu_{\alpha,r}$ the measure
\[
\nu_{\alpha,r}=\tfrac{1}{\pi\alpha}\chi_{\D(0,\sqrt{\alpha})
\setminus \D(0,\sqrt{r})}\,\diffA,
\] 
and define the function $V$ on $\D(0,\sqrt{\alpha})$ by
\begin{align}
\label{eq:align-Vr1}
V(z)&=\frac{r}{\alpha}\log|z|+U^{\nu_{\alpha,r}}(z)-c_\alpha\\
&=\frac{r}{\alpha}
\big(U^{\delta_0}(z) - U^{\frac{1}{\pi r}\chi_{\D(0,\sqrt{r})}}(z)\big) 
+ U^{\frac{1}{\pi\alpha}\chi_{\D(0,\sqrt{\alpha})}}(z)-c_\alpha,
\label{eq:align-Vr2}
\end{align}
where the constant $c_\alpha$ is given by
\begin{equation}\label{eq:c-alpha}
c_\alpha=\tfrac12\big(\log\alpha-1\big).
\end{equation}
\begin{prop}\label{prop:Vr} 
The function $V$ meets the bound
$V(z)\le \frac{|z|^2}{2\alpha}$ on $\D(0,\sqrt{\alpha})$, 
and if $\alpha\ge \alpha_0\coloneqq 16\e^3$ and $r=\frac14\alpha$,
the coincidence set for $V$ with $\psi_\alpha$ contains
$\mathbb{A}(\tfrac12\sqrt{\alpha},\sqrt{\alpha})$, and we have
$V\vert_{\T(0,2)}\le 
-\frac14$.
\end{prop}
\begin{proof}
That the global bound holds and that the coincidence set contains the annulus annulus
$\mathbb{A}(\sqrt{r},\sqrt{\alpha})$ is clear from the fact that the first term
in \eqref{eq:align-Vr2} is negative, and vanishes outside $\D(0,\sqrt{r})$.
Provided that $|z|^2\le r$ we have 
\[
V(z)=\frac{r}{2\alpha}\Big(\log \frac{|z|^2}{r} + \frac{|z|^2}{r}\Big),
\]
so if $r>2$ we find that
\[
V(z)= -\frac{r}{2\alpha}\log\big(\tfrac{r}{4\e^{4/r}}\big),\qquad |z|=2.
\]
Let $\alpha\ge\alpha_0\coloneqq 16\e^3$, and choose $r=\frac14\alpha$. 
We find that the coincidence set for $V$ equals
$\mathbb{A}(\tfrac12\sqrt{\alpha},\sqrt{\alpha})$ 
and that 
\[
V(2)= -\tfrac18\log \tfrac{\alpha}{16\e^{\frac{16}{\alpha}}}\le 
-\tfrac18\log \e^2=
-\tfrac14.
\]
This completes the proof
\end{proof}
\begin{prop}\label{prop:struct}
Assume that $\G$ is contained in the unit disk. 
Then the equilibrium measure takes the form
\[
\diff\mu_{\alpha,\G}=\diff\mu_{\alpha,\G}^{\s}
+\tfrac{1}{\pi\alpha}\chi_{\D(0,\sqrt{\alpha})\setminus\Omega}\,\diffA,
\]
where $\Omega=\Omega_\alpha$ is an open set containing $\G$ and where the 
singular part of the measure is supported on the boundary of $\G$.
Moreover, there exists a universal constant $\alpha_0$, such that 
the closed annulus $\overline{\mathbb{A}}(\tfrac12\sqrt{\alpha},\sqrt{\alpha})$ 
belongs to the support of $\mu_{\alpha,\G}$ whenever $\alpha\ge \alpha_0$.
\end{prop}

\begin{rem}
{\rm (a)}\; The domain $\Omega$ is the forbidden region defined in the introduction.

\noindent {\rm (b)}\;We may take $\alpha_0=16\e^3$. This is likely far from sharp, 
but easy to obtain.
\end{rem}

\begin{proof}
The decomposition of the measure $\mu_{\alpha,\G}$ is 
evident in view Proposition~\ref{prop:reg},
and the fact that $U^{\mu_{\alpha,\G}}$ may be 
regarded as the solution to the mixed obstacle problem
on a bounded domain $D=\D(0,\sqrt{\alpha})$ with 
boundary data $f=U^{\mu_{\alpha,\G}}\vert_{\T(0,\sqrt{\alpha})}$.
To see why the latter statement holds, we define
the function
\[
v_0(z)=\sup\big\{v(z):v\in\mathrm{SH}(D)\cap\calK_{\psi,g}^f\big\}
\]
and its upper semicontinuous regularization $v_0^*$, which by the regularity of $\psi$ has
$v^*_0=v_0$ except possibly on a polar subset of $\partial\G$.
Abusing notation slightly, we denote the regularized function by $v_0$,
and note that this function solves the obstacle problem for $\calK_{\psi,g}^f$. Moreover, 
$v_0\ge U^{\mu_{\alpha,\G}}$ with equality on $\partial D$.
If we glue $v_0$ and $U^{\mu_{\alpha,\G}}$ together along $\partial D$
we obtain a majorant $w$ of $U^{\mu_{\alpha,\G}}$. If we can show that $w$
is subharmonic, it follows that the Riesz measure of 
which meets all requirements for the defining minimization problem
for $\mu_{\alpha,\G}$. But this follows from the sub-mean value property,
which trivially holds on $D$ and on the interior of $D^\c$, while on $\partial D$
we have
\[
w(z)=U^{\mu_{\alpha,\G}}(z)\le \frac{1}{|B|}\int_B U^{\mu_{\alpha,\G}}\,\diffA
\le \frac{1}{|B|}\int_B w\,\diffA.
\]

It remains to prove the statement concerning the coincidence set. 
We first observe that the singular mass $\mu_{\alpha,\G}^{\s}(\C)$ 
meets the bound
\[
\mu_{\alpha,\G}^{\s}(\C)\le \frac{\e}{\alpha},
\]
by Lemma~\ref{lem:balance}. 
Secondly, we observe that away from $\G$, say for definiteness 
outside $\D(0,2)$, 
we may reconstruct the function 
$U=U^{\mu_{\alpha,\G}}-B_\alpha(\mu_{\alpha,\G})$ 
from its boundary values $f=U\vert_{\T(0,2)}$ by solving an obstacle problem 
in the same way as above, so that
\[
U(z)=\sup\big\{v(z):v\in\mathrm{SH}_1(\C\setminus \D(0,2)), 
u\le \tfrac{|z|^2}{2\alpha}, u\vert_{\T(0,2)}=f\big\}.
\]
We will complete the proof by estimating the 
boundary values of $U$ from below by a constant $C(\alpha)$, 
and show that the solution $U$ dominates the solution $V$ 
of Proposition~\ref{prop:Vr}, whose coincidence set
contains the annulus $\mathbb{A}(\tfrac12\sqrt{\alpha},\sqrt{\alpha})$.
We proceed to estimate $B_\alpha(\mu_{\alpha,\G})$ from above. 
We know that $B_\alpha(\mu_{\alpha,\G})$ is attained on 
both $\G$ and $\G^c$. On $\G$, we have that 
\[
U^{\mu_{\alpha,\G}^{\s}}(z)\le 
\mu_{\alpha,\G}^{\s}(\C)\sup_{w\in\partial \G}\log|z-w|\le 
\frac{\e\log 2}{\alpha}.
\]
This allows us to estimate 
\[
U^{\mu_{\alpha,\G}}(z)-\tfrac{|z|^2}{2\alpha}
\le\big(U^{\mu_\alpha}(z)-\tfrac{|z|^2}{2\alpha}\big)-
U^{\frac{1}{\alpha}\chi_\Omega}(z)+\tfrac{\e \log 2}{\alpha}
=c_\alpha-U^{\frac{1}{\alpha}\chi_{\Omega}}(z)+\tfrac{\e \log 2}{\alpha},
\]
where $c_\alpha$ is given by \eqref{eq:c-alpha} and where 
$\mu_\alpha=\tfrac{1}{\alpha}\chi_{\D(0,\sqrt{\alpha})}$.
From the estimate of the singular mass in Lemma~\ref{lem:balance}
and the fact that 
$\mu_{\alpha,\G}$ is a probability measure, we infer that $|\Omega|\le \e$.

We next observe that there are two {\em extremal distributions of mass}, 
which allow us to bound $U^{\chi_{\Omega}}(z)$. Indeed, by monotonicity of the 
logarithmic kernel, it holds that $U^{\chi_\Omega}(z)\le U^{\mu_\alpha^{z,1}}(z)$, 
where $\mu_\alpha^{z,1}=|\Omega|\delta_w$ where  $w=-\sqrt{\alpha}\tfrac{z}{|z|}$, 
which lies at maximal 
distance from $z$ within the allotted domain. Similarly, now using the 
fact that $\mu_{\alpha,\G}^\c$ has uniform density with 
respect to area measure on some plane region, 
we find that $U^{\chi_{\Omega}}(z)\ge U^{\chi_{\D(z,\sqrt{|\Omega|})}}(z)$. 
Performing the necessary computations, 
we see see that 
\[
\frac{1-\e}{2\alpha}\le U^{\frac{1}{\alpha}\chi_{\Omega}}(z)\le 
\frac{\e\log(2\sqrt{\alpha})}{\alpha},
\qquad z\in\G.
\]
In conclusion, we have that 
\[
B_\alpha(\mu_{\alpha,\G})\le c_\alpha+\frac{\e (2\log 2+1)-1}{2\alpha}.
\]
We turn to estimating $U\vert_{\T(0,2)}$ from below,
and first notice that the potential of the singular measure $\mu_{\alpha,\G}^\s$ is positive
on the circle $\T(0,2)$. As was also observed above, we have 
$U^{\mu_{\alpha,\G}^{\c}}(z)
=U^{\alpha^{-1}\chi_{\D(0,\sqrt{\alpha})}}(z)-U^{\frac{1}{\alpha}\chi_\Omega}$, 
so by our previous estimate on $B_\alpha(\mu_{\alpha,\G})$ 
we have that on $\T(0,2)$
\begin{equation}\label{eq:lower-bd-rad2}
U^{\mu_{\alpha,\G}}(z)-B_\alpha(\mu_{\alpha,\G})\ge - 
\frac{2\e\log(2\sqrt{\alpha})+\e\,(2\log 2+1)-5}{2\alpha}.
\end{equation}
The lower bound \eqref{eq:lower-bd-rad2} 
for $U\vert_{\T(0,2)}$
is monotonically decreasing in $\alpha$ for $\alpha>\e$, 
so we may replace $\alpha$ by $\alpha_0=16\e^3$. 
For this value the bound \eqref{eq:lower-bd-rad2} gives
$U\vert_{\T(0,2)}= \frac{5-\e(4+\log(256)}{32\e^3}> -\frac{1}{30}>-\tfrac{1}{4}$,
so that the function $V=V_{r,\alpha}$ of Proposition~\ref{prop:Vr} 
is a competitor for the obstacle problem. 
But then we have  $U\ge V$ on $\C$, and it follows that the support of 
$\mu_{\alpha_0,\G}$ covers the annulus 
$\mathbb{A}(\tfrac12\sqrt{\alpha},\sqrt{\alpha})$.
\end{proof}

We next show that $\mu_{\alpha,\G}$ stabilizes 
as the parameter $\alpha$ grows.

\begin{prop}
\label{prop:glue}
Assume that $\alpha\ge \alpha_0$.
Then we have that
\[
\mu_{\alpha,\G}=
\tfrac{\alpha_0}{\alpha}\mu_{\alpha_0,\G}
+\tfrac{1}{\pi\alpha}
\chi_{\mathbb{A}(\tfrac12\sqrt{\alpha_0},\sqrt{\alpha})}\diffA.
\]
\end{prop}

\begin{proof}
Provided that $\alpha\ge \alpha_0$, we have in view of 
Proposition~\ref{prop:struct} that $\mu_{\alpha,\G}$
takes the form
\[
\mu_{\alpha,\G} =\tfrac{1}{4}\mu^{\langle 0\rangle} 
+ \tfrac{1}{\pi\alpha}\chi_{A_{\alpha}}\diffA
\]
where $\mu^{\langle 0\rangle} $ is a probability measure supported on 
$\D(0,\tfrac12\sqrt{\alpha})$ 
whose potential is constant on $\T(0,\tfrac12\sqrt{\alpha})$, and where 
$A_{\alpha}$ is shorthand for the annulus 
$\A(\tfrac12\sqrt{\alpha},\sqrt{\alpha})$.
Observe that $\mu_{\alpha/4,\G}$ also is a probability measure on 
$\D(0,\tfrac12\sqrt{\alpha})$ with constant potential
on $\T(0,\tfrac12\sqrt{\alpha})$, 
in view of Proposition~\ref{prop:struct}.

The potential of the measure $\frac{1}{\pi\alpha}\chi_{\A_{\alpha}}\diffA$ 
is constant 
on $\D(0,\tfrac12\sqrt{\alpha})$ so in particular on $\G$, and it equals 
$c_{\alpha}-\frac{1}{4}c_{\alpha/4}$
there. Since
we know that the value $B_{\alpha}(\mu_{\alpha,\G})$ is attained
on $\G$ and equals $c_{\alpha}$, we may conclude that
\[
c_{\alpha}=B_{\alpha}(\mu_{\alpha,\G})
=B_{\alpha}(\tfrac{1}{4}\mu^{\langle 0\rangle} ) + c_{\alpha}-
\tfrac{1}{4}c_{\alpha/4}.
\]
Moreover, we have, by scaling,
\[
B_{\alpha}(\tfrac{1}{4}\mu^{\langle 0\rangle} )
=\frac{1}{4}B_{\alpha/4}(\mu^{\langle 0\rangle} ),
\]
so that $B_{\alpha/4}(\mu^{\langle 0\rangle} )
=c_{\alpha/4}=B_{\alpha/4}(\mu_{\alpha/4,\G})$. 

Finally, since the potential $U^{\mu^{\langle 0\rangle} }$ equals 
$\log |z|$ on 
the support of $\chi_{A_{\alpha}}$, it follows that 
the energy $-\Sigma(\mu_{\alpha})$ may be rewritten as
\[
-\Sigma(\mu_{\alpha})=
-\frac{1}{16}\Sigma(\mu^{\langle 0\rangle} )+C(\alpha)
\]
for a constant $C(\alpha)$ which does not depend on the choice of $\mu^{\langle 0\rangle} $,
within the given restrictions mentioned above. It follows that $\mu^{\langle 0\rangle} $
minimizes the energy over all measures $\mu\in\calM_\G$ with support on
$\D(0,\tfrac12\sqrt{\alpha})$ whose potentials are
constant on $\T(0,\tfrac12\sqrt{\alpha})$, under the additional constraint
$B_{\alpha/4}(\mu^{\langle 0\rangle} )=B_{\alpha/4}(\mu_{\alpha/4,\G})$. 
But $\mu_{\alpha/4,\G}$ 
belongs to this class,
and minimizes $-\Sigma(\mu)$ over the larger class
of measures $\mu\in\calM_\G$ with $B_{\alpha/4}(\mu)=c_{\alpha/4}$.
Hence, it follows that $\mu^{\langle 0\rangle} =\mu_{\alpha/4,\G}$.

Whenever $\tfrac{1}{4^{k+1}}\alpha\ge \alpha_0$, this process can be repeated
to express $\mu_{4^{-k}\alpha,\G}$ in terms of $\mu_{4^{-k-1}\alpha,\G}$
by adding an annular shell of uniform mass. For the unique integer $k$ for which
$\alpha_0\in (4^{-k-1}\alpha,4^{-k}\alpha)$, the above procedure instead 
leads to
\[
\mu_{4^{-k}\alpha,\G}=\frac{\alpha_0}{4^{-k}\alpha}\mu_{\alpha_0,\G} + 
\frac{1}{4^{-k}\alpha\pi}
\chi_{\mathbb{A}(\tfrac12\sqrt{\alpha_0},4^{-k}\sqrt{\alpha})}\diffA.
\]
This completes the proof.
\end{proof}

\section{Separation of the free boundary from the thin obstacle}
\label{s:smooth-non-deg}

\subsection{Local control of the Riesz mass on the coincidence set}
Denote by $\calG$ a bounded simply connected domain 
whose boundary $\Gamma=\partial\G$ is 
piecewise $C^2$\,--\,smooth without cusps. 
Denote the collection of all corners of $\Gamma$ by
$\calE$, and by $\calE_0$ the collection of 
all corners with interior angle $\pi\sigma$ for $\sigma\in(0,1)$. We let
$\calE_1=\calE\setminus\calE_0$.

\begin{figure}[t!]
\begin{subfigure}[t]{.38\textwidth}
  \centering
  \includegraphics[width=1\linewidth]{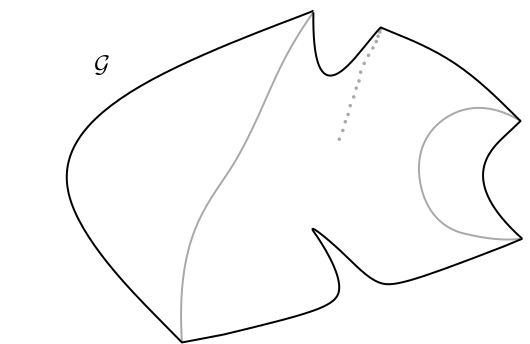}
\end{subfigure}
\hspace{25pt}
\begin{subfigure}[t]{.38\textwidth}
  \includegraphics[width=\linewidth]{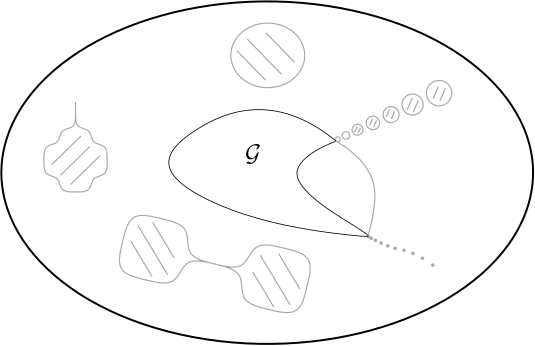}
\end{subfigure}
\caption{The coincidence sets on $\G$ (left) and 
$\C\setminus\overline\G$ (right), 
with contact between the free boundary $\calS$ and the fixed boundary $\partial\G$
only near corners. 
Taken together, the figures illustrate all different types of
free boundary points for an analytic obstacle.} 
\label{fig:G-thin-cont-coincidence}
\end{figure}

We continue the study the extremal measure $\mu_{\alpha,\G}$
for the minimization of $I_\alpha$ over the set of probability measures
which give no mass to $\G$. To ease notation, we let $\mu_0=\mu_{\alpha,\G}$,
$u_0=U^{\mu_{\alpha,\G}}$ and $\psi=\psi_\alpha$. We will assume throughout 
this section that $\alpha\ge\alpha_0(\G)$ (see Proposition~\ref{prop:struct}).
Recall that $u_0$ is said to be non-degenerate if 
$u_0<\psi$ along $\Gamma$,
c.f. \eqref{eq:non-deg}, and denote by $\diff(z,E)$ the distance
$\inf_{w\in E}|z-w|$, where $z\in\C$ and $E\subset\C$.

\begin{thm}\label{thm:s-ndeg}
Assume that $\Gamma$ is piecewise $C^2$\,--\,smooth without cusps.
Then it holds that
\begin{equation}\label{eq:eq-s-ndeg}
\diff(z,\partial\G)\lesssim
\diff(z,\calE_0)^{\max\{1,\gamma\}}, \qquad 
z\in\calI,
\end{equation}
where $\gamma=\frac{1-\sigma}{\sigma}$ and $\sigma$ is the smallest angle for 
corner points on $\partial\G$.
If $\Gamma$ is $C^2$-smooth, then $u_0$ is non-degenerate.
\end{thm}

For $j=0,1$ we denote by $\calE^\eta_j$ the fattened singular sets
\begin{equation}
\calE^\eta_j=\bigcup_{z\in\calE_j}\D(z,\eta)
\end{equation}
and put $\calE^\eta=\calE^\eta_0\cup
\calE_1^\eta$. Then for any fixed $\eta>0$ it holds that 
$\mathrm{d}(\calI,\Gamma\setminus \calE^{\eta}_0)>0$. 
The consequence of the theorem is illustrated in 
Figure~\ref{fig:G-thin-cont-coincidence}.

The idea is to use the first order variational 
formula for perturbations of 
\eqref{eq:obst-thin}, and given a sequence of 
points which approach the boundary 
at a rate in breach of Theorem~\ref{thm:s-ndeg}, try to
find suitable perturbations
which decrease the Dirichlet energy while yielding admissible 
measures ($\mu_\epsilon(\G)=0$).

The key to obtaining such a result is a local estimate 
of the distributional Laplacian $\Delta u_0$ near a point $z_0\in \calS$, 
valid in particular for $z_0\in\Gamma\cap\calS$. Here, we recall that 
$\calS$ is the coincidence set defined in \eqref{eq:coincidence}. 
We denote by $\calT$ the coincidence 
set with the full obstacle on $\Gamma$
\[
\calT=\{z\in\Gamma:u(z)=\psi(z)\},
\]
and by $\calT_r$ the $r$-neighbourhood of $\calT$ on $\Gamma$.

\begin{lem}\label{lem:neumann} 
For the full obstacle $\psi_\alpha(z)=\tfrac1{2\alpha}|z|^2$,
it holds for any point $z_0\in\calS$ that
\[
\mu_0(\D(z_0,r))\le \tfrac1{2\alpha}\e^2r^2,
\qquad 0<r<1.
\]
Moreover,
it holds that
\[
\frac{\mu(\calT_r)}{|\calT_r|}\le \tfrac{7}{\alpha}\e^2 r,
\qquad r\to0.
\]
\end{lem}

\begin{proof}
We set $v_0=\alpha u_0$, $\nu_0=\alpha\mu_0$ and $\psi=\tfrac12|z|^2$.
We have
\[
v_0(z)-v_0(0)=\int_{\C}(\log|z-w|-\log|w|)\diff\nu_0(w).
\]
Averaging over the small circle $|z|=r$m
we find that
\[
\frac{1}{r}\int_{|z|=r}(v_0(z)-v_0(z_0))\diffs(z)=
\int_{\D(0,r)}H(r,w)\diff\nu_0(w),
\]
where
\[
H(r,w)=\frac{1}{r}\int_{|z|=r}(\log|z-w|-\log|w|)\diffs(z)
=\log\frac{r}{|w|}\chi_{\D(0,r)}.
\]
The kernel $H(r,w)=\log\frac{r}{|w|}$ is bounded from below by $1$ on 
$[0,r\e^{-1}]$ and is positive everywhere on $[0,r]$, 
so 
\[
\int_{\D(z_0,r)}H(r,w)\diff\nu_0\ge \nu_0(\D(z_0,r/\e)).
\]
We next notice that
\[
\frac{1}{r}\int_{\T(0,r)}(\psi(z)-\psi(z_0))\diffs=
\frac{1}{2\pi}\int_{\D(0,r)\setminus\calS}H(r,w)
\Delta \psi(z)\diffA(w).
\]
Since $\Delta \psi= 2$, we find that
the right-hand side is bounded above by 
\[
\frac{1}{\pi}\int_{\D(0,r)}H(r,w)\diffA(w)= \frac{r^2}{2}.
\]
Putting the desired estimates together, we find that 
whenever $v_0(z_0)=\varphi(z_0)$, we have
\begin{multline*}
0\ge \frac{1}{r}\int_{\T(0,r)}\big(v_0(z)-\psi(z)\big)\diffs(z)
\\
=\frac{1}{r}\int_{\T(0,r)}\big(v_0(z)-v_0(z_0)-
(\psi(z)-\psi(z_0))\big)\diffs(z)
\ge\nu_0^{\s}(\D(0,r/\e))
-\frac{r^2}{2},
\end{multline*}
which after replacing $r$ by $\e r$ yields the desired bound
\begin{equation}\label{eq:bound-mu-gamma}
\nu_0^{\s}(\Gamma_{r})\le \tfrac{1}{2}\e^2 r^2.
\end{equation}

In order to reach the final conclusion, we may use the 
Vitali covering lemma to
find a collection of $N(r)<\infty$ disjoint 
balls $B_j$ of radius $\tfrac15 r$, 
centered at points $z_j\in\calT$, such that
that the balls $B_j'$ rescaled by $5$ cover 
$\bigcup_{z\in\calT}\D(z,r/5)$. Notice 
in particular that the latter set covers $\calT_{r/5}$.
We use the notation $\Gamma^{(z_j)}_{r}=\Gamma\cap\D(z_j,r)$. The set 
$\Gamma^{(z_j)}_{r/5}$ are disjoint and satisfy 
$|\Gamma^{(z_j)}_{r/5}|\ge \tfrac25 r$,
while the collection $\Gamma_{r}^{(z_j)}$ cover $\calT_{r/5}$. 
If we invoke the estimate \eqref{eq:bound-mu-gamma}, we find that
\[
\rho_{r/5}=\frac{\nu_0(\calT_{r/5})}{|\calT_{r/5}|}\le 
\frac{\sum_j\nu_0(\Gamma_{r}^{z_j})}{\sum_j|\Gamma^{z_j}_{r/5}|}
\le \frac{\tfrac12\e^2 r^2N(r)}{\tfrac25r N(r)}\le 
\frac{5\e^2}{4}r,
\qquad r\to0,
\]
which after rescaling yields that
\[
\rho_r\le \frac{25\e^2}{4}r.
\]
This completes the proof.
\end{proof}

\subsection{Quantitative non-degeneracy}
\label{ss:q-non-deg}
Recall that $\partial\G$ is a piecewise smooth 
curve with finitely many corners $\{z_j\}=\calE$, and let 
$w\in\partial\G$. 
We need estimates of $\omega(z,\D(w,\epsilon)\cap\partial\G,\G)$
in terms of $\mathrm{d}(z,\calE)$ and $\mathrm{d}(z,\partial\calG)$.
To this end, use Proposition~\ref{prop:conf-corner}
to control the quantity
\begin{equation}\label{eq:P-ker}
\frac{\diff\omega(z,\cdot,\G)}{\diffs}=P_{\G}(z,\cdot)=|\phi_z'(\cdot)|,
\end{equation}
where $\phi_z$ is a conformal mapping 
of $\G$ onto the unit disk so that $\phi_z(z)=0$.
Denote by $m_{\zeta}(z)$ the M{\"o}bius transformation 
$m_\zeta=(z-\zeta)/(1-\bar\zeta z)$.
If $\phi=\phi_{z_0}$ for some fixed interior point $z_0$ of $\G$, we have
\[
\phi_z'(w)=\partial_wm_{\phi(z)}(\phi(w))
=\partial_w\frac{\phi(w)-\phi(z)}{1-\overline{\phi(z)}\phi(w)}
=\phi'(w)\frac{1-|\phi(z)|^2}{(1-\overline{\phi(z)}\phi(w))^2}
\]
Thus, provided that $z$ is confined to the set $\G\setminus \calE_0^{\eta}$
for a fixed $\eta>0$,
we find that when $w$ approaches a point $\zeta\in\calE$, 
the density of harmonic measure
satisfies
\[
|\phi'_z(w)|\gtrsim (1-|\phi(z)|^2)|\phi'(w)|\gtrsim
\diff(z,\partial\G)|\phi'(w)|,
\]
where we have used Remark~\ref{rem:prop-conf-corner}
in the last step.
Hence, it follows from 
Proposition~\ref{prop:conf-corner} that
\begin{equation}\label{eq:elem-bd-hm}
\frac{\diff\omega(z,w,\G)}{\diffs(w)}
\gtrsim\diff(z,\partial\G)|w-\zeta|^\gamma, \qquad z\in\G\setminus\calE_0^\eta
\end{equation}
as $w\to\zeta\in\calE_0$, where $\gamma=(1-\sigma)/\sigma\in(0,\infty)$. 

We are now ready for a proof of the main result of this section.

\begin{proof}[Proof of Theorem~\ref{thm:s-ndeg}]
We begin with a sequence of points $w_k\in\calI$ 
approaching a boundary point $z_0$, 
and aim to control $\diff(w_k,\partial\G)$ from below. 
We begin to establish the bound \eqref{eq:eq-s-ndeg} when $z_0\in\calE_0$, and 
then show that $z_0\in\Gamma\setminus\calE_0^\eta$ is impossible.

Thus let $z_0\in\calE_0$, and suppose that the claim of the theorem is false,
so that there exists a sequence $(w_k)$ of points in $\calI$
converging to $z_0\in\partial\G$, for which
\[
\diff(w_k,\partial\calG)\le |w_k-z_0|^{\max\{\gamma,1\}} a_k,
\]
where $a_k$ are positive numbers
tending to zero.
We let $r_k=2\diff(w_k,\partial\calG)$ 
and define a subarc of $\partial\G$ by 
\[
T_k=\D(w_k,r_k)\cap\partial\G
\]
which since $\partial\G$ is piecewise smooth without cusps has length
comparable to $r_k$.\\
\\
\noindent {\sc Claim 1.}\;
We show that there exists an arc $U\subset\Gamma\setminus\calE^\eta$ 
for some $\eta>0$ 
with positive measure
$\mu_0^\s(U)=\delta>0$. Such a set exists, since if it did not,
then $\calI\cap\partial\G$ would be everywhere 
dense in $\mathrm{supp}(\mu_0)$. 
But then $u_0(z)=\psi(z)$ on $\mathrm{supp}(\mu_0)$
by semi-continuity of $u_0$, and in view of 
Lemma~\ref{lem:neumann} we have
for the set
\[
\calT_r=\{z\in\Gamma:\diff(z,\calI)\le r\}
\]
that
\[
\mu_0(\Gamma)=\mu_0(\calT)\le\mu_0(\calT_r)\le 
7\e^2|\calT_r|r\qquad \text{for all }\,r>0
\]
so $\mu_0(\Gamma)=0$. But then the obstacle is inactive, 
so it follows that $u_0=\psi$,
which is impossible in view of $\Delta u_0\vert_{\G}=0$.

We now define a perturbation $h_k$, by
\[
h_k(z)=\omega(z,U,\calG)-r_k^{-2}a_k^{\frac12}\omega(z,T_k,\G),
\]
and claim that for large enough $k$, 
the function $h_k$ is an admissible 
perturbation which decreases the Dirichlet energy. 
We will show (Claim 2) that $h_k\le 0$ on $\calI$, and (Claim 3) that
\[
\int_{\partial\G}h_k(z)\diff\mu_0^\s(z)>0.
\]

\medskip

\noindent {\sc Claim 2.} We show the inequality $h_k\le 0$ on $\calI$
by showing that the function $h$ is 
non-positive except for a small neighbourhood of 
$U$ in $\G$. This follows by an application
of the maximum principle in the domain
\[
\widetilde{\G}=\{z\in\G:\diff(z,U)\ge C_0\},
\]
where $C_0$ is small enough that $\calI$ belongs 
to the closure of $\widetilde{\G}$
and $\G\setminus\tilde{\G}$ is simply connected.
For this argument to work, we need to 
show that $h_k\le 0$ on the boundary of $\G$ away from the set $U$,
as well as on the cross-cut $\partial\tilde{\G}\setminus\partial\G$.
We may immediately observe that $h_k\equiv 0$ on 
$\partial\G\setminus (U\cup T_k)$ and negative on $T_k$ by definition.

Turning to analyzing $h_k$ on the cross cut, we simply use the estimate
\eqref{eq:elem-bd-hm}, to find that
\begin{equation}\label{eq:est-harm-Tk-E0}
\omega(z,T_k,\G)\gtrsim \diff(z,\partial\calG)\int_{T_k}
|w-z_0|^{\gamma}\diffs(w)
\cong \diff(z,\partial\calG)|w_k-z_0|^\gamma,
\qquad z\in\partial\tilde{\G}\setminus\partial\G.
\end{equation}
But this yields that
\[
h_k(z)\lesssim \diff(z,\partial\G)\Big(1-a_k^{\frac12}r_k^{-2}
|w_k-z_0|^\gamma\Big),
\qquad z\in\partial\tilde{\G}\setminus\partial\G,
\]
where we use the trivial bound 
$\omega(z,U,\partial\G)\lesssim \diff(z,\partial\calG)$, which holds since
$U\subset\partial\G\setminus\calE^\eta$.
Moreover, since by assumption
\[
r_k\le |w_k-z_0|^{\max\{\gamma,1\}} a_k\le |w_k-z_0|^\gamma a_k,
\]
where $a_k\to0$,
it holds that 
\[
a_k r_k^{-1}|w_k-z_0|^\gamma\ge 1,
\]
so that 
\[
a_{k}^{\frac12}r_k^{-2}|w_k-z_0|^{\gamma}\ge 
r_k^{-1}a_k^{-\frac12}\to\infty
\]
and so $h_k<0$ also here. It follows that for 
large enough $k$, say $k\ge k_0$, the
function $h_k$ is an admissible perturbation.

\medskip 

\noindent {\sc Claim 3.}
We need to study the quantity
\[
\int_{\partial\G}h_k(z)\diff\mu_0^\s(z)=\mu_0^\s(U) 
- a_k^{\frac12}r_k^{-2}\mu_0^\s(T_k).
\]
In view of the regularity estimate of Lemma~\ref{lem:neumann}
we have the control
\[
\mu_0(\D(w,r))\lesssim r^2,\qquad w\in\calS.
\]
In addition, we have the inclusion $T_k\subset \D(w_k,r_k)$ 
where we recall that $w_k\in\calI\subset\calS$.
As a consequence, it follows that
\[
\mu_0(T_k)\lesssim r_k^2.
\]
But then it follows, from the fact that 
$\mu_0^{\s}(U)=\delta>0$, that
\[
\int h_k\,\diff\mu_0^\s>0
\]
for $k$ large enough, say $k\ge k_1$. 

\medskip

For $k\ge \max\{k_0,k_1\}$ the perturbation is both 
admissible and decreases the Dirichlet
energy, which is impossible. Hence the result follows
for corner points of the first kind, $z_0\in\calE_0$.

We thus assume that $z_0\in\partial\G\setminus\calE_0$. 
Also in this case we argue by contradiction,
so assume that there exists $a_k=\ordo(1)$ and $w_k\in\calI$ 
approaching $z_0$ with
\[
|w_k-z_0|=a_k.
\]
We define the perturbation
\[
h_k=\omega(z,U,\G)-a_k^{-\frac32}\omega(z,T_k,\G),
\]
and perform the same estimates as above, 
the only difference being that $|\phi'(w)|$ is bounded away from zero near $z_0$, 
so that instead of \eqref{eq:est-harm-Tk-E0}
we now have the lower bound
\[
\omega(z,T_k,\G)\gtrsim \diff(z,\partial\calG)\int_{T_k}
\diffs(w)
\cong \diff(z,\partial\calG),
\qquad z\in\partial\tilde{\G}\setminus\partial\G.
\]
The rest of the argument goes through unchanged, so
that as a consequence, we find that for large 
enough $k$, the function $h_k$ is an admissible 
perturbation which decreases the Dirichlet energy.
Hence we conclude that $a_k=\ordo(1)$ is an impossibility, 
and the claim follows.
\end{proof}

\section{The emergence of quadrature domains}\label{s:variational}
\subsection{The perturbation scheme}
In view of Theorem~\ref{thm:basic-obst}, Proposition~\ref{prop:struct},
and Remark~\ref{rem:obst-defn}, the potential $u_0=U^{\mu_{\alpha,\G}}$
of the extremal measure $\mu_0=\mu_{\alpha,\G}$ solves the obstacle problem 
(see Definition~\ref{defn:obst-bdd})
for the class $\calK_{\psi_\alpha,g}(D)$, $g\in H^{\frac12}(\Gamma)$
and where we use the notation $D=\D(0,\sqrt{\alpha})$.

In order to extract more precise information about $\mu_{0}$, we
will perturb the implicit thin obstacle $g$ by a function 
$\epsilon h$, and solve the mixed 
obstacle problem with thin obstacle $g_\epsilon=g+\epsilon h$
and full obstacle $\psi_\alpha$ on $D$.
This procedure produces a family of functions $u_\epsilon$, 
the Riesz masses of which provide a family $\mu_\epsilon$ of 
competitors for the minimization problem,
under suitable conditions on the perturbation $h$. 
We then hope to reveal information about the minimizer by comparing
$I_\alpha(\mu_\epsilon)$ with $I_\alpha(\mu_{0})$
using Corollary~\ref{cor:variational}.

We recall that when the measure $\mu_{0}$ is non-degenerate
in the sense of \eqref{eq:non-deg}, 
Sakai's regularity theorem (see Theorem~\ref{thm:sakai})
implies that the interior coincidence set is finite.
Moreover, in view of Lemma~\ref{lem:balance} this set is non-empty. 
Whenever $h$ is harmonic on $\G$ with $h(\lambda)<0$ for $\lambda\in\calI$,
it turns out that $g_\epsilon=g+\epsilon h$ is an 
admissible perturbation for small $\epsilon>0$. 

\begin{prop}\label{prop:perturb}
Assume that $\mu_{\alpha,\G}$ is non-degenerate, 
that $\G$ is contained in the unit disk, that
$\alpha\ge \alpha_0$, and let 
\[
g=U^{\mu_{\alpha,\G}}\big\vert_{\partial \G}.
\]
Denote by $h$ the boundary values on $\partial \G$ 
of a function harmonic in a neighbourhood of $\G$ 
and negative on $\calI$.
Then for small enough $\epsilon>0$, 
the Riesz masses $\mu_{\epsilon}$ of 
the solutions $u_\epsilon$ to the obstacle 
problem with thin obstacle $g_\epsilon$ and
global obstacle $\psi_\alpha$ are admissible 
for the hole event, and we have that
$$
\int_{\partial \G}h\,\diff\mu_{\alpha,\G}^{\s}\le0.
$$
\end{prop}

\begin{proof}
Since $h$ is negative on $\calI$, it follows (by an argument
familiar from the proof of Theorem~\ref{thm:basic-obst}) 
that the Riesz mass $\mu_\epsilon$ of $u_\epsilon$
is admissible for the hole event.
Moreover, it is clear that 
$B_\alpha(\mu_\epsilon)=B_\alpha(\mu_0)=c_\alpha$, 
by virtue of the coincidence set containing the non-trivial 
annulus $\mathbb{A}(\tfrac12\sqrt{\alpha_0},\sqrt{\alpha})$. Indeed, 
the proof of Proposition~\ref{prop:struct} applies, 
since the boundary values have been perturbed by at most by 
$\epsilon \lVert h\rVert_\infty$. In particular
\eqref{eq:lower-bd-rad2} holds with $\ordo(1)$ added on 
the right hand side, which does not affect the remainder
of the argument.

The functions $u_\epsilon$ are minimizers 
of the Dirichlet energy over the sets $\calK_{\psi, g_\epsilon}$, 
so that the variational 
formula of
Corollary~\ref{cor:variational} applies, 
and yields that
\[
I_{\alpha}(\mu_\epsilon)=
I_{\alpha}(\mu_{\alpha,\G})-4\pi \epsilon\int_{\partial \G}
h\,\diff\mu_{0}^{\s}
+\ordo(\epsilon).
\]
That the boundary integral is negative is 
immediate from minimality of $\mu_0$
for the functional $I_\alpha$ over the class of admissible 
measures for the hole event. This proves the lemma.
\end{proof}

\subsection{Proof of the main theorem}
We first treat the case when $\mu_0$ is non-degenerate.

\begin{lem}\label{lem:pre-thm-sqdom}
Assume that $\alpha\ge \alpha_0$ and that $\mu_0=\mu_{\alpha,\G}$ is non-degenerate.
Then $\mu_0^\s$ there exist a finite set $\Lambda\subset\calG$ and positive weights 
$(\rho_\lambda)_{\lambda\in\Lambda}$ such that
\[
\mu_0^s=\sum_{\lambda\in\Lambda}\rho_\lambda\omega_{\G,\lambda}.
\]
\end{lem}

\begin{proof}
Since $U^{\mu_{0}}<B_\alpha(\mu_0)$ on $\partial\G$, Sakai's theorem implies
that $\calI$ is a finite subset of $\G$.
If $h$ is any harmonic function in $\G$ with 
$h\vert_{\calI}\equiv0$, then for $\delta>0$ we apply
Proposition~\ref{prop:perturb} to $h-\delta$ and $-h-\delta$ to find that
\[
-\mu_{0}^{\s}(\partial \G)\delta<
\int h\,\diff\mu_{0}^{\s}<\mu_{0}^{\s}(\partial \G)\delta.
\]
Since $\delta>0$ was arbitrary, it follows that
\[
\int h\,\diff\mu_{0}^{\s}=0.
\]
Let now $h$ be any smooth function on $\partial \G$, 
and extend $h$ to a harmonic function on $\G$.
The functions $(f_\lambda)_{\lambda\in\Lambda}$ defined by
\[
f_\lambda=\Re\prod_{\lambda'\ne \lambda} (z-\lambda')/(\lambda-\lambda'),
\qquad \lambda\in\calI
\] 
supply a basis for harmonic interpolation
on $\Lambda$. 
We use this basis to show that the integral of $h$ against $\mu_0^\s$ 
is given by point evaluations on $\calI$:
\begin{equation}\label{eq:int-h}
\int h\,\diff\mu_{0}^{\s}=
\int \Big(h-\sum_{\lambda\in\calI} h(\lambda)f_\lambda(z)\Big)\diff\mu_{0}^{\s}
+\sum_{\lambda\in\calI} \rho_\lambda h(\lambda)
=\sum_{\lambda\in\calI} \rho_\lambda h(\lambda),
\end{equation}
where $\rho_\lambda=\int f_\lambda\,\diff\mu_{0}^{\s}$. 
To see that the weights are non-negative, we observe
that if some weight $\rho_{j_0}$ is negative, then we set 
\[
h=-f_{\lambda_0}-\delta'\sum_{\lambda\ne \lambda_0}f_\lambda
\]
for some $\delta'>0$.
We know by Proposition~\ref{prop:perturb} that $\int h\,\diff\mu_{\alpha,\G}^{\s}<0$,
while at the same time we see that
\[
\int\big(-f_{\lambda_0}-\delta'\sum_{\lambda\ne \lambda_0}f_\lambda\big)
\diff\mu_{0}^{\s}=|\rho_{\lambda_0}|+\Ordo(\delta')
\]
which is positive for small enough $\delta'$, which is a contradiction.
Let $\Lambda=\{\lambda\in\calI:\rho_\lambda>0\}$.

By \eqref{eq:int-h}, integrating $h$ with 
respect to the singular part of the equilibrium 
measure is the same as integrating with respect 
to the measure $\sum_{\lambda\in\Lambda}\rho_\lambda\omega_{\lambda_\lambda,\calG}$.
Since $h$ can be taken as an arbitrary continuous function,
this implies that the two measures agree, which completes the proof
that $\mu_0^\s$ is a finite sum of harmonic measures.
\end{proof}

The above lemma allows us to proceed to the proof of
Theorem~\ref{thm:s-qdom}.

\begin{proof}[Proof of Theorem~\ref{thm:s-qdom}]
Suppose first that $\partial\G$ is $C^2$-smooth. In this case, we may apply 
Theorem~\ref{thm:s-ndeg} to conclude that $\mu_0=\mu_{\alpha,\G}$
is non-degenerate in the sense of \eqref{eq:non-deg}, and 
by Lemma~\ref{lem:pre-thm-sqdom} the 
singular component $\mu_{0}^\s$ of the measure 
$\mu_{0}$ takes the form claimed
in the statement of the theorem. We also know 
that the continuous part takes the form 
\[
\mu_{0}^{\c}=\frac{1}{\alpha}
\chi_{\D(0,\sqrt{\alpha})\setminus\Omega},
\]
where $\G\subset\Omega\subset\D(0,\sqrt{\alpha})$ and where 
\[
\Omega=\D(0,\sqrt{\alpha})\setminus (\G\cup 
\mathrm{supp}(\mu_{0}^{\c}))=
(\D(0,\sqrt{\alpha})\setminus \calS)\cap\G^c.
\] 
We set $U=\alpha U^{\mu_{0}^{\c}}$
and notice that $U$ is the envelope
\begin{equation}\label{eq:U-def-obst}
U(z)=\sup\big\{u(z):\tfrac{1}{\alpha}u\in
\mathrm{SH}_{\tau(\alpha)},\,u(w)\le \tfrac{1}{2}|w|^2-U^\nu(w)\big\}
\end{equation}
where $\nu=\sum_\lambda\rho_\lambda\delta_{\lambda}$ and $\tau(\alpha)=1-\nu(\G)/\alpha$. 
If we replace the family $\mathrm{SH}_{\tau(\alpha)}$ of subharmonic functions 
in \eqref{eq:U-def-obst} 
by the larger family $\mathrm{SH}$,
we obtain a function $V$ with $V\ge U$. However, 
we have $V=U$ in $\D(0,\sqrt{\alpha})$. Indeed, 
by Proposition~\ref{prop:struct}, the coincidence set
$\{U(w)=\tfrac12|w|-U^\nu(w)\}$ contains some non-trivial annulus 
$\A(\tfrac12\sqrt{\alpha_0},\sqrt{\alpha})$.
Hence, if we set
\[
U_0(z)=\sup\big\{u(z):u\in
\mathrm{SH}(\D(0,\sqrt{\alpha})),\,u(w)\le \tfrac{1}{2}|w|^2-U^\nu(w)\big\},
\]
then by a standard pasting argument $U_0$ may be glued together with 
$U\vert_{\C\setminus\D(0,\sqrt{\alpha})}$ to yield an admissible function
for the original problem over $\mathrm{SH}_{\tau(\alpha)}$,
so we have $U\vert_{\D(0,\sqrt{\alpha})}=U_0$.
As a consequence, $U=U_0\ge V $ on $\D(0,\sqrt{\alpha})$, so 
\[
U(z)=\sup\big\{u(z):u\in\mathrm{SH},\,u(w)\le
\tfrac{1}{2}|w|^2-U^\nu(w)\big\}.
\]
It follows from Theorem~\ref{thm:q-dom} 
that $\Omega$ is a quadrature domain with respect to $\nu$. 
Hence any non-degenerate equilibrium 
measure takes the form claimed by Theorem~\ref{thm:s-qdom}, 
and the proof is complete in the smooth case.

Assume now that $\G$ is a bounded simply connected domain
with piecewise smooth boundary without cusps.
We then approximate $\G$ from within by
$C^2$-smooth domains $\G_t$ (e.g.\ as in \cite[Theorem~1.12]{Verchota}) 
with associated extremal measures $\mu_{0,t}$
of the form
\[
\mu_{0,t}=\tfrac{1}{\alpha}\mathrm{Bal}(\nu_t,\G_t^c) + \tfrac{1}{\pi\alpha}
\chi_{\D(0,\sqrt{\alpha})\setminus\Omega_{\nu_t}}\diffA,
\]
for $\nu_t$ finitely supported atomic measures.
Using convergence properties
of subharmonic functions (see \cite[Theorem 3.2.13]{Convexity})
and standard Hilbert space techniques to ensure $I_\alpha$-minimality 
for weak limit points of $(\mu_{0,t})$
(see \cite[Proposition 3.5]{Brezis}), one concludes that 
$\mu_{0,t}\to\mu_0$ weakly, and that
the sequence $\nu_t$ has a limit point $\nu$ supported on $\mathcal{I}$. 
From this one can conclude that $\mu_{0}$ 
takes the required form
\[
\diff\mu_{0}=\tfrac{1}{\alpha}\mathrm{Bal}(\nu,\G^c) + 
\tfrac{1}{\pi\alpha}\chi_{\D(0,\sqrt{\alpha})\setminus\Omega_{\nu}}\diffA.
\]
The properties of the support of $\nu$ follow from
Sakai's theorem. This completes the proof.
\end{proof}

We mention here that there are indeed holes 
for which the associated forbidden region is a non-trivial 
quadrature domain. Namely, for a hole $\G$ obtained 
by two disjoint disks at an appropriate distance, 
or smashing together two overlapping disks, 
one obtains as forbidden region the so-called Neumann ovals.
This is discussed below, in
\S~\ref{s:examples}.

\section{Global H{\"o}lder regularity of the potential}
\label{s:reg}
\subsection{Control of the density near corner points} 
We continue the study of the minimizer $\mu_0=\mu_{\alpha,\G}$ 
of $I_\alpha$ over $\calM_\G$, 
and denote by $u_0=U^{\mu_{0}}$ its potential.
We assume that $\partial\G$ is piecewise smooth without cusps,
and that $\G\subset\D$. 
We fix $\alpha\ge \alpha_0$ (see Proposition~\ref{prop:struct}).

When the boundary $\partial\G$ is a $C^2$-smooth Jordan curve,
Theorem~\ref{thm:s-qdom} combined with Theorem~\ref{thm:s-ndeg} 
shows that the solution $u_0$ is Lipschitz continuous, 
and even $C^{1,1}(D^\pm)$ for each of the two components $D^\pm$ 
of $\C\setminus \partial\G$. 
Here, we will explore the regularity
properties of $u_0$ when $\partial \G$ is merely known to be piecewise smooth
without cusps. 

We recall the notation $\calE$ for the set of corner points, 
which splits into the set $\calE_0$ of corners with interior angle 
$\pi \sigma\in(0,\pi)$, and the set $\calE_1$ of corners with angle 
$\pi\sigma\in[\pi,2\pi)$. We recall the notation $\calE_0^\eta$,
$\calE_1^\eta$, and $\calE^\eta$ for the fattened singular sets,
and we fix a number $\eta=\eta(\G)>0$ small enough that
the disks $(\D(z,2\eta))_{z\in\calE}$ are disjoint.

We also recall the notation $\mu_0^\s$ for the singular part of $\mu_0$, which
coincides with the restriction $\mu_0\vert_{\partial\G}$.
In Theorem~\ref{thm:s-ndeg} we established that the interior coincidence set
$\calI$ may only approach
the boundary through the set $\calE_0$. Thus, on $\D(z_0,\eta)$ we have
\[
\diff(z,\partial\G)\lesssim |z-z_0|^{\max\{\gamma,1\}},
\qquad z\in\calI\cap\D(z_0,\eta),
\]
where $\frac{1-\sigma}{\sigma}$ and $\sigma=\sigma(z_0)\in(0,1)$ is such that 
the corner at $z_0$ has opening $\sigma\pi$.
By using the structural formula
\[
\mu_0^\s=\frac{1}{\alpha}\mathrm{Bal}(\nu,\G^c)
\]
where $\nu$ is a measure supported on $\calI$, 
we will be able to deduce that $u_0$ is H{\"o}lder continuous outside $\calE_0^\eta$,
and with a more precise argument
we will deduce H{\"o}lder regularity also on $\calE^\eta$.

In order to achieve this, we will control the density of $\mu_0^\s$ with 
respect to arc length as we move towards a singularity $z_0\in\calE_0$.
For this we need the following preliminary result.

\begin{lem}
\label{lem:nu-control-coinc}
For $\partial\G$ a simple closed piecewise smooth curve without cusps, 
it holds that
  \[
    \nu(\D(z,r))\lesssim \alpha r^2,\qquad z\in\calI,\;r\ge 2\diff(z,\partial\G)).
  \]
\end{lem}
\begin{proof}
  Denote by $z_j\in\calI$ a sequence of points converging to $\partial\G$,
  and denote by $r_j$ the associated radii with $r_j\ge 2\diff(z_j,\partial\G)$.
  Taking subsequences, we may assume that $z_j\to z_0$
  for some $z_0\in \partial\G$. In view of Theorem~\ref{thm:s-ndeg}
  we have $z_0\in\calE_0$.
  Now, for any $z\in\D(z_j,r_j)\cap\G$ we must have 
  $\diff(z,\partial\G)\le \tfrac{3}{2}r_j$. We know that a fixed portion
  of $\T(z, 2r_j)$ must lie outside $\G$ for $j$ large enough. 
  To see why, notice first that the closest point $w$ of $\partial\G$ to $z$
  satisfies $|z-w|\le \tfrac{3}{2}r_j$. 
  Since $z_0\in\calE_0$, it is geometrically obvious 
  that $w$ is a regular boundary 
  point ({\em not} a corner) 
  and so the line passing through
  $z$ and $w$ intersects $\partial\G$ orthogonally. 
  But after a rotation and rescaling by $\frac12r_j^{-1}$, 
  this means that $\T(z,2r_j)\cap\G^c$ looks asymptotically like an arc 
  \[
  \big\{z\in \T: \Im z\ge \tfrac{|z-w|}{2r_j}\big\}
  \supset \big\{z\in\T: \Im z\ge \tfrac34 \big\}.
  \]
  The significance of this is that the point of exit from $\D(z,2r_j)$ of 
  a Brownian motion started at $z$ is uniformly distributed on 
  $\T(z,2r_j)$, so at least a fixed portion, say $c_0>0$, 
  of Brownian motions in $\G$ started
  at $\D(z_j,r_j)$ will exit $\partial\G$ already within $\D(z_j,2r_j)$. 
  But taking this together with $\mu_0^\s=\frac{1}{\alpha}\mathrm{Bal}(\nu,\G^c)$ 
  tells us that
  \[
    \mu_0^\s(\D(z_j,2r_j))\ge \frac{c_0}{\alpha}\nu(\D(z_j,r_j)).
  \]
  Finally, in view of Lemma~\ref{lem:neumann} we have the inequality
  \[
    \mu_0(\D(z_j,2r_j))\lesssim r_j^2,
  \]
  from which the claim of the lemma follows.
\end{proof}

We are now ready for the first significant result of this section.

\begin{thm}
\label{thm:density-control}
Let $\alpha\ge \alpha_0$. 
Then, as $z\to z_0\in \calE_0$ with interior angle $\pi\sigma$, 
the density $\rho_0=\alpha\tfrac{\diff\mu_0^\s}{\diffs}$ with respect
to arc-length on $\partial\G$
satisfies
\[
\rho_0(z)\lesssim 
\begin{cases}
|z-z_0|,& 0<\sigma<\frac12\\
|z-z_0|\,\big|\log|z-z_0|\big|,&\sigma=\tfrac12\\
|z-z_0|^{\frac{1-\sigma}{\sigma}},& \frac12<\sigma<1,
\end{cases}
\]
for $z\in\D(z_0,\eta)$.
In particular, the density is bounded on $\D(z,\eta)$.
\end{thm}
\begin{proof}
The density at $z\in\partial\G$ is given by (cf.\ \eqref{eq:bal-dens-nu})
\[
\rho_0(\cdot)=\int_{\calI}\frac{\diff\omega(\xi,\cdot,\Omega)}{\diffs}\diff\nu(\xi)
=\int_{\calI}P_\G(\xi,\cdot)\diff\nu(\xi)
\]
where $P_\G(\xi,z)$ denotes the Poisson kernel 
for $\G$, $\xi\in\G$ and $z\in\partial\G$, 
so that $\xi\mapsto P_\G(\xi,z)$ is harmonic.
Around $z$, we define annular shells 
$\A_j\coloneqq \A(2^{-j}, 2^{-(j-1)})+z$.
By Lemma~\ref{lem:nu-control-coinc} we have $\nu(\A_j)\le \nu(\D(z_j,2^{-(j-1)}))
\lesssim 4^{-j}$.
We estimate
\[
\rho_0(z)\lesssim \sum_{j\ge -K_0}4^{-j}
\sup_{\xi\in\A_j}P_{\G}(\xi,z),
\] 
where $K_0$ is the smallest number such that $\G\subset\D(z,2^{K_0})$.
This is bounded by an absolute constant $K_1$ independent of $z$, 
by compactness of $\calG$.

We split the annular shells into two categories, which we treat in different ways.\\
\\
\noindent {\sc Type I.}\; The first kind consists of annuli $\mathbb{A}_j$ for 
$j$ such that $2^{-j}<2|z-z_0|$. Then the corner point is far enough from $z$
at the scale determined by $\diff(z,\partial\G)$ to have little impact,
and we will estimate harmonic measures from above by harmonic measures
in a half-plane-like regular domain containing $\G$.
The construction is illustrated in Figure~\ref{fig:HalfPlane}.

Specifically, denote by $\Gamma_0$ be the subarc of 
$\partial\G$ containing $z$, which embeds into the open $C^2$-smooth
arc $\tilde\Gamma_0$. We let 
$L$ be a $C^2$-smooth extension of the arc $\tilde\Gamma_0\cap\D(z_0,\eta)$ 
which splits the plane into two unbounded components and keeps $\calG$ 
on one side. 
Denote by $\Omega$ the component of $\C\setminus L$
containing$\G$.
We fix $z_1\in \Omega\setminus\D(z_0,\eta)$ and 
consider a family $f_z$ of conformal maps of $\Omega$
onto the upper half-plane $\HS$ conformally 
with $f_z(z)=0$ and $f_z(z_1)=\imag$.
In view of the smoothness of $L$, the fact that the family $(z)$ 
remains in a smooth subarc of $\partial\Omega$ and Kellog's theorem
(Theorem~\ref{thm:Kellog}), $(f_z)$ is locally uniformly $C^{1,1-\epsilon}$-smooth 
for any $\epsilon>0$.
We have that
\[
P_\G\big(\xi,z\big)\le P_\Omega\big(\xi,z\big)=
P_{\HS}\big(f_z(\xi),0\big)|f_z'(z)|
\]
where we use \eqref{eq:minotonicity-hm} in the first step.
Moreover, by using the explicit form of the Poisson kernel 
$P_{\HS}$ (see e.g. \cite[p.~4]{GarnettMarshall}) 
we find that
\[ 
  P_{\HS}\big(f_z(\xi),0\big)|f_z'(z)|\lesssim 
  \frac{\Im(f_z(\xi))}{|f_z(\xi)|^2}
  \le \frac{1}{|f_z(\xi)|},
\]
where we use the fact that for any $\epsilon>0$, $(f_z)_z$ is a 
uniformly $C^{1,1-\epsilon}$-smooth family of
conformal mappings of $\overline{\Omega}$ onto the 
closed half-plane, so that in particular
we have that $|f_z'(z)|\asymp |f_{z_0}'(z)|$ is 
bounded and bounded away from zero (the lower bound is needed below).
In summary we obtain the bound
\[
P_\G\big(\xi,z\big)\lesssim \frac{1}{|f_z(\xi)|}.
\]
For the range of $j$ considered here, we have for any $\delta_0>0$ that
\[
|f_z(\xi)|=|f_z'(z)(\xi-z)|+\Ordo(|\xi-z|^{2-\delta_0})\gtrsim|\xi-z|
\]
e.g. by Taylor's formula and Kellog's theorem (Theorem~\ref{thm:Kellog} above).
We stress that because for any $\epsilon>0$ the family $(f_z)$ is uniformly 
$C^{1,1-\epsilon}$-smooth, the implicit
constants involved are bounded are independently of $z$.
As a consequence, the Poisson kernel meets the uniform bound
\[
  P_\G(\xi,z)\lesssim |z-\xi|^{-1}
  \lesssim 2^{j},\qquad \xi\in\A_j.
\]
The total contribution to the density 
of these annuli is small:
\[
  \sum_{j\ge |\log_2 2|z-z_0||}4^{-j}
  \sup_{\xi\in\A_j}P_\G(\xi,z)
  \lesssim  
  \sum_{j\ge |\log_2 2|z-z_0||}2^{-j}
  \lesssim |z-z_0|.
\]
\begin{figure}
\includegraphics[width=.5\textwidth]{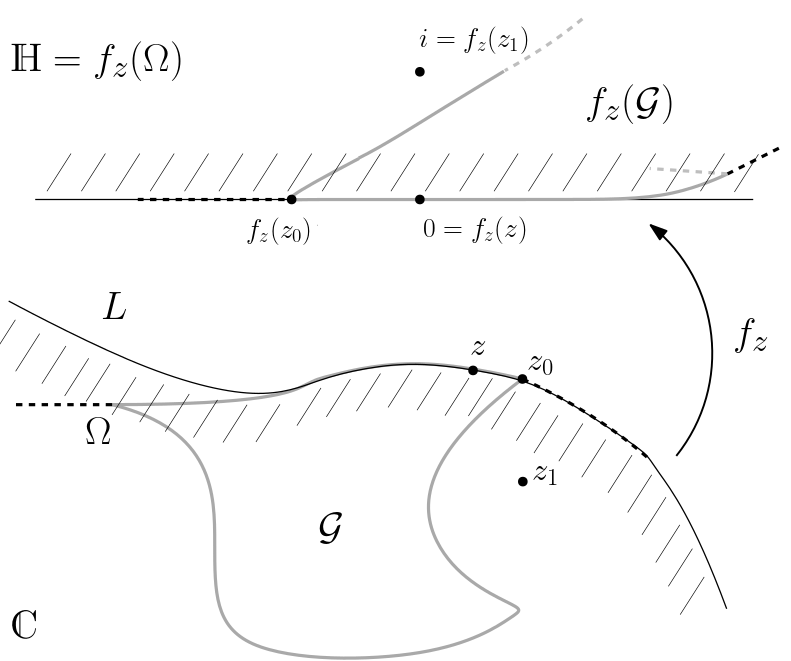}
\caption{The construction of the domain $\Omega$ enclosing $\G$,
and associated conformal maps $f_z:\Omega\to\HS$ with $f_z(z)=0$.}
\label{fig:HalfPlane}
\end{figure}

\noindent {\sc Type II.}\; For the annuli 
$\mathbb{A}_j$ for which $2^{-j}\ge 2|z-z_0|$, 
we need to use the fact that harmonic 
measure decays near corners. We recall (see e.g.\ \cite{GarnettMarshall})
that $P_\G(\xi,z)$ may be expressed in terms of a
conformal mapping $\phi:\G\to\D$ via
\[
  P_\G(\xi,z)=|\phi'(z)|
  \frac{1-|\phi(\xi)|^2}{|\phi(z)-\phi(\xi)|^2}.
\]
and that there exists a constant $a>0$ 
such that
\[
  |\phi(w)-\phi(z_0)| =a|w-z_0|^{\frac{1}{\sigma}}(1+\ordo(1)).
\]
In particular, if $|\xi-z_0|\ge 2|z-z_0|$, 
then by the reverse triangle inequality we also have
\[
  |\phi(z)-\phi(\xi)|\ge (|\phi(\xi)-\phi(z_0)|
  -|\phi(z)-\phi(z_0)|)
  \gtrsim |\xi-z_0|^{\frac{1}{\sigma}}.
\]
Moreover, we have
\[
  1-|\phi(\xi)|^2\lesssim |\xi-z_0|^{\frac{1}{\sigma}}.
\]
Since for $j\le |\log_22|z-z_0||$ we have $|\xi-z_0|\ge 2|z-z_0|$ 
whenever $\xi\in \mathbb{A}_j(z)$,
it holds that
\[
  P_{\G}(\xi,z)\lesssim \frac{|z-z_0|^{\gamma}}{|\xi-z_0|^{\frac{1}{\sigma}}}
  =\frac{|z-z_0|^{\frac{1}{\sigma}}}{|\xi-z_0|^{\frac{1}{\sigma}}}
  \frac{1}{|z-z_0|},
\]
so the contribution to the density of this part of the sum may be bounded 
as follows:
\[
  I(z,z_0)\coloneqq \sum_{j=-K_0}^{|\log_2 2|z-z_0||}\hspace{-5pt}4^{-j}\sup_{\xi\in\A_j}
  \frac{|z-z_0|^{\frac{1}{\sigma}}}{|\xi-z_0|^{\frac{1}{\sigma}}}\frac{1}{|z-z_0|}
  \lesssim |z-z_0|^{\frac{1-\sigma}{\sigma}}\hspace{-16pt}
  \sum_{0\le j\le |\log_22|z-z_0||}\hspace{-10pt}
  2^{\big(\tfrac{1-\sigma}{\sigma}-1\big)j}.
\] 
Computing the sum, using the bound (with $\rho$ fixed)
\[
  \sum_{0\le j\le R}\rho^j\lesssim 
  \begin{cases}
    \rho^R,& \rho>1,\\
    R,&\rho=1,\\
    1,&\rho<1,
  \end{cases}
\]
we find that
\[
  I(z,z_0)\lesssim 
  |z-z_0|^{\frac{1-\sigma}{\sigma}}\times
  \begin{cases}
	1,& \sigma>\tfrac12\\
	\big|\log |z-z_0|\big|,& \sigma=\tfrac12\\
	|z-z_0|^{1-\frac{1-\sigma}{\sigma}},& \tfrac12<\sigma<1.
  \end{cases}.
\]
This completes the proof.
\end{proof}

\subsection{Global H{\"o}lder regularity of the solution}
The main result of this section is now straightforward.

\begin{thm}
\label{thm:Lipschitz-reg}
  Assume that $\G$ is piecewise smooth without cusps.
  Then the potential $U^{\mu_{\alpha,\G}}$ is Lipschitz continuous 
  on $\C\setminus \calE^\eta$ and H{\"o}lder continuous on $\calE^\eta$. 
  In particular $u_0$ is globally H{\"o}lder continuous.
\end{thm}

\begin{proof}
For a given boundary point $z_0\in\Gamma\setminus\calE_0^\eta$, 
the family $(\phi_z)_{z\in\calI}$ 
of Riemann mappings $\phi_z:\G\to \D$ with $\phi_z(z)=0$, 
are all of the same regularity near $z_0$.
On the smooth part $\Gamma\setminus\calE^\eta$, they are all 
uniformly Lipschitz regular in a neighbourhood of $z_0$,
so the potential of harmonic measure seen from any point of $\calI$ 
is Lipschitz regular near $z_0$.
But then the potential of the whole balayage measure $U^{\mu_0^\s}$
is also Lipschitz. 
Similarly, for $z\in\calE_1^\eta$ near a corner 
with opening angle $\pi\sigma\in [\pi,2\pi)$,
the mappings $\phi_z$ are uniformly H{\"o}lder continuous
by Remark~\ref{rem:prop-conf-corner}.

We turn to analyzing the regularity in $\calE_0^\eta$.
  It turns out that we merely need to use that the density $\rho_0$ of
  $\mu_0$ with respect to arc-length is bounded. Indeed, 
  in view of Proposition~\ref{prop:reg} we have for $v_0=\alpha u_0$
  \begin{equation}\label{eq:Lip}
    |v_0(z)-v_0(w)|\lesssim 
  \int_{\partial\G}\Big|\log\big|\tfrac{z-\xi}{w-\xi}\big|\Big|
  \,\rho_0(\xi)\diffs(\xi) + |z-w|,
  \end{equation}
  since the continuous part of the measure $\mu_0$
  has a $C^1$-smooth (in particular 
  Lipschitz continuous) 
  potential.  
  By symmetry, it is enough to show that
  \[
    \int_{\partial\G\cap\{|z-\xi|>|w-\xi|\}}\log\big|\tfrac{z-\xi}{w-\xi}\big|
    \,\rho_0(\xi)\diffs(\xi)\lesssim |z-w|\big|\log|z-w|\big|.
  \]
  Since by Theorem~\ref{thm:density-control} the density 
  $\rho_0$ is bounded on $\D(z_0,\eta)$, we have the uniform bound
  \[
    |\partial\G\cap\D(z,\epsilon)|\lesssim \epsilon,\qquad z\in\D(z_0,\eta),
  \]
  and in addition it holds that
  \begin{multline*}
     \int_{\partial\G\cap\{|z-\xi|>|w-\xi|\}}\log\big|\tfrac{z-\xi}{w-\xi}\big|
    \,\rho_0(\xi)\diffs(\xi)\\
    \lesssim \lVert \rho_0\rVert_{L^\infty(\D(z_0,\eta))}
  \int_{0}^\infty\Big|\partial\G\cap\Big\{\xi:\tfrac{|z-\xi|}{|w-\xi|}>\e^\lambda\Big\}
  \Big|\,\diff \lambda.
  \end{multline*}
  We split the integral into two: one over the set $0\le \lambda \le |z-w|$ and one
  over $|z-w|<\lambda<\infty$. For the former, we simply note that the integrand is
  bounded, so
  \[
  \int_{0}^{|z-w|}\Big|\partial\G\cap\Big\{\xi:\tfrac{|z-\xi|}{|w-\xi|}>\e^\lambda\Big\}
  \Big|\,\diff \lambda \lesssim |z-w|.
  \]
  For the remaining integral, notice that 
  the set $\big\{\xi: \tfrac{|z-\xi|}{|w-\xi|}>\e^\lambda\big\}$ is
  contained in the disk $\D(w,(\e^\lambda-1)^{-1}|z-w|)$, so we find that
  \[
    |v_0(z)-v_0(w)|\lesssim |z-w| + \int_{|z-w|}^\infty|z-w|
    \,\frac{\diff \lambda}{\e^\lambda-1}
  \lesssim |z-w|\big|\log|z-w|\big|,
  \] 
  which shows that for any $\epsilon>0$, the function $u_0$ is H{\"o}lder continuous
  with exponent $1-\epsilon$ on $\calE_0^\eta$. 
This completes the proof.
\end{proof}

\section{Quantitative stability under domain perturbations}
\label{s:perturb-holes-quant}
\subsection{A word on Hausdorff distance and approximation of domains}
\label{ss:approx-dom}
For a domain $\G$ whose boundary is piecewise smooth and without cusps,
we want to consider domains 
which approach $\G$ from within and from the outside, respectively.
We define the Hausdorff distance $\diffH(\G,\G')$ between domains 
$\G$ and $\G'$ as
\[
\diffH(\G,\G')=\max\Big\{\sup_{z\in\G}\inf_{w\in\G'}|z-w|,\;
\sup_{z\in\G'}\inf_{w\in\G}|z-w|\Big\}.
\]
The approximation of $\G$ from within may be accomplished by 
considering level curves $\Gamma_t^-$ of the moduli of
the conformal mapping $\varphi:\G\to \D$:
\[
\Gamma_t^-=\{z\in\C:|\varphi(z)|=\e^{-t}\},\qquad t>0.
\]
Denote by $\G_t^-$ the bounded component of 
$\C\setminus\Gamma_t^-$.
The boundary of the domain $\G_t^-$ is a smooth Jordan curve,
and the domains approach $\G$ well as $t\to 0$, in the sense that 
\[
\diffH(\calG_t^-,\G)\le t^\beta
\]
for some $\beta>0$ which depends only on the angles
of points $z\in\calE$.

\begin{figure}
\includegraphics[width=.45\textwidth]{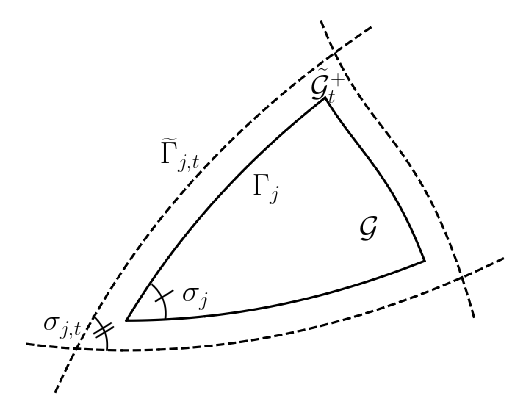}
\caption{The construction for approximating $\G$ from the outside
by $\G_t^+$.}
\label{fig:OuterParallel}
\end{figure}

The approximation of $\G$ from outside 
pertains to obtaining stability bounds for outwards perturbations
of $\G$. These are not needed for our applications, but we include this result for
general interest.
The approximation may be done 
with a regularized outer parallel curve construction, which goes
as follows (see Figure~\ref{fig:OuterParallel}).
For each $C^2$-smooth subarc $\Gamma_j$ of $\partial\G$, let $\tilde{\Gamma}_j$ denote an
open $C^2$-smooth arc containing $\Gamma_j$, and let $\tilde{\Gamma}_{j,t}$ denote 
the curve given by the parameterization ($\epsilon\le s\le 1-\epsilon$)
\begin{equation}\label{eq:param}
(x(s,t),y(s,t))=\Big(x(s)+\tfrac{t y'_\epsilon(s)}{\sqrt{x_\epsilon'(s)^2+y_\epsilon'(s)^2}}, 
y(s)-\tfrac{t}{\sqrt{x_\epsilon'(s)^2+y_\epsilon'(s)^2}}\Big),
\end{equation}
where $(x(s),y(s))$ for $0\le s\le 1$ parametrizes the arc $\tilde{\Gamma}_j$, and
where for a function $F$, $F_\epsilon'(z)$ denotes its averaged derivative 
over the $\epsilon$-neighbourhood of $z$.
If $t$ is small enough (depending only on $\G$) for
\[
\partial_s x(s,t)=x'(s)+\frac{ty_\epsilon''(s)}{((x'_\epsilon)^2
+(y_\epsilon'(s)^2)^{\frac32}}>0,\qquad \epsilon \le s\le 1-\epsilon
\] 
to hold, then $\tilde\Gamma_{j,t}$ is a $C^1$-smooth curve. In fact, and this is the reason
for regularization, $\tilde\Gamma_{j,t}$ is $C^2$-smooth under this condition.
If $\epsilon$
is chosen small enough, then $(x'_\epsilon,-y'_\epsilon)$ points
approximately in the outwards normal direction to $\Gamma$, in which case
the curve $\tilde\Gamma_{j,t}$ lies at a distance of at least, $\frac12t$
from $\G$.

The arcs $\tilde{\Gamma}_{j,t}$ and $\tilde{\Gamma}_{j+1,t}$ 
meet with angles $\sigma_t(j)$
at corners, which are in a one-to-one correspondence with those of $\partial\G$.
That is, $\sigma_t(j)=\sigma(\Gamma_j,\Gamma_{j+1})+\ordo(1)$ as $t\to 0^+$, 
where $\sigma(\Gamma_j,\Gamma_{j+1})$ denotes the angle formed by the subarcs $\Gamma_j$ and
$\Gamma_{j+1}$ of $\partial\G$.
The resulting piecewise smooth Jordan curve $\Gamma_t$ encloses a piecewise smooth 
Jordan domain $\G_t^+$ without cusps (for small enough $t$).
In addition, if $w_t=w_{t,j}$ denotes the parameterization \eqref{eq:param},
then it is clear that $w_t''\to w_0''$ e.g.\ in $L^1(0,1)$ 
(cf. Remark~\ref{rem:prop-conf-corner-stab}).

\subsection{Stability of the Zeitouni-Zelditch functional}
\label{ss:quantStab}
Assume that $\G_0$ is a bounded simply connected 
domain whose boundary $\partial\G$ is piecewise smooth
without cusps.
Denote by $(\G_t)_{t>0}$ a family
of simply connected Lipschitz domains, 
such that $\G_t\to\G$ in the sense
of convergence in Hausdorff distance as 
$t\to0$, and normalized such that
\[
\diffH(\G_t,\G_0)\le t,\qquad t>0.
\] 
In this section, we quantify how well the 
equilibrium measures $\mu_t=\mu_{\alpha,\G_t}$ approximates
$\mu_0$ in the sense of the functional $I_\alpha$.

\begin{prop}
\label{prop:quant-stab} 
Under the above conditions, it holds that
\[
I_\alpha(\mu_t)= I_\alpha(\mu_0)+\Ordo(t^\beta),
\]
for some constant $\beta>0$ depending only on $\G_0$, as $t\to 0$.
\end{prop}

The number $\beta$ can be explicitly related 
to angles at the corner points $z\in \calE$,
but we refrain from doing so for reasons of length.

\begin{proof} The proof is rather long, and is split 
into steps for the reader's convenience. We treat separately the cases when $\G_t\subset \G_0$
and when $\G_0\subset\G_t$, and finally argue that this implies the general result.

\medskip

\noindent {\sc Step 1.}\; We first look at the case when $\G_0\subset \G_t$ for $t>0$, 
so that in particular $\mu_t$ is
a competitor for the minimization problem defining $\mu_0$. 
We may, by monotonicity of $I_\alpha(\mu_{\alpha.\G})$ 
with respect to domain inclusion,
assume that $\G_t$ is enclosed by curves $\Gamma_t^+$ constructed as 
in \S~\ref{ss:approx-dom}.
By sweeping the measure $\chi_{\G_t}\mu_{0}$
to $\G_t^c$, we moreover obtain a good upper bound for 
$I_\alpha(\mu_t)$ while $I_\alpha(\mu_0)$ supplies a lower bound
\[
I_\alpha(\mu_0)\le I_\alpha(\mu_t)\le I_\alpha(\mu_t^*)
\] 
where $\mu_t^*=\mathrm{Bal}(\mu_0,\G_t^\c)$,
and it remains to explain how the right-hand side may be estimated.
We recall that $\mathrm{Bal}(\mu_0,\G_t^c)$ is interpreted as
\[
\mathrm{Bal}(\mu_0,\G_t^c)=
\mathrm{Bal}(\chi_{\G_t}\mu_0,\G_t^c) + \chi_{\G_t^c}\mu_0.
\]
The balayage operation decreases the 
energy $-\Sigma(\mu)$ of a measure $\mu$, 
so it suffices to estimate
$B_\alpha(\mu_t^*)$. We claim that
there exists a positive number $\beta>0$ and a 
constant $C<\infty$ depending only on $\G$ such that
\[
0\le \sup_{z\in\G_t}U^{\mu_t^*-\mu_0}(z)\le C t^\beta,
\]
as $t\to0$. Notice first that the potential of the balayage measure 
$\eta_{t,\zeta}=\mathrm{Bal}(\delta_\zeta,\G_t^c)$ is given by (cf.\ \eqref{eq:bal-pot-nu})
\[
U^{\eta_{t,\zeta}}(z)=\log|z-\zeta|-g_{\G_t}(z,\zeta),
\]
where $g_{\G_t}$ is the Green function for $\G_t$,
and hence we may write
\[
U^{\mu_t^*-\mu_0}(z)=
\int_{\G_t\setminus\G_0} (U^{\eta_{t,\zeta}}(z)-\log|z-\zeta|)\diff\mu_0(\zeta)=
-\int_{\G_t\setminus\G_0} g_{\G_t}(z,\zeta)\diff\mu_0(\zeta).
\]
That this identity holds may moreover be seen by 
noting that $U=U^{\mu_t^*-\mu_0}$ is the unique solution
to the boundary value problem
\[
\begin{cases}
\Delta U=-\mu_0& \text{in }\,\G_t\\
U=0&\text{on }\,\partial\G_t.
\end{cases}
\]
We let $E(z,\lambda,t)$ denote the sub-level set
\[
E(z,\lambda,t)=
\big\{\zeta\in\G_t\setminus\G_0:g_{\G_t}(z,\zeta)\le -\lambda\big\}
\]
and find that
\begin{equation}\label{eq:basic-bound-Lebesgue}
\sup_{z\in\G_t}U^{\mu_t^*-\mu_0}(z)
=\sup_{z\in\G_t\setminus\G_0}
\Big(\int_{0}^\infty\mu_0\big(E(z,\lambda,t)\big)
\diff \lambda\Big),
\end{equation}
where we have used the maximum principle to restrict 
the supremum to $\G_t\setminus\G_0$ and 
the layer cake formula to obtain the integral
on the right-hand side.

We proceed to estimate the size of $E(z,\lambda,t)$, 
and to this end we use the fact that $\G_t$ is simply connected.
Hence, there exists a conformal mapping $\varphi_t:\G_t\to\D$, 
which takes a fixed interior point $z_0$ of $\G$
to the origin. The Green function is then given by
\[
-g_{\G_t}(z,\zeta)=
\log\left|\frac{1-\overline{\varphi_t}(z)\varphi_t(\zeta)}
{\varphi_t(\zeta)-\varphi_t(z)}\right|,
\qquad (z,\zeta)\in \G_t\times \G_t.
\]
By construction, the domains $\G_t$ approximate $\G$ well in the sense
of Remark~\ref{rem:prop-conf-corner-stab}, so that by Remark~\ref{rem:prop-conf-corner},
there exist positive constants $c_1, c_2$, $\beta_0$ and $\beta_1$ with $\beta_0>\beta_1$
depending on $\G$ but not on $t$, such that
\[
c_1|z-\zeta|^{\beta_0}\le 
|\varphi_t(\zeta)-\varphi_t(z)|\le c_2|z-\zeta|^{\beta_1},
\qquad (z,\zeta)\in \G_t\times\G_t
\]
as $t\to0^+$.
A point $\zeta\in\G_0$ belongs to $E(z,\lambda,t)$ if and only if
\begin{equation}\label{eq:phi-bdd-1}
\left|\frac{1-\overline{\varphi_t}(z)\varphi_t(\zeta)}
{\varphi_t(\zeta)-\varphi_t(z)}\right|> \e^\lambda.
\end{equation}
But for $z\in \G_t\setminus\G_0$ 
we have the bound $1-|\varphi_t(z)|^2=\Ordo(t^{\beta_1})$, so 
by the triangle inequality it holds that
\[
\left|\frac{1-\overline{\varphi_t}(z)\varphi_t(\zeta)}
{\varphi_t(\zeta)-\varphi_t(z)}\right|= 
\left|\frac{1-|\varphi_t(z)|^2 + 
\overline{\varphi_t}(z)\big(\varphi_t(z)-\varphi_t(\zeta)\big)}
{\varphi_t(\zeta)-\varphi_t(z)}\right| 
\le 1+\frac{Ct^{\beta_1}}{|z-\zeta|^{\beta_0}}.
\]
In view of this estimate, the inequality \eqref{eq:phi-bdd-1} implies that
\[
  |z-\zeta|\lesssim t^{\beta_1/\beta_0}(\e^{\lambda}-1)^{-\frac{1}{\beta_0}},
\]
for some $\beta_2>0$.
It follows that $E(z,\lambda,t)\subset 
\D(z,D_0t^{\beta_2}(\e^\lambda-1)^{-\frac{1}{\beta_0}})$ 
for some constant $D_0$. We split the integral in $\lambda$ over two regions:
one where $\lambda<t^\tau$ and one where $\lambda>t^\tau$, for a constant
$\tau>0$ to be determined. The first is we estimate as
\[
\int_0^{t^\tau}\mu_0\Big((\G_t\setminus \G_0)\cap
\D(z,D_0t^{\beta_2}(\e^{-\lambda}-1)^{-\frac{1}{\beta_0}}\Big)\diff\lambda
\le \int_0^{t^\tau}\mu_0(\G_t)\diff\lambda \lesssim t^\tau,
\]
where we use the fact that $\mu_0$ is a probability measure.
Using Theorem~\ref{thm:density-control} to control the singular
part of the measure, and Proposition~\ref{prop:reg} to control the continuous part,
we find that it holds that
\[
\mu(\D(z,\rho))\lesssim \rho^{\beta_4},\qquad \text{for all }\,z\in\partial\G, \;\rho>0
\]
for some $\beta_4>0$.
Hence, for the second integral (over the region $\lambda>t^\tau$), we estimate
\[
\int_{t^\tau}^\infty\mu_0\Big(\D\big(z,D_0t^{\beta_2}
(\e^\lambda-1)^{-\frac{1}{\beta_0}}\big)\Big)\diff \lambda
\lesssim \int_{t^{\tau}}^\infty t^{\beta_2\beta_3}
(\e^\lambda-1)^{-\frac{\beta_3}{\beta_0}}\,\diff\lambda
\]
for some $\beta_3$.
A simple change of variables shows
that we have
\[
\int_{t^\tau}^\infty (\e^\lambda-1)^{-\frac{\beta_3}{\beta_0}}\,\diff\lambda
\lesssim t^{-\frac{\tau\beta_3}{\beta_0}},
\]
so we may estimate the integral by
\[
\int_{t^\tau}^\infty\mu\big(\D(z,D_0t^{\beta_2}
(\e^\lambda-1)^{-\frac{1}{\beta_0}})\big)\diff \lambda\lesssim 
t^{{\beta_3\beta_2}-\frac{\tau\beta_3}{\beta_0}}=t^{\beta_4}
\]
where $\beta_4>0$ provided that $\tau$ is chosen small enough.
We find that
\[
\sup_{z\in\G_t\setminus \G_0}|U^{\mu_t^*-\mu_0}(z)|\le t^\beta
\]
for some $\beta>0$. It follows that 
\[
B_\alpha(\mu_t)=B_\alpha(\mu_0)+\Ordo(t^\beta), 
\]
whenever $\G_t$ approximates $\G_0$ from the outside, which completes the proof.

\medskip

\noindent {\sc Step 2.}\; We turn to the case when $\G_t\subset\G_0$ for all $t>0$. 
By monotonicity
with respect to domain inclusions, we may without loss
of generality fix $\varphi$ to be a conformal mapping
$\varphi:\G_0\to\D$, and make the specific choice 
$\G_t=\varphi^{-1}(\e^{-t^\tau}\D)$ for some 
appropriate choice of a constant $\tau>0$.

We set $\mu_t^*=\mathrm{Bal}(\mu_t,\G_0^c)$, and notice that 
$I_\alpha(\mu_{\alpha,\G_0})$ is sandwiched in between $I_{\alpha}(\mu_t)$ and
$I_\alpha(\mu_t^*)$ by monotonicity in the domain $\G_0$ and optimality of 
$\mu_{\alpha,\G}$ for the minimization of $I_\alpha$. 
Hence, we need to bound the difference 
\[
I_\alpha(\mu_t^*)-I_\alpha(\mu_t),
\]
by a positive power of $t$. 
As before, the balayage operation only decreases the energy $-\Sigma(\mu)$,
so it suffices to estimate $U^{\mu_t^*-\mu_t}(z)$ for
$z\in\G_0$. We may assume that $\mu_t^\c(\G_0)=0$, since otherwise
we split the measure $\mu_t$ into its singular and
continuous parts, and treat the latter separately as follows: 
By the maximum principle, it is enough to consider the Balayage
potential for $z\in\G_0\setminus\G_t$, and for such $z$ we clearly have
\[
0\le U^{\mathrm{Bal}(\mu_t^{\c},\G_0^c)-\mu_t^\c}(z)\le
\frac{1}{\pi\alpha}
\int_{\G_0\setminus\G_t}|g_{\G_0}(z,w)|\,\diffA(w) \lesssim t^\beta
\]
by mirroring the computations from the first part of the proof.

Hence, $U^{\mu_t}$ may be assumed to be harmonic in $\G_0\setminus \partial\G_t$,
so by the maximum principle, it is sufficient to estimate $U^{\mu_t^*-\mu_t}$
for $z\in \G_0\setminus\G_t$. In fact,
by an additional application of the maximum principle we 
find that it is enough to estimate this for 
$z\in\partial\G_t$.
We use the fact that $\mu_t^\s=\mathrm{Bal}(\nu_t,\G_t^c)$ 
for some atomic measure $\nu_t$, finitely supported on 
$\G_t$. Moreover,
$\mu_t^*=\mathrm{Bal}(\mu_t,\G^c_0)$. 
Therefore, the potentials of $\mu_t^\s$ and $\nu_t$ agree outside
$\G_0$, so in particular their balayage measures to $\G^c_0$ are equal:
\[
\mu_t^*=\mu_t^\c+\mathrm{Bal}(\mu_t^\s,\G^c_0)=\mu_t^\c+\mathrm{Bal}(\nu_t,\G^c_0).
\]
As a consequence, we find that
\[
U^{\mu_t^*-\mu_t}(z)=
U^{\mathrm{Bal}(\nu_t,\G_0^c)-\mathrm{Bal}(\nu_t,\G_t^c)}(z)=\int 
\big(g_{\G_0}(z,w)-g_{\G_t}(z,w)\big)\diff\nu_t(w).
\]
Due to the special choice of approximating domains $\G_t$, the
conformal mapping $\varphi_t$ of $\G_t$ onto $\D$ is given by
$\varphi_t(z)=\e^{t^\tau}\varphi(z)$ for an appropriate positive parameter $\tau$, so 
we find that for $z\in\partial\G_t$ and $w\in\G_t$ we have
\begin{multline*}
|g_{\G_0}(z,w)-g_{\G_t}(z,w)|=
\log\left|\frac{1-\e^{-2t^\tau}\overline{\varphi(w)}\varphi(z)}
{1-\overline{\varphi(w)}\varphi(z)}\right|- t^\tau\\
\lesssim \log\left(1 + \frac{t^{\tau}}
{|1-\overline{\varphi(w)}\varphi(z)|}\right)+\Ordo(t^{\tau}).
\end{multline*}
If $m_\xi$ denotes the M{\"o}bius transformation,
the conformal mapping $\varphi_w(z)=m_{\varphi(w)}(\varphi(z))$ 
maps $\G_0$ into $\D$, so the ratio 
$|\varphi(z)-\varphi(w)|/|1-\bar{\varphi}(w)\varphi(z)|$ is bounded by one. 
From this we may conclude that
\[
|1-\bar{\varphi}(w)\varphi(z)|\ge
|\varphi(z)-\varphi(w)|\gtrsim |z-w|^{\beta_1},
\]
where the last inequality follows from 
the assumed regularity of the boundary $\partial \G_0$ and Remark~\ref{rem:prop-conf-corner}.
By possibly changing the value of $\beta$, we find that
\[
1 + \frac{t^{\tau}}
{|1-\varphi(z)\bar{\varphi}(w)|}\lesssim 
1+\frac{t^\beta}{|z-w|^\beta},\qquad z\in\partial\G_t.
\]
Next, we recall Lemma~\ref{lem:nu-control-coinc}, 
which says that there exists some constant 
$C_1$ (universal) so that
\[
\nu_t(\D(z,r))\le C_1 r^2,\qquad r\ge 2\mathrm{d}(z,\partial\G_t),\;z\in\calI_t.
\]
We next split the integral against $\nu_t$ as follows: 
for each $j\ge 0$, we let $r_j=2^{-j}\mathrm{diam}(\G_t)$ and
set $A_j(z)=z+\mathbb{A}(r_j,r_{j+1})$. Then, 
whenever $A_j(z)\cap \calI_t$ contains some point $z_j$, we have 
\[
\nu_t(A_j(z))\le \nu_t(\D(z_j,r_j))\le C_1 r_j^2,
\]
and if no such $z_j$ may be found, it holds that $\nu_t(A_j(z))=0$.
We may then estimate the sought-after quantity:
\begin{multline*}
U^{\mu_t^*-\mu_t}(z)\le\sum_{j\ge 0}\nu_t(A_j(z))\sup_{w\in A_j(z)}
(g_{\G_0}(z,w)-g_{\G_t}(z,w))\\
\le  C_2\sum_{j\ge 0}\log\Big(1+\frac{t^\beta}{2^{-\beta j}}\Big)2^{-2j}+
\Ordo(t^\tau),
\end{multline*}
where $C_2$ is some constant; 
this is easily seen to be of order 
$\Ordo(t^\beta)$ as $t\to 0$ for $\beta$ chosen small enough. 
This completes the proof also in this case.

\medskip

\noindent {\sc Step 3.}\; Finally, we let $\G_t$ be any family 
of smooth domains at Hausdorff distance
at most $t$ from $\G_0$. We can then find domains $\G_t^\pm$ by the same
process as above, with
\[
\G_t^-\subset\G_0,\qquad
\G_t\subset\G_t^+
\] 
with corresponding extremal measures $\mu_t^\pm$ for $I_\alpha$, such that
$I_\alpha(\mu_0)$ and $I_\alpha(\mu_t)$ are both bounded between
$I_\alpha(\mu_t^-)$ and $I_\alpha(\mu_t^+)$:
\[
I_\alpha(\mu_t^-)\le I_\alpha(\mu_0),\;
I_\alpha(\mu_t)\le I_\alpha(\mu_t^+),
\]
with $\diffH(\G_t^\pm,\G_0)\le t$,
and we may apply the above approach to the measures $\mu_t^\pm$ to conclude
\[
I_\alpha(\mu_{t})=I_\alpha(\mu_0)+\Ordo(t^\beta),\qquad t\to 0,
\]
which completes the proof.
\end{proof}

\section{The limiting conditional zero distribution}
\label{s:Weyl}
\subsection{The general main theorem for random zeros}
For the formulation of the main theorem, we need several notions. We denote by
\[
\mu_L^{\C} = \sum_{ w\,:\, F_L(w) = 0 } \delta_w
\]
the empirical measure of the zeros of $F_L$, and by 
$\mu_{L,\G}^{\C}$ the same measure conditioned on the 
hole event $\calH_L(\G) = \{ F_L \ne 0 \mbox{ in } \G \}$.

We define the \emph{Schwarz potential} $u_\G$ with respect to the
quadrature domain $\Omega_\nu$ as the solution to the boundary value problem
\[
\begin{cases}
\Delta u_\G=1-\nu&\text{on }\;\Omega_\nu;\\
u_\G=0&\text{on }\,\partial\Omega_\nu.
\end{cases}
\]

\begin{thm}\label{thm:main-GEF2}
Let $\G$ be a Jordan domain with piecewise $C^2$-smooth boundary without cusps
and let $\mu_{\G}^\C$ be given by
\[
\diff\mu_\G^\C=\diff\mathrm{Bal}(\nu,\G^c)
+ \tfrac{1}{\pi}\chi_{\C\setminus\Omega_\nu}\diffA,
\]
where $\nu$ and $\Omega_\nu$ are as in Theorem~\ref{thm:s-qdom}. 
Then the empirical measures $L^{-2} \mu_{L,\G}^{\C}$ converge to $\mu_{\G}^\C$ vaguely in distribution
as $L\to\infty$.
In addition, the hole probability satisfies
\[
\frac{1}{L^4}\log \mathcal{P}(\calH_L(\G))=
-\frac{1}{\pi\alpha^2}\int_{\Omega_\nu}u_\G \, \diffA+\Ordo(L^{-2}\log L^2)
\]
as $L\to\infty$.
\end{thm}

Let $\varphi$ be a continuous test function with compact support. We write
\[
n_L(\varphi) = \sum_{w : F_L(w) = 0} \varphi(w),
\]
for the \emph{linear statistic} of the zeros of $F_L$ with 
respect to $\varphi$. We also denote by $\mathcal{P}_L^\G$ the 
probability measure conditioned on the event $\calH_L(\G)$, and 
for a compactly supported smooth test function we write
\[
\dir(\varphi) = \int_\C |\nabla \varphi|^2 \, \diff A.
\]

Vague convergence in distribution of the conditional empirical 
measures $\mu_{L,\G}^{\C}$ is equivalent to the convergence in 
distribution of the random variables $L^{-2} n_L(\varphi)$ to the 
limit $\int \varphi \, \diff \mu_{\G}^\C$. Using the same general 
approach as in \cite{GhoshNishry1}, we prove Theorem \ref{thm:main-GEF2} 
by bounding from above the probability that, conditional on the hole 
event $\calH_L(\G)$, the linear statistic $n_L(\varphi)$ is far 
from $L^2 \int \varphi \, \diff \mu_{\G}^\C$.

Let $\epsilon > 0$. For a compactly supported smooth test 
function $\varphi$, we prove in this section the (conditional) large deviation upper bound
\[
\mathcal{P}_L^\G\Big(\Big\{
\big| n_L(\varphi) - L^2 \begingroup\textstyle\int\endgroup 
\varphi \, \diff \mu_\G^\C \big| > \epsilon L^2 \Big\}\Big)\le 
\exp\Big(-\frac{c\epsilon^2L^4}{\dir(\varphi)}+\Ordo\big(L^2\log^2 L\big)\Big)
\]
as $L\to\infty$.

\subsubsection{Negligible events}
In the proof of Theorem \ref{thm:main-GEF2} we may assume without 
loss of generality that $\G$ is contained in the unit disk. We will 
say that an event is \emph{negligible} (with respect to the hole probability, 
depending $L$) if its probability is at most $\exp(-2 L^4)$. 
The precise constant $2$ is not important, but we do use the fact 
that the hole probability for $\G$ is bounded from below by the 
hole probability for $\D$, which decays asymptotically like $\exp(-\tfrac{e^2}{4} L^4)$.

\begin{rem}
Notice that the union of polynomially many (in $L$) negligible 
events consists of a negligible event, for $L$ large.
\end{rem}

\subsection{A guide to the proof of Theorem~\ref{thm:main-GEF2}}
\label{ss:guide:probab}
Roughly, the proof may be split into four steps:

\begin{center}
\begin{tabularx}{\textwidth}{lX}
{\sc Step 1.} & A truncation argument, which replaces
the GEF $F_L$ by a polynomial.\\
{\sc Step 2.} & Obtaining a lower bound for the hole probability.\\
{\sc Step 3.} & Deriving an effective large deviation upper bound for 
linear statistics of the zeros. This also provides an upper bound for the hole probability.\\
{\sc Step 4.} & Deducing the convergence of conditional empirical measures
\end{tabularx}
\end{center}

{\sc Steps 1, 3} 
and {\sc 4} are very similar
to the 
corresponding arguments in \cite{GhoshNishry1}. 
The use of truncation in {\sc Step 1} leads 
to small (random) perturbation in the location of the zeros 
of the polynomial compared with those of $F_L$. The technical 
difficulties induced by this perturbation are rather mild in 
\cite{GhoshNishry1}, since there the domain $\G$ is a disk 
(which, by its convexity, is stable under small perturbations).
Here we have to rely instead on our 
quantitative stability results from 
\S~\ref{s:perturb-holes-quant} (these are also used in {\sc Step 2}).

In \cite{GhoshNishry1} the lower bound for the probability of 
the hole event is obtained explicitly, by constructing an appropriate 
event in terms of the random variables $\{\xi_n\}$. This construction 
depends crucially on the fact that the domain $\G$ is a disk 
(by using the circular symmetry of Taylor series). Our {\sc Step 2} 
requires a completely different approach, which is based on 
discretization of the continuous minimizer $\mu_{\alpha, \G}$ of the functional $I_\alpha$.

\subsection{Truncation of the power series}
Since it is difficult to handle directly the zeros of the GEF $F_L$, 
we first approximate (some of) them by zeros of the \emph{Weyl polynomial}
$$
P_{N,L}(z) = \sum_{n=0}^N \xi_n\frac{(Lz)^n}{\sqrt{n!}},\qquad z\in\C.
$$
For this to work we need to control the size of
\[
T_N(z) = T_{N,L}(z) = \sum_{n=N+1}^\infty \xi_n \frac{(L z)^n}{\sqrt{n!}}, \qquad z \in \C,
\]
that is, the tail of the Gaussian Taylor series $F_L$. Note that $P_{N,L}$ 
and $T_N$ are independent Gaussian analytic functions.
We use the following crude bound for the tail (e.g. \cite[Lemma 3.3]{GhoshNishry1}).
\begin{prop}\label{prop:T_N_bound}
Let $A, B \ge 1$ be fixed parameters. Let $L > 0$ be sufficiently 
large and put $N = \left\lceil L^2 \log L\right\rceil$. 
With probability at least $1 - \exp(-C L^6)$ we have
\[
\sup_{z \in \D(0,B)}|T_N(z)| \le \exp( -A L^2 \log L ),
\]
where $C>0$ is an absolute constant.
\end{prop}

\begin{rem}
A small simplification compared with \cite{GhoshNishry1} is 
that we choose the parameter $N$ to be non-random.
\end{rem}

In order to control the perturbation of the zeros of $F_L$ we will 
apply Rouch\'e's theorem.
Thus, we need a lower bound for $|F_L|$ 
away from its zeros. Let $H$ be an entire function, and 
$B > 0, \eta \in (0, \tfrac14]$ real parameters. Denote by 
$w_1, \dots, w_m$ the zeros of $G$ in $\D(0, B)$, including multiplicities. We define
\[
m_H(B;\eta) = \min\big\{ |H(z)| : z \in \D(0,B) , \, 
\mathrm{d}(z, w_j) \ge \eta \mbox{ for all } j \in \{1, \dots, m\} \big\}.
\]
The following result is obtained by combining Lemmas~3.5 and 3.6, 
Corollary~3.4, and Theorem~3.7 from \cite{GhoshNishry1}.
\begin{prop}\label{prop:lower_bnd_away_from_zeros}
Let $B, \tau \ge 1$ be fixed parameters, and $L > 0$ be 
sufficiently large. Then, with probability at least 
$1 - \exp(-C B^4 L^4)$, we have that
\[
m_{F_L}(B, L^{-\tau}) \ge \exp(-C \tau B^2 L^2 \log L),
\]
where $C>0$ is an absolute constant.
\end{prop}

\subsection{Joint density of the zeros of $P_{N,L}$} \label{ss:joint-density-and-heuristic}
Thinking for the moment of $N \in \N, L > 0$ as free parameters, we write
\begin{equation}\label{eq:P-N,L-in-prod}
P_{N,L}(w) =\sum_{n=0}^N \xi_n\frac{(Lw)^n}{\sqrt{n!}} 
= \xi_N \frac{L^N}{\sqrt{N!}} \prod_{j=1}^N (w - z_j) 
\eqqcolon \xi_N \frac{L^N}{\sqrt{N!}} Q_{\bf z}(w).
\end{equation}
By a change of variables (see e.g. \cite[Appendix A, Lemma A.1]{GhoshNishry1}) 
the joint density of the zeros with respect to the product 
measure $\diffA({\bf z})\coloneqq \diffA^{\otimes N}(z_1,\ldots,z_N)$ takes the form
\begin{equation}\label{eq:joint-density}
f_{N,L}({\bf z})=\frac{1}{\calZ_{N,L}} |\Delta(\textbf{z})|^2
\Big(\frac{L^2}{\pi} \int_\C |Q_{\bf z}(w)|^2 \e^{-L^2|w|^2}
\diffA(w)\Big)^{-(N+1)},
\end{equation}
where $|\Delta(\textbf{z})| = \prod_{i<j}|z_i-z_j|$ is the 
Vandermonde determinant, and $\calZ_{N,L} 
= \pi^N L^{N(N+1)} N!^{-1} \prod_{k=1}^N (k!)^{-1}$ is a normalizing constant.
In alignment with \cite{GhoshNishry1}
we use the notation
\[
S({\bf z})=\frac{L^2}{\pi} \int_{\C}|Q_{\bf z}(w)|^2 \, \e^{-L^2|w|^2}\diffA(w).
\]

\begin{rem}
Henceforth, the empirical probability measure of the 
points $\textbf{z} = (z_1, \dots, z_N)$ will be denoted by
\[
\mu_{\textbf{z}} = \frac1N \sum_{j=1}^N \delta_{z_j},
\]
where $\delta_{z_j}$ is a point mass at $z_j$.
\end{rem}

\subsection{The conditional limiting measure}
Put $\alpha = N L^{-2}$. Recall that,
\[
I_{\alpha}(\mu) = -\Sigma(\mu) + 2 B_\alpha(\mu) = -\Sigma(\mu)+2\sup_{z\in\C}\big(U^\mu(z)
-\tfrac{|z|^2}{2\alpha}\big),
\]
and denote by $\mu_{\alpha} = \chi_{\D(0,\sqrt{\alpha})}\, \diffA$ 
the global minimizer of $I_\alpha$.  Also recall that the minimum 
value of $I_\alpha$ over the class of probability measures 
$\calM_{\G}$ which charge zero mass to $\G$ is attained uniquely at $\mu_{\alpha, \G}$.

Roughly speaking, the 
idea of Zeitouni and Zelditch in \cite{ZZ} is to 
approximate the Vandermonde term $|\Delta(\textbf{z})|^2$ in 
\eqref{eq:joint-density} by $\exp(N^2 \Sigma(\mu_{\textbf{z}}))$ (appropriately regularized), 
and to replace $S({\bf z})$ by $\exp(2 N B_\alpha(\mu_{\textbf{z}}) )$ 
(for a more precise statement, see e.g. Proposition \ref{prop:approx-Fekete-points}).

Looking at \eqref{eq:joint-density} at the logarithmic scale, 
and expressing the normalizing constant $\calZ_{N,L}$ in terms 
of $I_\alpha(\mu_\alpha)$, the probability of the hole event 
in $\G$ for $P_{N,L}$ 
is equal, up to smaller error terms, to the maximum of 
$-N^2(I_{\alpha}(\mu) - I_{\alpha}(\mu_\alpha))$ over 
$\mu \in \calM_{\G}$, that is (by Lemma \ref{lem:cost_of_hole_event} below) 
to $-L^4\int_{\Omega_\nu}u_{\G} \, \frac{\diffA}{\pi}$. 
Moreover, the probability of zero configurations which are 
not `close' to $\mu_{\alpha, \G}$ is negligible with 
respect to the hole probability, so that (following \cite{GhoshNishry1}), we show 
that the limiting measure of the zeros of $F_L$ on the 
hole event in $\G$ is given by the Radon measure
\begin{equation}\label{eq:limit_measure_def}
\mu_\G^\C \coloneqq \lim_{\alpha\to\infty} \alpha \, \mu_{\alpha, \G}.
\end{equation}
That the above limit exists can be seen by appealing to 
Proposition \ref{prop:glue}. Moreover, by that proposition we see that
\[
\mu_\G^\C = \alpha_0 \, \mu_{\alpha_0,\G} + \tfrac{1}{\pi}
\chi_{\{z \, : \, |z| \ge \sqrt{\alpha_0}\}} \diffA.
\]
Recall that there is a finite measure $\nu = \nu_G$ supported in $\G$, such that
\[
\mu_{\alpha,\G}=\tfrac{1}{\alpha}\mathrm{Bal}(\nu,\G^c)
+\tfrac{1}{\pi\alpha}\chi_{\D(0,\sqrt{\alpha})\setminus \Omega_\nu}\diffA,
\]
where $\Omega_\nu$ denotes the subharmonic quadrature
domain with respect to $\nu$ (which contains $\G$). Define the 
\emph{Schwarz potential} $u = u_\G$ associated to the data $(\nu,\Omega_\nu)$ 
as the unique solution to the PDE,
\[
\begin{cases}
\tfrac{1}{2\pi}\Delta u = 1-\nu& \text{on }\Omega_\nu, \\
u=0& \text{on }\partial\Omega_\nu.
\end{cases}
\]

\begin{lem}\label{lem:cost_of_hole_event}
With the above definitions, we have
\[
\alpha^2 \big( I_\alpha(\mu_{\alpha,\G})-I_\alpha(\mu_{\alpha}) \big)
= \frac{1}{\pi}\int_{\Omega_\nu}u_{\G} \, \diffA.
\]
\end{lem}
\begin{proof}
We use the notation
\[
\Sigma(\mu,\nu)=\int U^\mu \, \diff\nu=\int U^\nu \, \diff\mu,
\]
provided that $U^\mu\in L^1(\nu)$ and vice versa.
It then holds that
\begin{align}
I_\alpha(\mu_{\alpha,\G})-I_\alpha(\mu_\alpha)
=\Sigma(\mu_\alpha)-\Sigma(\mu_{\alpha,\G})
&=\Sigma(\mu_\alpha) - 
\Sigma\big(\mu_{\alpha,\G},\mu_{\alpha,\G}^\s
+\tfrac{1}{\alpha}\chi_{\D_{\sqrt{\alpha}}\setminus \Omega_{\nu}}\big)\\
&=\Sigma(\mu_\alpha) - 
\Sigma\big(\mu_{\alpha,\G},\tfrac{1}{\alpha}\nu
+\tfrac{1}{\alpha}\chi_{\D_{\sqrt{\alpha}}\setminus \Omega_{\nu}}\big),
\end{align}
where the last equality holds because 
$\mu_{\alpha,\G}^\s=\tfrac{1}{\alpha}\mathrm{Bal}(\nu,\G^c)$
and since the potential is harmonic across $\G$. 
Next, observe that $U^{\mu_{\alpha,\G}}=U^{\mu_\alpha}$
on the support of the measure 
$\nu+\chi_{\D_{\sqrt{\alpha}}\setminus \Omega_{\nu}}$. 
In view of the identity 
$U^{\chi_{\Omega_{\nu}}-\nu}=u_\G$ we may write
\[
I_\alpha(\mu_{\alpha,\G})-I_\alpha(\mu_\alpha)
=\frac{1}{\alpha}\Sigma(\mu_\alpha,\chi_{\Omega_{\nu}}-\nu)
=\frac{1}{\pi\alpha^2}\int_{\Omega_\nu}u_\G \, \diffA.
\]
This completes the proof.
\end{proof}

\subsection{Fekete points and discretization of the limiting measure}
\label{ss:Fekete-appl}
Before embarking on the proof of a lower bound for the hole probability, 
we explain how to construct a discrete approximation for the conditional limiting measure.

Consider the measure $\mu_{\alpha,\G}$ for $\alpha \ge \alpha_0$ and 
its logarithmic potential $U^{\mu_{\alpha,\G}}$. By 
Theorem~\ref{thm:Lipschitz-reg} and Proposition~\ref{prop:glue}, 
there is a constant $\gamma = \gamma_\G \in (0,1]$ such that the 
potential $U^{\alpha \, \mu_{\alpha,\G}} \mid_{\{ \D(0,\sqrt{\alpha_0})\}}$
is H{\"o}lder 
continuous with exponent $\gamma$ and norm equal to $C_U^\gamma$ 
(both $\gamma$ and $C_U^\gamma$ do not depend on $\alpha$).
Since $U^{\mu_{\alpha,\G}} = U^{\mu_{\alpha}}$ outside $\D(0,\sqrt{\alpha_0})$, we conclude that
\[
\lVert U^{\mu_{\alpha,\G}}\rVert_{C^{0, \gamma}} = \Ordo( \alpha^{- \gamma/2} ).
\]

\begin{lem}\label{lem:discrete-approx-of-min-measure}
Let $\G$ be a domain contained in the unit disk and satisfying 
the conditions of Theorem~\ref{thm:main-GEF2}
and put $N = \left\lceil L^2 \log L\right\rceil = \alpha L^2$. 
For all $L$ sufficiently large (depending on $\G$) there 
is a set of points $\calF_N^\G = (z_1, \dots, z_N) = \textbf{z}$ with the following properties:
\begin{enumerate}
\item \label{itm:Fek-restric} The points lie outside of 
$\G$: $\calF_N^\G \subset \D(0,\sqrt{\alpha}) \setminus \G$.
\item \label{itm:Fek-sep} The points are separated:
\[
\inf_{j \ne k}|z_j - z_k|\ge L^{-3 / \gamma}.
\]
\item \label{itm:Fek-energy} The logarithmic energy of $\mu_{\textbf{z}}$ satisfies the bound
\[
-\Sigma^*(\mu_{\textbf{z}}) \le -\Sigma(\mu_{\alpha, \G}) + \frac{2}{\gamma L^2}.
\] 
\item \label{itm:Fek-potential} The logarithmic potential of $\mu_{\textbf{z}}$ meets the bound
\[
U^{\mu_{\textbf{z}}}(\zeta) \le U^{\mu_{\alpha,\G}}(\zeta) 
+ \frac{12}{\gamma L^2} \qquad \forall \zeta \in \C.
\]
\end{enumerate}
\end{lem}

\begin{proof}

We apply Proposition \ref{prop:Ghosh-Gotz-Saff-sep} and 
Theorem~\ref{thm:Ghosh-Gotz-Saff} with $\Lambda = \D(0, \sqrt{\alpha}) \setminus \G$, 
$Q = U^{\mu_{\alpha,\G}}$ (so that $\mu_{Q,\Lambda} = \mu_{\alpha,\G}$) 
and we denote by $\calF_N^\G = (z_1, \dots, z_N) = \textbf{z}$ 
the corresponding (weighted) Fekete points, restricted to $\Lambda$. 
Property~\eqref{itm:Fek-restric} holds by definition.
We have,
$$
\inf_{z^\prime \ne z^{\prime\prime}\in\calF_N^\G}|z^\prime 
- z^{\prime\prime}|\ge \frac12 \exp(-C \alpha^{-\gamma/2}) N^{-1/ \gamma},
$$
which gives Property~\eqref{itm:Fek-sep}, when $L$ is sufficiently large. Moreover,
\[
-\Sigma^*(\mu_{\textbf{z}}) \le -\Sigma(\mu_{\alpha, \G}) + E_1(N,\alpha),
\]
where $E_1(N, \alpha)  = N^{-1} [ C \alpha^{-\gamma/2} 
+ 2 \log 2 \sqrt{\alpha} + \gamma^{-1}\log N \big]$.
In addition, by the proof of Proposition~\ref{prop:struct} 
we have that $D(Q,\Lambda) = \sup_{z\in\Lambda} |U^{ \mu_{\alpha,\G}}| 
= \Ordo(\log \alpha)$. 
Therefore,
\[
U^{\mu_{\textbf{z}}}(\zeta) \le U^{\mu_{Q,\Lambda}}(\zeta) 
+ E_2(N,\alpha) \qquad \forall \zeta \in \C,
\]
where
$
E_2(N,\alpha) \le 2 N^{-1} [C \alpha^{-\gamma/2} 
+ \log 2 \sqrt{\alpha} + (2 + 3 \gamma^{-1}) \log N + C \log \alpha].
$
Properties \eqref{itm:Fek-energy} and \eqref{itm:Fek-potential} 
are established, for $L$ sufficiently large, by examining 
the relations between $L, N$ and $\alpha$.
\end{proof}

The following properties of Fekete points 
(and small perturbations of them) are crucial for the 
proof of the lower bound. Recall that 
$|\Delta(\textbf{z})|^2 = \prod_{i\ne j}|z_i-z_j| = \exp\big(N^2 \Sigma^*(\mu_{\textbf{z}})\big)$
and $S({\bf z})= \pi^{-1} L^2\int_{\C}|Q_{\bf z}(w)|^2 \e^{-L^2|w|^2}\diffA(w).$

\begin{prop}\label{prop:approx-Fekete-points}
Let $\calF_N^\G = \textbf{z} = (z_1, \dots, z_N)$ 
be a set of Fekete points, and $\tau > 3/\gamma + 2$ as 
in the previous lemma. Moreover, let the points 
$\textbf{w} = (w_1, \dots, w_N) \in \D(0, \sqrt{\alpha})^N$ 
satisfy $|w_i - z_i| \le L^{-\tau}$ for all $i \in \{1, \dots, N\}$. 
For $L$ sufficiently large, we have
\begin{enumerate}
\item \label{itm:Vandermonde_low_bound} 
$|\Delta(\textbf{w})|^2 \ge \exp\big(N^2 \Sigma(\mu_{\alpha,\G}) - (C/\gamma) (N/L)^2 \big).$
\item \label{itm:S_upp_bound} $S(\textbf{w}) 
\le \exp\big(2 N B_\alpha(\mu_{\alpha,\G}) + \tfrac{C \alpha}{\gamma} \big).$
\end{enumerate}

\end{prop}

\begin{proof}
The reverse triangle inequality
\[
|w_j - w_k| \ge |z_j - z_k| - 2 \max_i |w_i - z_i| \ge |z_j - z_k| - 2 L^{-\tau},
\]
shows that Properties~\eqref{itm:Fek-sep} and 
\eqref{itm:Fek-energy} in Lemma~\ref{lem:discrete-approx-of-min-measure} 
hold also for $\textbf{w}$, perhaps with different (absolute) 
constants on the respective right hand sides. In particular, this proves 
Property~\ref{itm:Vandermonde_low_bound}.

We now show that Property~\eqref{itm:Fek-potential} of 
Lemma~\ref{lem:discrete-approx-of-min-measure} also holds 
for $\textbf{w}$ (with a modified constant). Let 
$\gamma \in (0,1]$ be the H{\"o}lder exponent of $U^{\mu_{\alpha,\G}}$
and let $\sigma_t$ be the normalized 
Lebesgue measure on the circle $\T(0, t)$. If we take $t = L^{-2/\gamma}$, then
\[
U^{\mu_{\textbf{z}} * \sigma_t}(\zeta) 
= \int_\T U^{\mu_{\textbf{z}}}(\zeta + t \e^{i \theta}) \, 
\frac{\diff \theta}{2 \pi} \le U^{\mu_{\alpha, \G}}(\zeta) 
+ \frac{24}{\gamma L^2}, \quad \zeta \in \C,
\]
for $L$ sufficiently large, where we used 
Property~\eqref{itm:Fek-potential} of Lemma~\ref{lem:discrete-approx-of-min-measure}, 
and the choice of $t$.
Put $a \vee b \coloneqq \max\{a,b\}$. 
Observe that for $L$ sufficiently large,
\[
\big| \log(|\zeta - w_j|\vee t) - \log(|\zeta - z_j|\vee t) \big| 
\le \frac{2 L^{-\tau}}{t} \le \frac1{\gamma L^2}, \quad \zeta \in \C.
\]
Since
\[
U^{\mu_{\textbf{w}}} (\zeta) \le U^{\mu_{\textbf{w} * \sigma_t}}(\zeta) 
= \frac1N \sum_{k=1}^N \log(|\zeta - w_j|\vee t) \le 
U^{\mu_{\textbf{z} * \sigma_t}}(\zeta) + \frac{1}{\gamma L^2},
\]
we obtain the required bound for $U^{\mu_{\textbf{w}}}$.

We return to the proof of Property \eqref{itm:S_upp_bound}. Write
\[
S(\textbf{w}) = \int_\C \exp\big( 2N \, U^{\mu_{\textbf{w}}}(\zeta) 
- L^2 |\zeta|^2 \big) \frac{L^2}{\pi}\, \diffA(\zeta) \eqqcolon I_1 + I_2
\]
where $I_1$ in the integral over $D(0, \sqrt{\alpha})$ and 
$I_2$ over $\C \setminus D(0, \sqrt{\alpha})$. For $I_1$ we have the bound
$$
\int_{\D(0,\sqrt{\alpha})} \exp\Big( 2N ( U^{\mu_{\alpha, \G}}(\zeta) 
- \tfrac{|\zeta|^2}{2\alpha} ) + \tfrac{C N}{\gamma L^2} \Big) \frac{L^2}{\pi}\, 
\diffA(\zeta)  \le \alpha L^2 \cdot \exp\big( 2 N B_\alpha(\mu_{\alpha,\G}) + \tfrac{C \alpha}{\gamma}  \big).
$$
We bound $I_2$ by
\[
\int_{\{|\zeta| \ge \sqrt{\alpha}\}} \exp\big( 2N \, 
\log |\zeta| - L^2 |\zeta|^2 \big) \frac{L^2}{\pi}\, \diffA(\zeta)
= \calE \big[ |W|^{2N} \textbf{1}_{\{|W|\ge\sqrt{\alpha}\}} \big],
\]
where $W \sim N_\C(0, L^{-1})$. Now an application of the 
Cauchy-Schwarz inequality shows that $I_2 = \ordo(I_1)$ if $L$ is large.
\end{proof}

\subsection{Lower bound for the hole probability}
Given $\delta > 0$ we write
\[
\G_\delta^+ = \{ w\in \C : d(w,\G) \le \delta \},
\]
for the $\delta$-neighbourhood of $\G$.
Combining Propositions \ref{prop:T_N_bound} and \ref{prop:lower_bnd_away_from_zeros} 
we see that it is enough to construct an event $\calH_{\G, L}^N$ on which the 
polynomial $P_{N,L}$ has no zeros inside $\G_\delta^+$, with $\delta = L^{-\tau}$, 
and $\tau = \frac6\gamma$. Notice we may choose $B$ sufficiently large so that 
in Proposition~\ref{prop:lower_bnd_away_from_zeros} the 
exceptional event is negligible.
Then we choose $A$ in Proposition~\ref{prop:T_N_bound} large so that $P_{N,L}$ 
dominates the tail $T_N$ in $\D(0, B)$ which contains $\G_\delta^+$.

The construction of the event $\calH_{\G, L}^N$ is based on a collection of 
small perturbations of weighted Fekete points $\calF_N^\G = (z_1, \dots, z_N)$ 
from the previous section. More precisely, by the uniform cone condition, 
there is a constant $c_\G >0$ such that
\[
|\calN(z_j, \delta)| \coloneqq |\D(z_j, 2 \delta) \cap \big(\G_\delta^+\big)^c 
\cap \D(0, \sqrt{\alpha}) | \ge c_\G \delta^2, \qquad \forall j \in \{1, \dots, N\}.
\]
Note that such a lower bound is immediate if $z_j$ is sufficiently 
far from $\partial \G$. We define the event $\calH_{\G, L}^N$ by
\[
\calH_{\G, L}^N = \big\{ \textbf{w} = (w_1, \dots, w_N) \in \C^N \, : 
\, w_j \in \calN(z_j, \delta) \quad \forall j \in \{1, \dots, N\} \big\},
\]
and observe that $|\calH_{\G, L}^N| \ge (c_\G \delta^2)^N 
\ge \exp( - 13 \gamma^{-1} L^2 \log^2 L )$, for $L$ sufficiently large.

The expression \eqref{eq:joint-density} for the joint density of 
the zeros, and \eqref{eq:P-N,L-in-prod} give us
\[
\calP \big( \calH_{\G, L}^N \big) = \int_{\calH_{\G, L}^N} 
f_{N,L}({\bf w}) \, \diffA(\textbf{w})  =\frac{1}{\calZ_{N,L}}  
\int_{\calH_{\G, L}^N} |\Delta(\textbf{w})|^2 \,
S(\textbf{w})^{-(N+1)} \, \diffA(\textbf{w}).
\]
The normalizing constant $\calZ_{N,L}$ is 
given by (see e.g. \cite[Appendix A, Lemma A.1]{GhoshNishry1})
\begin{align*}
\calZ_{N,L} & = \frac{\pi^N L^{N(N+1)}}{N!} \prod_{k=1}^N \frac1{k!}
= \exp\big( -\tfrac12 N^2 \log \alpha + \tfrac34 N^2 + \Ordo(L^2 \log^2 L ) \big)\\
& = \exp\big( -N^2 I_\alpha(\mu_{\alpha}) + \Ordo(L^2 \log^2 L) \big),
\end{align*}
where we recall $\mu_{\alpha} = \chi_{\D(0,\sqrt{\alpha})}\, 
\diffA$ is the global minimizer of $I_\alpha$. 

By appealing to Proposition \ref{prop:approx-Fekete-points}, 
we get the bound
\[
\calP \big( \calH_{\G, L}^N \big) \ge |\calH_{\G, L}^N|
\exp\Big( N^2 \big( I_\alpha(\mu_{\alpha}) + \Sigma(\mu_{\alpha,\G}) 
- 2 B_\alpha(\mu_{\alpha,\G}) \big) - E(N,L) \Big),
\]
where
\begin{align*}
E(N,L) & = (C/\gamma) (N/L)^2 + 2 N B_\alpha(\mu_{\alpha,\G}) 
+ (C/\gamma)(N+1)\alpha + \Ordo(L^2 \log^2 L)\\
& = (C/\gamma) L^2 \alpha^2 + N(\log \alpha + 1) 
+(C/\gamma)(N+1)\alpha + \Ordo(L^2 \log^2 L)\\
& = (1/\gamma) \Ordo(L^2 \log^2 L).
\end{align*}
We have shown that
\[
\calP \big( \calH_{\G, L}^N \big) \ge 
\exp\Big( -L^4 \alpha^2 \big( I_\alpha(\mu_{\alpha, \G}) 
- I_\alpha(\mu_{\alpha})\big) + (1/\gamma) \cdot \Ordo(L^2 \log^2 L) \Big),
\]
and therefore the lower bound for the hole probability 
in Theorem~\ref{thm:main-GEF2} follows from Lemma~\ref{lem:cost_of_hole_event}.

\subsection{Large deviation upper bound for linear statistics}
Here we give a brief account of the proof of the large 
deviation upper bound. This proof is largely based on 
\cite[\S~7.3]{GhoshNishry1},
with some important differences, 
concerning the approximation of $\G$, and some 
rather minor technical simplifications.

\subsubsection{Preliminary large deviation bound}
As before, we let $N = \alpha L^2 = \left\lceil L^2 \log L\right\rceil$. 
We use $\textbf{z} = (z_1, \dots, z_N)$ to denote a possible zero 
configuration (that is, the zeros of $P_{N,L}$). 
We also recall the expression for the joint density of the zeros
\[
f_{N,L}({\bf z}) = \frac{1}{\calZ_{N,L}}  |\Delta(\textbf{z})|^2 \,
S(\textbf{z})^{-(N+1)},
\]
where $\calZ_{N,L} = \exp\big( -N^2 I_\alpha(\mu_{\alpha}) + \Ordo(L^2 \log^2 L) \big)$.
Outside a probabilistically negligible event $E_1^N$ (cf. \cite[Claim 4.6]{GhoshNishry1})
we have
\[
S(\textbf{z}) \le \exp(C L^6).
\]
We define
\[
S^\dagger(\textbf{z}) = \sup_{w \in \C} \big\{|Q_{\textbf{z}}(w)|^2 
\e^{-L^2 |w|^2} \big\} = \exp(2 N B_\alpha(\mu_{\textbf{z}}) );
\]
note that in \cite{GhoshNishry1} the above expression is 
denoted by $A(\textbf{z})$. The Bernstein-Markov property 
$S^\dagger(\textbf{z}) \le S(\textbf{z})$ holds 
(e.g. see \cite[Appendix A, Lemma A.4]{GhoshNishry1}), and 
by \cite[Claim 4.5]{GhoshNishry1}, for $b > 1$
\[
\int_{\C^N} S^\dagger(\textbf{z})^{-b} \, \diffA(\textbf{z}) 
\le \exp\Big(b L^2 + N \log \big(\tfrac{Cb}{b-1}\big) \Big).
\]
Recall the definition of the discrete energy functional
\[
I_\alpha^*(\mu_{\textbf{z}}) = -\Sigma^*(\mu_{\textbf{z}}) + 2 B_\alpha(\mu_{\textbf{z}}),
\]
and observe that
\[
|\Delta(\textbf{z})|^2 \,
S(\textbf{z})^{-N} \le \exp\Big( N^2 \big( \Sigma^*(\mu_{\textbf{z}}) 
- 2 B_\alpha(\mu_{\textbf{z}}) \big) \Big)
= \exp\big( -N^2 I_\alpha^*(\mu_{\textbf{z}}) \big).
\]
Let $\calA^N \subset \C^N$ be a collection of possible 
positions of the zeros of $P_{N,L}$, we now obtain a large 
deviation upper bound for $\calP (\calA^N)$ as follows
\begin{align*}
\calP\big(\calA^N\big) & = \int_{\calA^N} f_{N,L}(\textbf{z}) \, \diffA(\textbf{z})\\
& \le J_1 \exp\Big( -N^2 \inf_{\textbf{z} \in \calA^N} 
\big( I_\alpha^*(\mu_{\textbf{z}}) - I_\alpha(\mu_{\alpha}) \big) 
+ \Ordo(L^2 \log^2 L) \Big) 
 + \calP\big(E_1^N\big),
\end{align*}
where
\begin{align*}
J_1 & = \int_{\C^N \setminus E_1^N} S(\textbf{z})^{-1} \, 
\diffA(\textbf{z}) \le \exp( C L^6 / N^2) \int_{\C^N} 
S^\dagger(\textbf{z})^{-1 - N^{-2}} \, \diffA(\textbf{z})\\
& \le \exp(CL^6/N^2 +(1+ N^{-2})L^2 + CN\log N) = \exp(\Ordo(L^2 \log^2 L)).
\end{align*}
We now replace $I_\alpha^*$ by $I_\alpha$. 
Put $\mu_{\textbf{z}}^t = \mu_{\textbf{z}} * \sigma_t$, 
where $\sigma_t$ is the normalized Lebesgue measure on 
the circle $\T(0,t)$. Using 
\cite[Claims 4.7, 4.8]{GhoshNishry1}
we get for $t$ small
\[
I_\alpha^*(\mu_{\textbf{z}}) \ge I_\alpha(\mu_{\textbf{z}}^t) 
- C \big( \tfrac1N \log \tfrac1t + \tfrac{t}{\sqrt{\alpha}} + \tfrac1{L^2} \big).
\]
We put $t = L^{-\tau^\prime}$, with $\tau^\prime \ge 2$ 
fixed, so that the error term above is $\Ordo(L^{-2})$.

Putting all of this together we have shown
\begin{equation}\label{eq:large_dev_upp_bnd}
\calP\big(\calA^N\big) \le \exp\Big( -N^2 \inf_{\textbf{z} \in \calA^N} 
\big( I_\alpha(\mu_{\textbf{z}}^t) - I_\alpha(\mu_{\alpha}) \big) + 
\Ordo(L^2 \log^2 L) \Big) + \calP\big(E_1^N\big),
\end{equation}
where $\calP\big(E_1^N\big)$ is negligible with respect to the hole probability.

\subsubsection{Upper bound for linear statistics}
Fix a parameter $\tau \ge 1$, let $\delta = L^{-\tau}$ and define
\[
\G_\delta^- = \{ w \in \G \, : \, \diff(w, \G^c) \ge \delta \},
\]
to be the $\delta$-interior of $\G$.
Using Propositions \ref{prop:T_N_bound} and \ref{prop:lower_bnd_away_from_zeros} 
(with $B$ large), together with Rouch\'e's theorem, we find that outside a 
probabilistically negligible exceptional event $n_{P_N}(\G_\delta^-) \le n_L(\G)$. Moreover,
let $\varphi$ be a continuous function supported in $\D(0, B - 1)$, 
with modulus of continuity given by $\omega(\varphi; \cdot)$, 
then by \cite[Corollary~3.4]{GhoshNishry1}
\[
|n_{L}(\varphi) - n_{P_N}(\varphi)| \le L^4 \omega(\varphi; \delta),
\]
outside an exceptional event.

Our goal is to bound from above the probability of the event
\[
H_{L, \G, \epsilon} \coloneqq 
\Big\{ F_L \ne 0 \mbox{ in } \G, \quad 
\Big| n_L(\varphi) - L^2 \begingroup\textstyle\int\endgroup 
\varphi \, \diff \mu_\G^\C \Big| > \epsilon L^2 \Big\}.
\]
\begin{rem}
To simplify matters we will take $\epsilon > 0$ fixed. 
Checking the details of the proof below, one can see 
that $L^{-1} \log L = \ordo(\epsilon)$ is the actual requirement. 
Similarly, as in \cite[\S~7]{GhoshNishry1} it is possible 
to take $\varphi$ depending on $L$, but we will not pursue 
this here (the details are similar).
\end{rem}

Since $\varphi$ is compactly supported, for $\alpha$ 
sufficiently large we have by \eqref{eq:limit_measure_def} that
\[
L^2 \int
 \varphi \, \diff \mu_\G^\C = N \int \varphi \, \diff \mu_{\alpha, \G}.
\]
In addition, $n_{P_N}(\varphi) = N \int \varphi \, \diff \mu_{\textbf{z}}$, thus,
\[
| n_{P_N}(\varphi) - N \int \varphi \, \diff \mu_{\textbf{z}}^t | \le N \omega(\varphi, t).
\]
Put $\G^\prime = \G_{\delta+t}^-$, and recall $\delta = L^{-\tau}, t = L^{-\tau^\prime}$. 
Using \cite[Lemma 3.14, Claim 5.9]{GhoshNishry1}, we get
\begin{align*}
\Big| \int \varphi \, \diff \mu_{\alpha, \G} - \int \varphi \, \diff \mu_{\alpha, \G^\prime} \Big|
& \le \frac{1}{\sqrt{2\pi}} \sqrt{\dir(\varphi)} 
\sqrt{-\Sigma(\mu_{\alpha, \G} - \mu_{\alpha, \G^\prime})}\\
& \le \frac{1}{\sqrt{2\pi}} \sqrt{\dir(\varphi)} 
\sqrt{I(\mu_{\alpha, \G}) - I(\mu_{\alpha, \G^\prime})}.
\end{align*}
If $\tau, \tau^\prime$ are chosen sufficiently large 
(depending only on $\G$) then by Proposition \ref{prop:quant-stab} 
we have $I(\mu_{\alpha, \G}) - I(\mu_{\alpha, \G^\prime}) = \Ordo(L^{-4})$.
Collecting all these estimates, we find that on the event 
$H_{L, \G, \epsilon}$ (and discarding an exceptional event of 
negligible probability), we have, for $L$ sufficiently large,
\[
\mu_{\textbf{z}}^t(\G^\prime) = 0 \mbox{ and } \Big| \int \varphi \, 
\diff \mu_{\textbf{z}}^t - \int \varphi \, \diff \mu_{\alpha, \G^\prime} \Big| 
> \frac{\epsilon}{\alpha} - \omega(\varphi, t) - \frac{L^4}{N} \omega(\varphi, \delta) 
- C L^{-2} > \frac{\epsilon}{2 \alpha}.
\]
Using \cite[Lemma 3.14, Claim 5.9]{GhoshNishry1} 
and Proposition \ref{prop:quant-stab} once again, we get
\[
I(\mu_{\textbf{z}}^t) \ge I(\mu_{\alpha, \G^\prime}) 
+ \frac{c \epsilon^2}{\dir(\varphi) \alpha^2} \ge I(\mu_{\alpha, \G}) 
+ \frac{c^\prime \epsilon^2}{\dir(\varphi) \alpha^2}.
\]
Finally, by \eqref{eq:large_dev_upp_bnd} we conclude that
\[
\calP\big( H_{L, \G, \epsilon} \big) \le 
\exp\Big( -N^2 \big( I(\mu_{\alpha, \G}) - I(\mu_{\alpha}) \big) - 
\frac{c \epsilon^2}{\dir(\varphi)} L^4 + \Ordo(L^2 \log^2 L) \Big),
\]
which combined with Lemma~\ref{lem:cost_of_hole_event} 
proves the large deviation upper bound for linear statistics.

\subsection{Convergence of the conditional empirical measures}
The proofs here are essentially the same as the ones in 
\cite[\S~7.4]{GhoshNishry1}. We provide few details below, 
and refer the interested reader to that paper for the rest.

From the above results it is possible to deduce that
\[
\Big|\calE \big[ n_{F_L}(\varphi) | H_{L,\G} \big] 
- L^2 \int \varphi \, \diff \mu_\G^\C \Big| \le C_\varphi L \log^2 L,
\]
where $C_\varphi$ depends on the test function $\varphi$.
In order to prove the vague convergence in distribution 
of the conditional empirical measures we have to show that the random variables
$
L^{-2} n_{F_L}(\psi) |_{H_{L,\G}}
$
converge in distribution to $\int \psi \, \diff \mu_\G^\C$ for 
every continuous test function with compact support $\psi$ 
(see e.g. \cite[Chapter 4]{Kallenberg}). 
This is achieved using 
an approximation of $\psi$ in $\sup$-norm by a smooth test function 
$\varphi$ which is compactly supported in a slightly larger set, and 
using the large deviation upper bound. 

This completes our outline of the proof of Theorem~\ref{thm:main-GEF2}.\qed

\section{The inverse problem, disk-like domains and examples}
\label{s:examples}

\subsection{A sufficient condition for extremality}
In this section, we discuss some cases in which one can say
more about the measure $\mu_{\alpha,\G}$ than 
provided by the general classification encountered in
Theorem~\ref{thm:s-qdom}. 
In order to show that a measure is indeed an 
extremal measure for the minimization of $I_\alpha$
over the class $\calM_\G$, we will verify 
that the variational inequalities
for $I_\alpha$ found
in \cite{GhoshNishry1} are met. 
We summarize this variational principle in a proposition.

\begin{prop}\label{prop:var-ineq}
Let $\mathcal{M}$ be a non-empty closed convex set, 
whose elements are compactly supported 
probability measures on $\C$ with 
finite logarithmic energy.
The minimization problem
\[
\text{minimize }\;\;I_{\alpha}(\mu)\qquad
\text{subject to }\;\;\mu\in\mathcal{M}
\]
admits a unique solution $\mu_0=\mu_0(\calM)$. Moreover, $\mu_0$ is 
characterized by the property that 
for all $\nu\in \mathcal{M}$,
we have
\[
B_\alpha(\nu)-\int U^\nu\,\diff\mu_0\ge B_\alpha(\mu_0)
-\int U^{\mu_0}\,\diff\mu_0.
\]
with equality if and only if $\nu=\mu_0$.
\end{prop}

The proof of existence of minimizers and of the $\Rightarrow$ direction 
is supplied in \cite{GhoshNishry1}.

\begin{proof}[Proof of $\Leftarrow$]
Assume that $\mu_0$ meets the variational inequality
for all $\nu\in\calM$. We may assume that $\calM$ contains 
other measures than $\mu_0$, or the result follows directly. 
Then we find that for any admissible $\nu\ne \mu_0$,
we have
\begin{multline*}
I_\alpha(\nu)=2B_\alpha(\nu)-\Sigma(\nu)=2B_{\alpha}(\nu)-2\int U^{\mu_0}\,\diff\nu+
2\int U^{\mu_0}\,\diff\nu-\Sigma(\nu)
\\
\ge 2B_{\alpha}(\mu_0)-2\Sigma(\mu_0) +2\int U^{\mu_0}\,\diff\nu-\Sigma(\nu) 
=I_\alpha(\mu_0)-\Sigma(\mu_0-\nu)>I_\alpha(\mu_0)
\end{multline*}
where we use the positivity of $-\Sigma(\eta)=\dir(U^\eta)$ 
for a signed measure $\eta$ with total mass zero,
so $\mu_0$ is the minimum. 
\end{proof}

Using these variational inequalities for the 
extremal problem, we can find sufficient conditions
for extremality on the hole event, in a form which is convenient to
approach the inverse problem (Problem~\ref{prob:inverse}).

\begin{lem}\label{lem:verif}
Let $\nu$ be a measure supported in $\G$, 
and suppose that $\Omega_\nu$ is a subharmonic quadrature 
domain with respect to $\nu$, containing $\G$. Let $\mu$ be the measure
$$ 
\diff\mu=\frac{1}{\alpha}\,\diff\mathrm{Bal}\big(\nu,\G^c\big) + 
\frac{1}{\pi\alpha}\chi_{\D(0,\sqrt{\alpha})\setminus\Omega_\nu}\diffA.
$$
Assume moreover that $B_\alpha(\mu)=c_\alpha$ and that 
$$
\mathrm{supp}(\nu)\subset\calI=
\{z\in \G:U^\mu(z)=\tfrac{1}{\alpha}|z|^2+c_\alpha\}.
$$
Then $\mu$ is the minimizer of $I_\alpha$ over 
$\calM_\G$, i.e.\ $\mu=\mu_{\alpha,\G}$.
\end{lem}

\begin{proof}
Let $\eta$ be any competing measure in $\calM_1(G)$.
We split $\mu=\mu^{\s}+\mu^{\c}$ according to the Lebesgue decomposition, 
with $\mathrm{supp}(\mu^\s)\subset\partial\G$, and note that
\begin{multline*}
\int U^\eta\diff\mu= \int_{\D(0,\sqrt{\alpha})\setminus\Omega_\nu}U^\eta(z)
\frac{\diffA(z)}{\pi\alpha}+\frac{1}{\alpha}
\int_{\partial \G}U^\eta\diff\mathrm{Bal}(\nu,\G^c)\\
\le \Big(1-\frac{|\Omega_\nu|}{\alpha}\Big)B_\alpha(\eta)
+\int_{\D(0,\sqrt{\alpha})\setminus
\Omega_\nu}\frac{|z|^2}{2\alpha}\frac{\diffA(z)}{\pi\alpha}
+\frac{1}{\alpha}\int_{\calI} U^\eta\diff\nu,
\end{multline*}
where we use the harmonicity of $U^\eta$ in 
$\G$ to pass from integration against the balayage measure
to integration against $\nu$.
By adding and subtracting a quantity 
independent of the competing measure $\eta$, we may further rewrite
$$
\frac{1}{\alpha}\int_{\calI} U^\eta\,\diff\nu\le 
\frac{|\nu|}{\alpha}B_\alpha(\eta)
+\frac{1}{\alpha}\int_{\calI}\frac{|z|^2}{2\alpha}\,\diff\nu
=\frac{|\Omega_\nu|}{\alpha}B_\alpha(\eta)
+\frac{1}{\alpha}\int_{\calI}\frac{|z|^2}{2\alpha}\,\diff\nu,
$$
so it follows that
$$
\int U^\eta\,\diff\mu\le \Big(1-\frac{|\Omega_\nu|}
{\alpha}\Big) B_\alpha(\eta) + 
\frac{|\Omega_\nu|}{\alpha}B_\alpha(\eta)+E=B_\alpha(\eta)+E,
$$
where $E$ is the $\eta$-independent quantity
$$
E=\frac{1}{\alpha}\int_{\calI}\frac{|z|^2}{2\alpha}
\diff\nu +\int_{\D(0,\sqrt{\alpha})
\setminus\Omega_\nu}\frac{|z|^2}{2\alpha}\,\frac{\diffA(z)}{\pi\alpha}.
$$
Similarly when we replace $\eta$ by $\mu$, 
we have equality in each of the above steps, and 
find that 
$$
\int U^\mu\,\diff\mu=B_\alpha(\mu)+E.
$$
Summarizing what we have established, we have
$$
B_\alpha(\eta)-\int U^\eta\,\diff\mu\ge -E
$$
while
$$
B_\alpha(\mu)-\int U^\mu\,\diff\mu=-E
$$
so it follows that for any probability measure $\eta$
with $\eta(\G)=0$, it holds that
\[
B_\alpha(\eta)-\int U^\eta\,\diff\mu\ge B_\alpha(\mu)-\int U^\mu\,\diff\mu,
\]
which completes the proof (by Proposition~\ref{prop:var-ineq}).
\end{proof}

\subsection{Disk-like domains}\label{s:a-circ}
Recall that a bounded simply connected 
domain $\G$ is said to be disk-like 
with radius $r$ and center $z_0\in \G$ 
if the conformal mapping $\varphi=\varphi_{z_0}:\G\to\D$, 
normalized by $\varphi(z_0)=0$ 
and $\varphi'(z_0)>0$, meets the bound
\begin{equation}\label{eq:dist-local-sect}
|\varphi(z)|\ge \frac{|z-z_0|}{r}
\e^{-\frac{|z-z_0|^2}{2\e r^2}},\qquad z\in \overline{\G},
\end{equation}
where $r$ is given by $r=|\varphi'(z_0)|^{-1}$.
In fact, for a disk-like domain $\G$, the point $z_0$ 
is a local minimizer of $z\mapsto\varphi_z'(z)$, which is to say that
$z_0$ is a \emph{conformal center} of $\G$. The number $r$ is the
\emph{inner conformal radius} with respect to $(\G,z_0)$.
For more details on these notions, we refer to \cite[\S~6.3]{PolyaSzego}.

The potential of the Balayage measure
$\mathrm{Bal}(\delta_w,\G^c)$ is given by
\[
U^{\mathrm{Bal}(\delta_w,\G^c)}(z)=\log|z-w| - g_{\G}(z,w)
=\log\frac{|z-w|}{|\phi_w(z)|},
\]
where $g_{\G}$ denotes the Green function
for $\G$ with pole at $w$, and $\phi_w$ is a conformal mapping
of $\G$ onto $\D$ with $\phi_w(w)=0$. 

\begin{prop}\label{prop:a-circ}
Let $\G$ be an disk-like domain of radius $r$ and 
center $z_0$, and let $\alpha$ be large enough for the 
disk $\D(z_0,\sqrt{\e}r)$ to be contained in the disk 
$\D(0,\sqrt{\alpha})$. Then the minimizer 
$\mu_{\alpha,\G}$ of $I_\alpha$ over $\calM_{\G}$ is given by
\[
\diff\mu_{\alpha,\G}=\frac{\e r^2}{\alpha}\,\diff\omega_{z_0,\G}+
\frac{1}{\pi\alpha}\chi_{\D(0,\sqrt{\alpha})\setminus
\D(z_0,\sqrt{\e}r)}\,\diffA,
\]
where $\omega_{z_0,\calG}$ denotes harmonic measure from the 
point $z_0$ in $\G$. Conversely, let $\G\subset\D$
denote a Jordan domain with piecewise smooth boundary without cusps whose
forbidden region is a disk. Then $\G$ is disk-like.
\end{prop}

\begin{rem} Curiously, it is not immediately clear that $\G$ could not be disk-like 
with respect to several different pairs $(r_j,z_j)$.
However, it follows from Proposition~\ref{prop:a-circ} that it is so,
since the minimizer of $I_\alpha(\mu)$ over $\calM_\G$ is unique.
\end{rem}

\begin{rem}\label{rem:p-a-circ}
In addition, the minimal value of the functional satisfies
\[
I_\alpha(\mu_{\alpha,\G})-I_\alpha(\mu_{\alpha})=
\frac{1}{\pi\alpha}\int_{\D(z_0,\sqrt{\e}r)}
\frac{|z-z_0|^2}{2\alpha}\diffA(z)=\frac{\e^2r^4}{4\alpha^2},
\]
where $\mu_\alpha$ denotes the unconstrained equilibrium measure 
$\mu_\alpha=(\pi\alpha)^{-1}\chi_{\D(0,\sqrt{\alpha})}\diffA$.
In view of Theorem~\ref{thm:main-GEF2} this justifies the remark
following Theorem~\ref{thm:a-circ} concerning hole probabilities.
\end{rem}
\begin{proof}
The first claim of the theorem is that whenever 
$\alpha$ is large enough so that 
$\D(z_0,r\sqrt{\e})\subset\D(0,\sqrt{\alpha})$, the 
extremal measure for $I_\alpha$ over $\calM_\G$ is the measure
\begin{equation}\label{eq:measure-a-circ-claim}
\diff\mu_0=\frac{\e r^2}{\alpha}\omega_{z_0,\G}
+\frac{1}{\pi\alpha}\chi_{\D(0,\sqrt{\alpha})
\setminus\D(z_0,r\sqrt{\e})}\,\diffA.
\end{equation}
Here, it is important to notice that $\G\subset\D(z_0,\sqrt{\e}r)$. Indeed,
it this is not the case then there exists a point $z\in\G \cap\T(z_0,\sqrt{\e}r)$.
As a consequence, $|z-z_0|^2=\e r^2$, so that
\[
|\varphi(z)|\ge\frac{|z-z_0|}{r}\e^{-\frac{|z-z_0|^2}{2\e r^2}}=1,
\]
which says that $z\in\partial\G$, which is a contradiction.

We will apply Lemma~\ref{lem:verif}, and hence we begin 
to examine the relative potential $R_{\mu_0}$, where for a measure
$\mu$ the relative potential $R_\mu$ is given by
\begin{equation}\label{eq:rel-pot}
R_{\mu}(z)=U^{\mu}(z)-\tfrac{1}{2\alpha}|z|^2.
\end{equation}
We have, with $c_\alpha=\tfrac12(\log\alpha-1)$, that
\begin{multline*}
U^{\mu_{0}}(z)-\frac{|z|^2}{2\alpha}
=c_\alpha+U^{\mu_{0}^{\s}}(z)
-U^{\frac{1}{\pi\alpha}\chi_{\D(z_0,\sqrt{\e}r)}}(z)
\\
=c_\alpha+\frac{\e r^2}{\alpha}\log\frac{|z-z_0|}{|\varphi(z)|}
-\frac{\e r^2}{\alpha}\log r-\frac{|z-z_0|^2}{2\alpha}.
\end{multline*}
At the point $z_0$, the value of the relative potential 
equals $c_\alpha$, as is seen by using the fact that
$\varphi'(z_0)=\tfrac{1}{r}$.

We claim that the distortion bound \eqref{eq:dist-local-sect}
says that the relative potential reaches its maximum on the entire 
set $\{z_0\}\cup\mathrm{supp}(\mu_{0}^{\c})$. 
On the complement $\G^c$ we may write the relative potential as
\[
U^\mu(z)-\frac{|z|^2}{2\alpha}=c_\alpha
+\frac{1}{\alpha}\Big(U^{\e r^2\delta_{z_0}}(z)
-U^{\frac{1}{\pi}\chi_{\D(z_0,\sqrt{\e} r)}}(z)\Big)
\]
which is bounded above by $c_\alpha$.
Moreover, since the point mass $\e r^2\delta_{z_0}$ and the area measure
$\pi^{-1}\chi_{\D(z_0,\sqrt{\e}r)}\diffA$ have the 
same potential outside 
$\D(z_0,\sqrt{\e}r)$,
we have equality on the set 
$\C\setminus\D(z_0,\sqrt{\e}r)$. 
As a consequence of the above, it is enough 
to study $R_{\mu_{0}}(z)$ on $\G$. Subtracting 
$c_\alpha$, the condition that $R_{\mu_{0}}(z)$ 
attains its maximum on $z_0$ reads
\[
\frac{\e r^2}{\alpha}\log\frac{|z-z_0|}{|\varphi(z)|}
-\frac{\e r^2}{\alpha}\log r-\frac{|z|^2}{2\alpha}\le 0,
\qquad z\in\G
\]
or, equivalently
\[
\lvert \varphi(z)\rvert\ge \frac{|z-z_0|}{r}
\e^{-\frac{|z-z_0|^2}{2\e r^2}},\qquad z\in \G
\]
which is precisely the assumed disk-likeness condition.

We next turn to the converse statement.
From the obstacle problem characterization of the equilibrium 
measure, 
it is clear that $U^{\mu_{\alpha,\G}}$ equals 
$\tfrac{|z|^2}{2\alpha}+c_\alpha$ on the support 
of the continuous part of the measure.
The piecewise smoothness assumption implies, via 
Theorem~\ref{thm:s-qdom}, that the singular measure is 
the balayage to $\partial \G$ of a measure $\nu$ on $\G$.
The assumption that the forbidden region is a disk, 
say $\D(0,R)$, now shows that
\[
U^{\mu_{\alpha,\G}^\s}(z)=U^{\frac{1}{\pi\alpha}\chi_{\D(0,R)}}(z),
\qquad z\in \D(0,\sqrt{\alpha})\setminus\D(0,R).
\]
From this it follows that the atomic measure is a 
single point mass. 
Indeed, on the annulus $\D(0,\sqrt{\alpha})\setminus\D(0,R)$ we have
\[
R_{\mu_{\alpha,\G}}(z)=\frac{1}{\alpha}\int\log|z-w|\,\diff\nu
-\frac{R^2}{\alpha}\log|z|=0.
\]
But since $\nu$ has regular support 
(a loop-free finite union of analytic curves and countably many points), the set 
$D(R,\nu)\coloneqq \D(0,R)\setminus\mathrm{supp}(\nu)\cup\{0\}$ is connected and open.
But then 
\begin{equation}\label{eq:harm-cont}
\int\log|z-w|\,\diff\nu
=R^2\log|z|,\qquad z\in D(R,\eta)
\end{equation}
by harmonic continuation.
But then the two functions in \eqref{eq:harm-cont} agree as distributions, so 
taking the Laplacian of both sides we see that this means that 
$\nu=R^2\delta_0$.
As a consequence of this, the measure 
$\mu_{\alpha,\G}$ takes the form
\[
\mu_{\alpha,\G}=\frac{R^2}{\alpha}\diff\omega_{0,\G}+\frac{1}{\pi\alpha}
\chi_{\D(0,\sqrt{\alpha})\setminus\D(0,R)}\diffA.
\]
We can then compute the logarithmic potential of 
$\mu_{\alpha,\G}$ in terms of the Riemann mapping $\varphi$
of $\G$, which takes the origin to the origin with positive derivative. 
Expressing the fact that $U^{\mu_{\alpha,\G}}(z)
-\tfrac{|z|^2}{2\alpha}\le c_\alpha$ in terms the conformal
mapping yields that $\G$ is disk-like.  
\end{proof}

\begin{proof}[Proof of Theorem~\ref{thm:a-circ}]
This is now immediate in view of Theorem~\ref{thm:main-GEF}
and Proposition~\ref{prop:a-circ}.
\end{proof}

Examples of domains $\G$ which are disk-like 
include the 
square, ellipses up to a critical eccentricity (see \cite{mathematica-Neumann}), 
and a certain eye shaped domain. 
The latter is worth a special mention, as it is an 
extremal domain among disk-like holes. 
With $r=1$ and $z_0=0$ we set 
\[
\varphi(z)=z\e^{-\frac{z^2}{2\e}}.
\]
The conformal mapping $f\colon\D\to \G$ is the 
inverse $f=\varphi^{-1}$, and is given by
\[
f(w)=-\imag\sqrt{\e}W\big(-\tfrac{z^2}{\e}\big),\qquad w\in\D
\]
where $W$ is the principal branch of the Lambert $W$-function.
The function $f$ is a conformal mapping, 
which maps $\D$ onto an elongated eye shaped domain which 
touches $\partial\D(0,\sqrt{\e})$ at its \emph{corners}
at $\{-\sqrt{\e},\sqrt{\e}\}$. The relative potential 
takes the extremal value $c_\alpha$ on the entire line
segment $[-\sqrt{\e},\sqrt{\e}]$. Hence this domain is 
degenerate in the sense of \eqref{eq:non-deg}.

\subsection{A family of Neumann ovals}
We next discuss a family of holes whose 
forbidden regions are \emph{Neumann ovals}. 
For $t>1$, we let $\G_t$ denote the disjoint 
union of two unit disks centered at $-t$
and $t$. For $0<t<1$ we continue this 
family of domains by letting $\G_t$
be the unique subharmonic quadrature domain 
with nodes $\{-t,t\}$
and a unit mass at each one of them.
More generally, the symmetric Neumann oval $\Omega_{r,t}$
with respect to
$r(\delta_t+\delta_{-t})$ is the Jordan domain
whose boundary is given by the equation 
\[
(x^2+y^2)^2-2r^2(x^2+y^2)-2t^2(x^2-y^2)=0,\qquad (x,y)\in\R^2,
\]
see for instance \cite{Shapiro}.
Denote by $\mu_t$ the minimizer of $I_\alpha$ over $\calM_{\G_t}$.
For $t>\sqrt{\e}$ and $\alpha$ large enough (with respect to $t$), 
it is evident that each component 
of the hole expels a circular forbidden region of 
radius $\sqrt{\e}$, concentric with the hole component. 
At $t=\sqrt{\e}$, the forbidden regions touch, and their union
equals the degenerate Neumann oval with mass $2\sqrt{\e}$ 
and nodes at $\{-\sqrt{\e},\sqrt{\e}\}$. 
As $t$ decreases from this critical point, we expect a 
nontrivial family of solutions. 

\begin{example}
There exist functions $r=r(t),\lambda=\lambda(t)$ such that 
$\mu_{t}=\mu_{t,r(t),\lambda(t)}$, where 
\[
\mu_{t,r,\lambda}=\frac{r}{\alpha}
\mathrm{Bal}(\delta_{-\lambda}+\delta_{\lambda},G_t^c)
+\frac{1}{\pi\alpha}\chi_{\D(0,\sqrt{\alpha})\setminus\Omega_{r,\lambda}}\diffA.
\]
Moreover, the function 
$\lambda(t)$ 
is decreasing in $t$ and there exists a number 
$t_0>0$ such that $\G_t$ is disk-like for $t\le t_0$, and $\lambda(t)=0$
for such $t$. 
\end{example}

Since the article is already too long, we will not provide a 
proof of this statement. 
A numerical verification may be done with the help of 
computer algebra software and Lemma~\ref{lem:verif},
see \cite{mathematica-Neumann}. 
Indeed, in view of Lemma~\ref{lem:verif} we 
need only ensure that one may choose parameters $r=r(t)$
and $\lambda=\lambda(t)$ such that the relative potential 
\[
R_{\mu_{t,r,\lambda}}=U^{\mu_{t,r,\lambda}}
-U^{\mu_\alpha}
\]
attains its maximal value at $\lambda$, where we recall that 
$\mu_\alpha=\tfrac{1}{\pi\alpha}\chi_{\D(0,\sqrt{\alpha})}\diffA$.
For illustrations of the Neumann oval family 
discussed above, see Figure~\ref{fig:potential-oval}.
The hole is in both instances the disjoint 
union of two disks of the same radius. In the
former the forbidden regions barely touch,
while in the second figure the forbidden 
regions have merged to form a Neumann oval.

\begin{figure}[t!]
\vspace{12pt}
\begin{subfigure}[t]{.44\textwidth}
  \centering
  \includegraphics[width=1\linewidth]{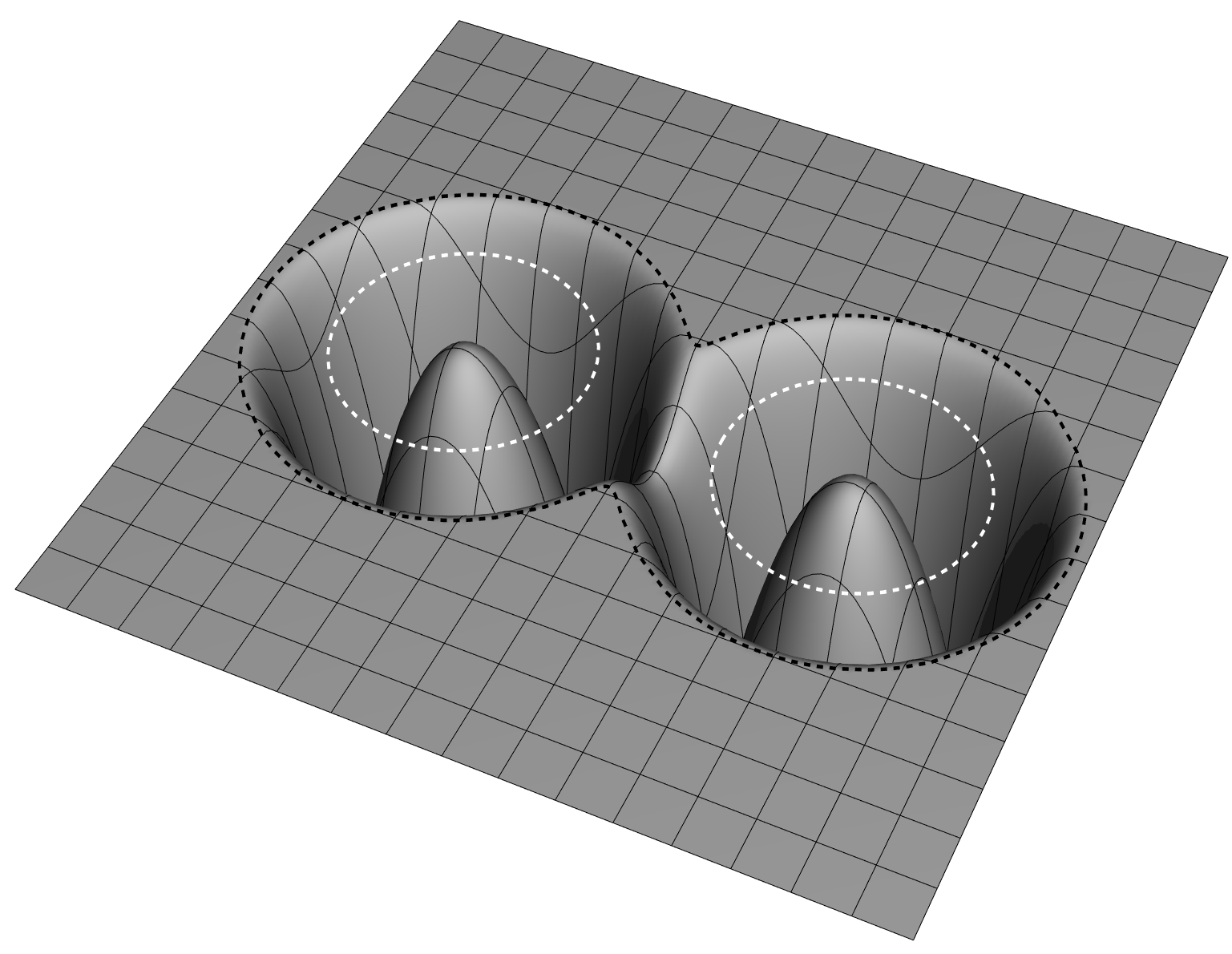}
  \subcaption{A hole formed by two separated disks (white) with forbidden region 
  (black) given by a two-point quadrature domain.}
\end{subfigure}
\hspace{25pt}
\begin{subfigure}[t]{.42\textwidth}
  \includegraphics[width=\linewidth]{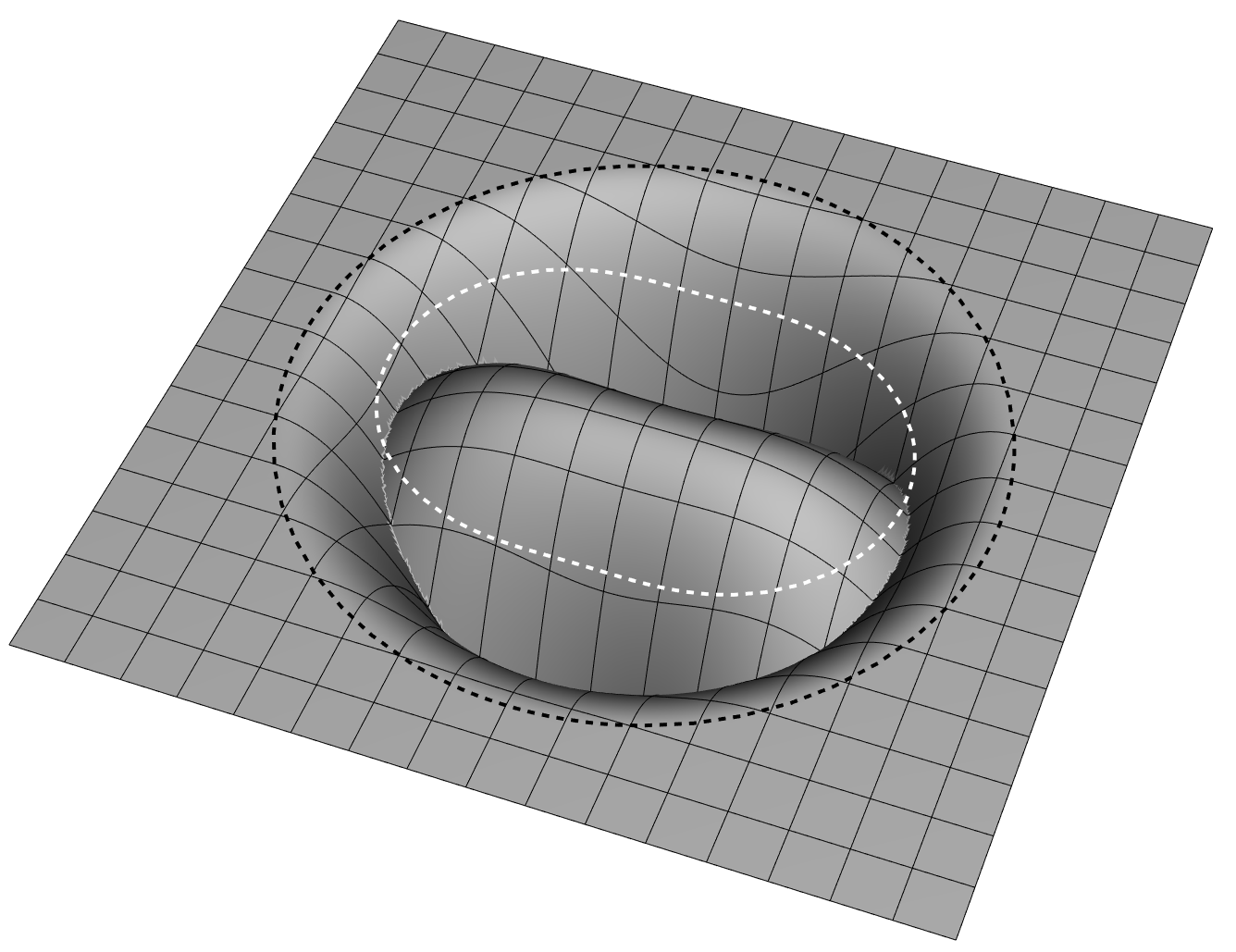}
  \subcaption{A disk-like region (white) formed by merging overlapping disks,
  enclosed by a circular forbidden region (black).}
\end{subfigure}
\caption{} 
\label{fig:potential-oval}
\end{figure}

\appendix

\section{Approximation of measures with Fekete points}
\label{app:Fekete}
We extend the scope of Theorems 1 and 2 by G{\"o}tz and Saff 
from \cite{SaffGotz} so that they will apply in our setting. 
These proofs were obtained together with S. Ghosh.

\subsection{Definitions}
and meets the conditions of Theorem~\ref{thm:main-GEF}, and put
$$
\Lambda = \D(0,\sqrt{\alpha})\setminus\G,
$$
where $\alpha > \alpha_0$ (as defined in Proposition~\ref{prop:Vr}).

We denote by $Q$ a H{\"o}lder continuous weight 
function on $\C$ with H{\"o}lder exponent 
$\gamma^\prime > 0$, and let $\mu_{Q,\Lambda}$ denote the 
unique minimizer of the functional
\[
J_{Q}(\mu) \coloneqq -\Sigma(\mu)+2\int Q(z)\diff\mu(z)
\]
among all probability measures $\mu$ supported on $\Lambda$. 
The measure $\mu_{Q,\Lambda}$ is the  
\emph{equilibrium measure} for the weight $Q$ on the set $\Lambda$.
The potential $U^{\mu_{Q,\Lambda}}$ is known to be
H{\"o}lder continuous with some exponent $\gamma^{\prime\prime} > 0$. 
Without loss of generality we may assume that the H{\"o}lder exponent 
of both $Q$ and $U^{\mu_{Q,\Lambda}}$ is at 
least $ \gamma$ for some constant $ \gamma > 0$.

It is known (e.g. \cite[Theorem I.3.1]{SaffTotik}) that $\mu_{Q,\Lambda}$ is 
uniquely characterized by the condition
\[
\begin{cases}
U^{\mu_{Q,\Lambda}}(z)=Q(z)+C(Q,\Lambda)&\text{on }\;\;
\mathrm{supp}(\mu_{Q,\Lambda})\\
U^{\mu_{Q,\Lambda}}(z)\le Q(z)+C(Q,\Lambda)&\text{on }\;\;\C,
\end{cases}
\]
where $C(Q,\Lambda)$ is a constant.

Given $N$ points $\textbf{z} = (z_1,\ldots,z_N) \in \C^N$ let
\[
\mu_{\textbf{z}} = \frac1N \sum_{j=1}^N \delta_{z_j},
\]
denote their empirical probability measure. 
Assuming these points are distinct, their \emph{discrete logarithmic energy} is given by
\[
-\Sigma^*(\mu_{\textbf{z}}) = \int_{\C^2\setminus\{w_1=w_2\}}
\log \frac1{|w_1-w_2|} \, \diff \mu_{\textbf{z}}(w_1) \diff  \mu_{\textbf{z}}(w_2) 
= - \frac1{N^2}\sum_{1\le i\ne j\le N}\log |z_i-z_j| .
\]
We define a \emph{Fekete configuration} of points $\calF_N$
with respect to the weight (or external field) $Q$ and confined 
to the set $\Lambda$ as a minimizer $\calF_N = (z_1,\ldots,z_n) \subset \Lambda^N$
of the discrete weighted energy functional
\[
J_{Q}^*(\mu_{\textbf{z}}) \coloneqq -\Sigma^*(\mu_{\textbf{z}}) 
+ 2\int Q(w) \diff\mu_{\textbf{z}}(w)
= -\frac1{N^2}\sum_{1\le i\ne j\le N}\log |z_i-z_j| + \frac2{N} \sum_{j=1}^N Q(z_j).
\]
By \cite[Theorem III.1.2]{SaffTotik} it is known that
\[
\calF_N \subset \Lambda^* \coloneqq \{ w \in \Lambda : U^{\mu_{Q,\Lambda}}(w)=Q(w)+C(Q,\Lambda) \}.
\]
Abusing notation, we will variously refer to the set $\calF_N$ by $\textbf{z}$ and $(z_1,\ldots,z_n)$.
\begin{rem}
Our use of weighted Fekete points is somewhat non--standard 
since the external field $Q$ and the number of points $N$ 
(and clearly $\Lambda$) both depend on the parameter $\alpha$. 
This slightly complicates the proofs of the following results.
\end{rem}

\subsection{Separation of Fekete points and approximation of energy and potential}
Using the same strategy of proof as \cite{SaffGotz} we prove the following results.

\begin{prop}[Separation of Fekete points]\label{prop:Ghosh-Gotz-Saff-sep}
Let $\Lambda, Q,  \gamma, \mu_{Q,\Lambda}, \calF_N$ be defined as above. Then,
\[
\inf_{z^\prime \ne z^{\prime\prime}\in\calF_N}|z^\prime 
- z^{\prime\prime}|\ge \frac12 A_1 N^{-1/ \gamma},
\]
where
\[
A_1 = \exp(-\lVert U^{\mu_{Q,\Lambda}}\rVert_{C^{0, \gamma}}).
\]
\end{prop}

\begin{thm}[Approximation of energy and potential]\label{thm:Ghosh-Gotz-Saff}
Let $\Lambda, Q,  \gamma, \mu_{Q,\Lambda}, \calF_N$ be defined as above, and put
\begin{align*}
E_1(N, \gamma,\Lambda) & = \frac1N \big[ \lVert U^{\mu_{Q,\Lambda}}\rVert_{C^{0, \gamma}} 
+ 2 \log 2 \sqrt{\alpha}+ \gamma^{-1}\log N \big],\\
E_2(N, \gamma,\Lambda) & = \frac2N \big[ 4 \lVert U^{\mu_{Q,\Lambda}}\rVert_{C^{0, \gamma}} 
+ \log 2 \sqrt{\alpha}+
(2+3 \gamma^{-1})\log N 
+D(Q,\Lambda) \big],
\end{align*}
where
\[
D(Q,\Lambda)=\sup_{z\in\Lambda}|U^{\mu_{Q,\Lambda}}(z)|.
\]
Then,
\[
-\Sigma^*(\mu_{\textbf{z}}) \le -\Sigma(\mu_{Q,\Lambda}) + E_1(N, \gamma,\Lambda).
\]
In addition, for $N \ge 11$, it holds that for all $w \in \C$
\[
U^{\mu_{\calF_N}}(w)\le U^{\mu_{Q,\Lambda}}(w)+E_2(N, \gamma,\Lambda).
\]
Moreover, if $\tau \ge 1 + 1/\gamma$ is fixed, then, 
whenever $\mathrm{d}(w,\calF_N)\ge N^{-\tau}$ we have
\[
U^{\mu_{\calF_N}}(w)\ge U^{\mu_{Q,\Lambda}}(w) 
- E_2(N, \gamma,\Lambda) - \frac{(\tau - 1 / \gamma) \log N}{N}.
\]
\end{thm}

\begin{proof}[Proof of Proposition~\ref{prop:Ghosh-Gotz-Saff-sep}]
We begin by forming the interpolating polynomials
\[
P_j(z)=\prod_{k\ne j}\frac{z-z_k}{z_j-z_k},\qquad z\in\C,\;\;1\le j\le N.
\]
That $\calF_N=(z_1,\ldots,z_N)$ is energy-minimal is equivalent to
\begin{equation}\label{eq:Fekete-comp}
-\sum_{j\ne k}\log |z_j-z_k|+2 N \sum_{j=1}^N Q(z_j)\le 
-\sum_{j\ne k}\log |w_j-w_k|+2 N \sum_{j=1}^N Q(w_j),
\end{equation}
whenever $(w_1, \dots, w_N) \in \Lambda^N$. In 
particular, for any $z\in\Lambda$
we may apply \eqref{eq:Fekete-comp} with 
\[
\begin{cases}
w_k=z_k,&k\ne j\\
w_j=w
\end{cases}
\]
for any fixed $j$ to obtain
\begin{equation}\label{eq:bound-Log-Pj}
\sum_{k:k\ne j}\log|w-z_k|-\sum_{k:k\ne j}
\log |z_j-z_k|\le N(Q(w)-Q(z_j)),\qquad 1\le j\le N.
\end{equation}

Let $w \in \mathrm{supp}(\mu_{Q,\Lambda})$. 
Since $Q(w)=U^{\mu_{Q,\Lambda}}(w) - C(Q,\Lambda)$, 
we find by \eqref{eq:bound-Log-Pj} that
\[
\sum_{k:k\ne j}\log|w - z_k|-\sum_{k:k\ne j}
\log |z_j-z_k|\le N(U^{\mu_{Q,\Lambda}}(w)-
U^{\mu_{Q,\Lambda}}(z_j)),\qquad 1\le j\le N.
\]
If we denote by $\nu$ the sub-probability 
measure $\nu=\tfrac{1}{N}\sum_{k \ne j}\delta_{z_j}$
we may divide the previous equation by $N$ and rewrite it as
\[
U^\nu(w)-U^\nu(z_j)\le U^{\mu_{Q,\Lambda}}(w)
-U^{\mu_{Q,\Lambda}}(z_j),\qquad w \in \mathrm{supp}(\mu_{Q,\Lambda}).
\]
The principle of domination \cite[Theorem II.3.2]{SaffTotik}
implies that the above inequality holds for all $w \in \C$, 
and consequently
we obtain
\[
|P_j(w)|\le \exp\big(N(U^{\mu_{Q,\Lambda}}(w)
-U^{\mu_{Q,\Lambda}}(z_j))\big).
\]
This says that $|P_j(z)|\le A_1^{-1} 
= \exp(\lVert U^{\mu_{Q,\Lambda}}\rVert_{C^{0, \gamma}})$ whenever 
$|z-z_j|\le N^{-1/ \gamma}$.

Now, if $|z-z_j|\le \tfrac{1}{2} N^{-1/ \gamma}$ 
we have by the standard Cauchy estimate that
\[
|P_j'(z)|\le \int_{\T\big(z,\frac{1}{2}N^{-1/ \gamma}\big)}
\frac{|P_j(\zeta)|}{|\zeta-z|^2}\,\diff\sigma(\zeta)
\le 2 A_1^{-1} N^{1/ \gamma}.
\]
If for two distinct points $z_j$ and $z_k$ in $\calF_N$ 
we were to have $|z_j-z_k|\le \tfrac{1}{2} N^{-1/  \gamma}$,
then 
\[
1=|P_j(z_j)-P_j(z_k)|\le \int_{[z_j,z_k]}|P_j'(z)|\,
\diff|z|\le 2 A_1^{-1} N^{1/ \gamma} |z_j-z_k|.
\]
But this shows that $|z_j-z_k|\ge \tfrac12 A_1  N^{-1/\gamma}$, 
which completes the first step.
\end{proof}

\begin{proof}[Proof of Theorem~\ref{thm:Ghosh-Gotz-Saff}]
The proof is split into three steps. In the first step we 
obtain the upper bound for the discrete logarithmic energy.
In the second step, we obtain a lower bound for the 
difference of the potentials. We conclude in the third step
with a corresponding upper bound.
\\
\\
\noindent {\sc Step 1.}\; Dividing \eqref{eq:bound-Log-Pj} 
by $N^2$ and summing over $1\le j\le N$, we find that for $w \in \Lambda$
\[
-\Sigma^*(\mu_{\textbf{z}})+\int Q\,\diff \mu_{\textbf{z}} 
\le
- (1 - \tfrac1N) U^{\mu_{\textbf{z}}}(w) + Q(w)
\le 
-U^{\mu_{\textbf{z}}}(w) + Q(w)+\frac{\log 2 \sqrt{\alpha}}{N},
\]
which gives after rearranging terms
\begin{equation}\label{eq:fekete-bound-sum}
-U^{\mu_{\textbf{z}}}(w) + U^{\mu_{Q,\Lambda}}(w) \ge
-\Sigma^*(\mu_{\textbf{z}})+\int U^{\mu_{Q,\Lambda}}\,
\diff\mu_{\textbf{z}}-\frac{\log 2\sqrt{\alpha}}{N},\quad w \in \Lambda.
\end{equation}
We would like to replace the discrete energy with a continuous one.
We introduce the regularized measure 
$\mu_{\textbf{z}}^r=\mu_{\textbf{z}}*\sigma_r$, where $\sigma$ is the normalized 
arc length measure on $\T(0,r)$, and $r > 0$ is a small 
parameter which remains to be chosen. A simple estimate shows that
\begin{multline}\label{eq:fekete-r-est}
\Big|\int U^{\mu_{Q,\Lambda}}\diff(\mu_{\textbf{z}}^r-\mu_{\textbf{z}})\Big|
\le\frac{1}{N}\sum_{j=1}^N\int_{\T(z_j,r)}
|U^{\mu_{Q,\Lambda}}(t)-U^{\mu_{Q,\Lambda}}(z_j)|\diff\sigma_r (t)
\\
\le \lVert U^{\mu_{Q,\Lambda}}\rVert_{C^{0, \gamma}} r^{ \gamma} \eqqcolon  C_U^\gamma r^{ \gamma}.
\end{multline}
We next observe that by subharmonicity of the logarithm, it holds that
\[
-\Sigma^*(\mu_{\textbf{z}})\ge -\Sigma(\mu_{\textbf{z}}^r)-\frac{|\log r|}{N}.
\]
For a signed measure $\nu$ of vanishing total mass, 
we have (e.g. \cite[Lemma I.1.8]{SaffTotik})
\[
  -\Sigma(\nu)=\int |\nabla U^{\nu}|^2\diffA(z),
\]
where the right-hand side is evidently positive. 
Applying this to $\nu=\mu_{\textbf{z}}^r-\mu_{Q,\Lambda}$
we find that
\[
-\Sigma(\mu_{\textbf{z}}^r)\ge -2\int U^{\mu_{Q,\Lambda}}
\,\diff\mu_{\textbf{z}}^r + \Sigma(\mu_{Q,\Lambda}).
\]
The first term on the right may be bounded, using \eqref{eq:fekete-r-est}:
\[
-\int U^{\mu_{Q,\Lambda}}\,\diff\mu_{\textbf{z}}^r\ge 
-\int U^{\mu_{Q,\Lambda}}\,\diff\mu_{\textbf{z}}-
 C_U^\gamma r^{ \gamma}.
\]
Adding the above bounds yields
\begin{equation}\label{eq:fekete-sigma-disc-cont-lower}
-\Sigma^*(\mu_{\textbf{z}})+\int U^{\mu_{Q,\Lambda}}\,\diff\mu_{\textbf{z}}^r\ge
-\int U^{\mu_{Q,\Lambda}}\,\diff\mu_{\textbf{z}}
+\Sigma(\mu_{Q,\Lambda})-\frac{|\log r|}{N}- C_U^\gamma r^{ \gamma}.
\end{equation}
Returning to \eqref{eq:fekete-bound-sum}, 
invoking \eqref{eq:fekete-sigma-disc-cont-lower} 
and choosing $r=N^{-1/ \gamma}$, it follows that
\begin{equation}\label{eq:fekete-pre-lower-epsilon}
U^{\mu_{Q,\Lambda}}(w)-U^{\mu_{\textbf{z}}}(w)\ge 
\epsilon_N- \gamma^{-1}\frac{\log N}{N}
-\frac{ C_U^\gamma}{N}-\frac{\log 2 \sqrt{\alpha}}{N},
\end{equation}
where
\[
\epsilon_N=\Sigma(\mu_{Q,\Lambda}
)-\int U^{\mu_{Q,\Lambda}}\,\diff\mu_{\textbf{z}}=
\int \big(U^{\mu_{Q,\Lambda}}
-U^{\mu_{\textbf{z}}}\big)\diff\mu_{Q,\Lambda}.
\]
The inequality \eqref{eq:fekete-pre-lower-epsilon} 
initially holds only on $\Lambda$, but extends to an 
inequality on $\C$ by the principle
of domination, since $\Sigma(\mu_{Q,\Lambda})$ is finite.

Now let $\nu_1$ be the (unweighted) equilibrium measure 
of the set $\Lambda$, and note that $U^{\nu_1}$ is equal to 
some constant value quasi-everywhere in $\Lambda$, and that 
it exceeds this value everywhere in $\C$. If we integrate 
$U^{\mu_{Q,\Lambda}}(w)-U^{\mu_{\textbf{z}}}(w)$ against 
$\nu_1$, then by \eqref{eq:fekete-pre-lower-epsilon} we have
\begin{equation}\label{eq:lower-upp-epsilon-N}
\epsilon_N \le 
 \gamma^{-1}\frac{\log N}{N}+\frac{ C_U^\gamma}{N}
 +\frac{\log 2 \sqrt{\alpha}}{N}.
\end{equation}
Similarly, if we integrate 
\eqref{eq:fekete-bound-sum} against $\nu_1$, then we get
\[
-\Sigma^*(\mu_{\textbf{z}}) \le -\int U^{\mu_{Q,\Lambda}}\,
\diff\mu_{\textbf{z}} + \frac{\log 2\sqrt{\alpha}}{N},
\]
which together with the definition of $\epsilon_N$ and 
\eqref{eq:lower-upp-epsilon-N} gives the 
required upper bound
for the energy $-\Sigma^*(\mu_{\textbf{z}})$.\\
\\
\noindent {\sc Step 2.}\; In this step, we obtain a lower bound
for $\epsilon_N$. We start with
a preliminary bound. Define the functions
\[
h_j(w)=-\frac{1}{N}\sum_{k:k\ne j}\log|w-z_k|,\qquad 1\le j\le N.
\]
Also recall that
\[
D(Q,\Lambda)=\sup_{z\in\Lambda}|Q(z)+C(Q,\Lambda)|
=\sup_{z\in\Lambda}|U^{\mu_{Q,\Lambda}}(z)|.
\]
For the functions $h_j$, we show below that
\begin{equation}\label{eq:Lemma4}
|h_j(z_j)+U^{\mu_{Q,\Lambda}}(z_j)-\epsilon_N|\le e(N,Q,\Lambda)
\end{equation}
where 
\[
e(N,Q,\Lambda)=2  \gamma^{-1} \frac{\log N}{N}
+\frac{\log 2 \sqrt{\alpha}}{N} + \frac{2  C_U^\gamma}{N} +
\frac{D(Q,\Lambda)}{N}.
\]

For $w\in\Lambda$ and $z_j\in\calF_N$, 
arguing in the same way as in \eqref{eq:bound-Log-Pj},
we have by extremality of $\calF_N$ that
\[
h_j(z_j)+U^{\mu_{Q,\Lambda}}(z_j)\le h_j(w)+U^{\mu_{Q,\Lambda}}(w),
\]
or equivalently that
\[
h_j(z_j)\le -U^{\mu_{Q,\Lambda}}(z_j)+h_j(w)+U^{\mu_{Q,\Lambda}}(w).
\]
Integrating this against $\mu_{Q,\Lambda}$ we find that
\begin{multline*}
h_j(z_j)\le -U^{\mu_{Q,\Lambda}}(z_j)-\frac{1}{N}
\sum_{k:k\ne j}U^{\mu_{Q,\Lambda}}(z_k)+\Sigma(\mu_{Q,\Lambda})
\\ 
\le -U^{\mu_{Q,\Lambda}}(z_j)+\epsilon_N+\frac{U^{\mu_{Q,\Lambda}}(z_j)}{N}\le 
-U^{\mu_{Q,\Lambda}}(z_j)+\epsilon_N+\frac{D(Q,\Lambda)}{N}.
\end{multline*}
This completes the upper bound of \eqref{eq:Lemma4}.

To obtain the lower bound, let $w$ satisfy $|w-z_j|=N^{-1/ \gamma}$,
and note that by \eqref{eq:fekete-pre-lower-epsilon} it holds that
\[
h_j(w)=-U^{\mu_{\textbf{z}}}(w)+\frac{\log|w-z_j|}{N}
\ge -U^{\mu_{Q,\Lambda}}(z) + \epsilon_N 
-2 \gamma^{-1} \frac{\log N}{N}
-\frac{ C_U^\gamma}{N}-\frac{\log 2 \sqrt{\alpha}}{N}.
\]
This together
with the H{\"o}lder continuity of $U^{\mu_{Q,\Lambda}}$ yields
(recalling the definition of $ C_U^\gamma$)
\[
h_j(w) \ge -U^{\mu_{Q,\Lambda}}(z_j) -\frac{ C_U^\gamma}{N} + \epsilon_N 
-2 \gamma^{-1}\frac{\log N}{N} -\frac{ C_U^\gamma}{N}
-\frac{\log 2 \sqrt{\alpha}}{N}.
\]
Since $h$ is superharmonic, the same bound for $h_j$ holds at the point $w=z_j$,
which completes the proof of \eqref{eq:Lemma4}.

We return to bound the quantity $\epsilon_N$ from below, which was the 
main purpose of the current step. The claim is that
\begin{equation}\label{eq:claim1-fekete}
\epsilon_N \ge -\frac{f(N,Q,\Lambda)}{N},
\end{equation}
where
\[
f(N,Q,\Lambda)
=4  C_U^\gamma + \log 2 \sqrt{\alpha}+
(2+3 \gamma^{-1})\log N 
+D(Q,\Lambda).
\]
To see why this holds, put $\delta_1 = \tfrac14 A_1 N^{-1/ \gamma}$, 
where $A_1 = \exp(- C_U^\gamma) \le 1$ is the constant from 
Proposition~\ref{thm:Ghosh-Gotz-Saff}, and using this proposition observe that
\[
  \sup_{|w-z_j|\le \delta_1}|\nabla h_j(w)|\le 
\sup_{|w-z_j|\le \delta_1}\frac{1}{N}\sum_{k:k\ne j}
\frac{1}{|w-z_k|}\le (\delta_1)^{-1}.
\] 
Put $\eta = N^{-\beta} \delta_1$, with $\beta \ge 1$. 
By the H{\"o}lder continuity of $U^{\mu_{Q,\Lambda}}$,
and the gradient bound for $h_j$, 
it follows that for $w$ such that 
$|w-z_j|=\eta$ we have
\[
|U^{\mu_{Q,\Lambda}}(w)-U^{\mu_{\textbf{z}}}(w)-\epsilon_N|
\le |U^{\mu_{Q,\Lambda}}(z_j)+h_j(z_j)-\epsilon_N| + 
\frac{ C_U^\gamma}{N} + \eta (\delta_1)^{-1} + \frac{|\log \eta|}{N}.
\]
Put $\calD(\eta) = \calD(\eta; \beta)=\cup_j\D(z_j,\eta)$.
Invoking the bound \eqref{eq:Lemma4}, we find 
that for $w\in \partial \calD(\eta)$ it holds that
\begin{multline}\label{eq:pre-11}
|U^{\mu_{Q,\Lambda}}(w)-U^{\mu_{\textbf{z}}}(w)-\epsilon_N|\le
\frac{ C_U^\gamma}{N} + \eta (\delta_1)^{-1} 
+ \frac{|\log \eta|}{N}+e(N,Q,\Lambda) \\
\le \frac{2 \beta \log N}N + \frac1N \big[ 3 \gamma^{-1} \log N 
+ 4  C_U^\gamma + D(Q,\Lambda) + \log 2 \sqrt{\alpha} \big]
\eqqcolon \frac{2 \beta \log N}N + e^\prime(N,Q,\Lambda),
\end{multline}
assuming e.g.\ that $N \ge 11$.

Let $\nu_2$ be the (unweighted) equilibrium 
measure of the set $\calD(\eta)$. Note that $U^{\nu_2}$ has a 
constant value (everywhere) inside $\calD(\eta)$, 
and that it exceeds this value outside it. 
Integrating $U^{\mu_{Q,\Lambda}}(w)-U^{\mu_{\textbf{z}}}(w)$ 
against $\nu_2$, we see that $\epsilon_N$ is bounded from below by
\[
\int \big[ U^{\mu_{Q,\Lambda}}(w)-U^{\mu_{\textbf{z}}}(w)\big] 
\, \diff \nu_2(w) - \frac{2 \beta \log N}N - e^\prime(N,Q,\Lambda) 
\ge - \frac{2 \beta \log N}N - e^\prime(N,Q,\Lambda).
\]
Taking $\beta = 1$ we obtain 
\eqref{eq:claim1-fekete}, and this completes Step 2.\\
\\
Now that the bound \eqref{eq:claim1-fekete} has 
been established, the lower bound for the difference
$U^{\mu_{Q,\Lambda}}(w)-U^{\mu_{\textbf{z}}}(w)$
follows from \eqref{eq:fekete-pre-lower-epsilon}, and reads
\begin{equation}\label{eq:Fekete-lower bound}
U^{\mu_{Q,\Lambda}}(w)-U^{\mu_{\textbf{z}}}(w)\ge 
\epsilon_N- \gamma^{-1}\frac{\log N}{N}-\frac{ C_U^\gamma}{N}
-\frac{\log 2 \sqrt{\alpha}}{N}\ge - \frac2N f(N,Q,\Lambda).
\end{equation}
\\
\noindent {\sc Step 3.}\; Now we take 
$\beta = \tau - 1 / \gamma \ge 1$. We complete the 
proof of Theorem~\ref{thm:Ghosh-Gotz-Saff}
by obtaining an upper bound for 
$U^{\mu_{Q,\Lambda}}(w)-U^{\mu_{\textbf{z}}}(w)$.
This is readily derived from \eqref{eq:pre-11}. 
This initially holds only on $\partial \calD(\eta)$. However,
on $\calD(\eta)^c$, the function 
$U^{\mu_{Q,\Lambda}}-U^{\mu_{\textbf{z}}}$ 
is subharmonic and vanishes at infinity.
Applying the maximum principle, and using 
\eqref{eq:lower-upp-epsilon-N}, we find that for all $w\in\C$
\[
U^{\mu_{Q,\Lambda}}(w)-U^{\mu_{\textbf{z}}}(w)\le \epsilon_N +
\frac{2 \beta \log N}N + e^\prime(N,Q,\Lambda) \le 
\frac{2 \beta \log N}N + \frac2N f(N,Q,\Lambda),
\]
which completes the proof.
\end{proof}

\bigskip

\noindent {\bf Acknowledgements.}\; We would like to express our 
sincere gratitude to Subhro Ghosh for valuable discussions in the early 
stages of this project and for allowing us to reproduce the joint
work on Fekete points in the appendix,
to Fedor Nazarov for suggesting a version of the variational
principle underlying the proof of Proposition~\ref{prop:a-circ}, 
and to John Andersson for suggestions on how to approach matters of
stability in thin obstacle problems.
We have enjoyed stimulating discussions with Kari Astala, Yan Fyodorov, H{\aa}kan Hedenmalm,
David Jerison, Avner Kiro, Nikolai Makarov, 
Sylvia Serfaty, Henrik Shahgholian, Mikhail Sodin, Pierpaolo Vivo, Oren Yakir and Ofer Zeitouni,
for which we wish to thank them. We thank Mikhail Sodin and Oren Yakir for helpful suggestions that led to an improved presentation of the results.

\medskip

The research of the first-named author was funded by ISF Grant 1903/18 and ERC Advanced
Grant 692616.
The second author was funded by KAW Foundation Grant 2017.0389,
with travel support from ERC Advanced Grant 692616.

\end{document}